\documentclass[12pt]{amsart}

\usepackage{amsfonts,amssymb,amsmath,textcomp,amsthm}
\usepackage{xcolor}
\usepackage{hyperref}
\usepackage{float}

\usepackage{graphicx}
\usepackage[table,xcdraw]{xcolor} % cores (inclui suporte a coloração de tabelas)
\usepackage{scalefnt}             % ajustes finos de escala de fonte
\usepackage{mathrsfs}             % \mathscr

\usepackage{tikz}                 % gráficos vetoriais
\usetikzlibrary{arrows}           % setas clássicas (como no IJAC)

\usepackage{pgfplots}             % gráficos de funções/dados
\pgfplotsset{compat=1.15}         % mesma versão usada no IJAC

\usepackage{amstext} % for \text macro
\usepackage{array}   % for \newcolumntype macro
\newcolumntype{C}{>{$}c<{$}} % math-mode version of "c" column type

\usepackage{placeins}      % para controlar floats
\usepackage{caption}       % para ajustar espaço da legenda
\usepackage{amsmath,amssymb}
\usepackage{enumitem}  % para personalizar listas numeradas

\usepackage{ltablex} % longtable + tabularx
\keepXColumns        % mantém a largura relativa das colunas X
\usepackage{array} 
\renewcommand{\arraystretch}{1.02} % altura das linhas

\newtheorem{theorem}{Theorem}[section]
\theoremstyle{plain}

\newtheorem{corollary}[theorem]{Corollary}

\numberwithin{equation}{section}

\theoremstyle{definition}
\newtheorem{lemma}{Lemma}
\theoremstyle{definition}
\newtheorem{definition}[theorem]{Definition}

\newtheorem{remark}[theorem]{Remark}

\begin{document}
	\title{Representations of Code Loops by Binary Codes}

	\author[Miguel Pires]{Rosemary Miguel Pires}
	\author[Grishkov]{Alexandre Grishkov}
	\author[Rasskazova]{Marina Rasskazova}
	\address[Miguel Pires]{Instituto de Ci\^encias Exatas, Universidade Federal Fluminense, Rua Desembargador Ellis Hermydio Figueira, Aterrado, Volta Redonda, RJ, Brazil, CEP 27213-145 } 
	\email[Pires]{rosemarypires@id.uff.br}
	\address[Grishkov]{Instituto de Matem\'atica e Estat\'istica, Universidade de S\~ao Paulo, Rua do Mat\~ao, 1010 Butant\~a, S\~ao Paulo, SP, Brazil, CEP 05508-090  and Sobolev Institute of Mathematics, Omsk, Russia}
	\email[Grishkov]{shuragri@gmail.com}
	
	\address[Rasskazova]{Centro de Matem\'atica, Computa\c c\~ao  e Cogni\c c\~ao, Universidade Federal do ABC, Avenida dos Estados, 5001, Bairro Bangu, Santo Andr\'e, SP, Brasil,
		CEP 09280-560; and Siberian State Automobile and Highway University, Prospekt Mira, 5, Omsk, Omsk Oblast, Rússia, 644080}
	\email[Rasskazova]{marina.rasskazova@ufabc.edu.br}
	
	\begin{abstract}
		Code loops are Moufang loops constructed from doubly even binary codes. Then, given a code loop $L$, we ask which doubly even binary code $V$ produces $L$. In this sense, $V$ is called a representation of $L$. In this article we define and show how to determine all minimal and reduced representations of nonassociative code loops of rank $3$ and $4$.
	\end{abstract}


	\maketitle
	\section{Introduction}
	
	Code loops were introduced by Robert Griess \cite{2} from doubly even binary codes as follows.

	Let $K=\mathbb{F}_{2}=\{0,1\}$ be a field with two elements and $K^n$ be an $n$-dimensional vector space over $K$. Let $u=(u_1,\dots,u_n)$ and $v=(v_1,\dots,v_n)$ in $K^n$ and define $|v|=|\{i\; |\;v_i=1\}|$ (the Hamming weight)  and  $|u \cap v|=|\{i\; |\;u_i=v_i=1\}|.$ A \textit{doubly even binary code} is a subspace $V\subseteq{K^{n}}$ such that $|v|\equiv{0\;}(\mbox{mod}{\;4})$.
	%% and \linebreak $|u \cap v|\equiv{0\;}(\mbox{mod}{\;2})$ for all $u,v \in V$.
	%
	
	Let $V$ be a doubly even binary code and the function $\phi:V\times V\to \{1,-1\}$, called \textit{factor set},  defined by: 
	
	\begin{align}
		\phi(v,v)&=(-1)^{\frac{|v|}{4}}, \nonumber \\
		\phi(v,w)&=(-1)^{\frac{|v{\cap}w|}{2}}\phi(w,v), \nonumber \\
		\phi(0,v)&=\phi(v,0)=1,   \nonumber \\
		\phi(v+w,u)&=\phi(v,w+u)\phi(v,w)\phi(w,u)(-1)^{|v{\cap}w{\cap}u|}. \nonumber
	\end{align}

	In order to define a code loop, let $V$ be a doubly even binary code and \linebreak $\phi:V\times V\to \{1,-1\}$ be a factor set. Consider the set $L(V)= \{1,-1\} \times V $ and define a product `$\cdot $' in $L(V)$ in the following way
	\begin{align}
		&v.w = \phi(v,w)(v + w), \nonumber\\
		&v.(-w) = (-v).w = -(v.w). \nonumber
		%&(-v).(-w) = v.w. \nonumber
	\end{align} 
	
	% With this product, we obtain a Moufang loop called \textit{code loop}. We say that $L(V)$ has \textit{rank} $m$, if the dimension of the $K$-vector space $V$ is equal to $m$. Note that $v \in L(V)$ and $-v \in L(V)$ means the elements $(1,v)$ and $(-1,v)$, respectively.

	With this product, we obtain a Moufang loop called \textit{code loop}. 
	Equivalently, a Moufang loop $L$ such that $L/A$ is an elementary abelian $2$-group for some central subgroup $A$ of order $2$ is called a code loop. 
	We say that $L(V)$ has \textit{rank} $m$ if the dimension of the $K$-vector space $V$ is equal to $m$. 
	Note that $v \in L(V)$ and $-v \in L(V)$ mean the elements $(1,v)$ and $(-1,v)$, respectively.
	
	It's important to recall that code loops have some useful properties related to the commutator, associator, and square associated with their elements that we will use to get the main results presented in this article. Chein and Goodaire \cite{4} proved that  code loops have a unique nonidentity square, a unique nonidentity commutator, and a unique nonidentity associator. In other words, for any $u,v,w \in V$: 
	\begin{align}
		v^2 &= (-1)^{\frac{|v|}{4}}0, \label{eq1}\\ %\nonumber
		\left[u,v\right] &= u^{-1}v^{-1}uv = (-1)^{\frac{|u \cap v|}{2}}0,\label{eq2} \\
		(u,v,w)&=((uv)w)( (u(vw))^{-1})=(-1)^{|u\cap v\cap w|}0. \label{eq3}
	\end{align}
	
	Let $V$ be a doubly even binary code and $v \in V$. Note that, if $v^{2}=-1$, then $|v|\equiv{4\;}(\mbox{mod}{\;8})$. Otherwise, we have $|v|\equiv{0\;}(\mbox{mod}{\;8}).$
	
	Now let $u,v \in V$. As a consequence of the definition of doubly even binary code, we obtain $|u\cap v|\equiv{0\;}(\mbox{mod}{\;2})$. Then in the case $[u,v]=-1$, we obtain $|u\cap v|\equiv{2\;}(\mbox{mod}{\;4})$ and in the other case, $|u\cap v|\equiv{0\;}(\mbox{mod}{\;4}).$

	So far, we've shown how to construct code loops from doubly even binary codes, but we would like to investigate the other direction. In other words, given a certain code loop $L$, we would like to determine the doubly even codes $V$ such that $L\simeq L(V)$. 
	
	A {\textit{representation}} of a given code loop $L$ is a doubly even binary code $V \subseteq {K^{m}}$ such that $L \simeq L(V)$ and the number $m$ is called the \textit{degree of the representation}. So the degree of the representation is the length of the code. We notice that there are many different representations for the same code loop. When necessary, we denote the degree of $V$ by $\mbox{deg}(V)$.

	In this article we determine all minimal representations of nonassociative code loops of rank $3$ and $4$ (Sections 3 and 4). We will also show how to determine reduced representations. In order to do this, first, it is necessary to know all these code loops, up to isomorphism. In \cite{AR}, these code loops were classified using the concept of  characteristic vector associated with them. We recall the main theorems of classification of nonassociative code loops of rank $3$ and $4$ in Section 2.
	
	\section{Nonassociative Code Loops of Rank 3 and 4}\label{section2}
	
	Let $L$ be a code loop with generator set $X=\left\{x_{1},\dots,x_{n}\right\}$ and center $\left\{1,-1\right\}$. Then we define the characteristic vector of $L$, denoted by $\lambda_{X}(L)$ or $\lambda(L)$, by \[\lambda(L)=(\lambda_{1},\dots,\lambda_{n};\lambda_{12},\dots,\lambda_{1n},\dots,\lambda_{(n-1)n};\lambda_{123},\dots,\lambda_{12n},\dots,\lambda_{(n-2)(n-1)n}),\] where $\lambda_{i}, \lambda_{ij},\lambda_{ijk} \in \mathbb{F}_{2}$, $(-1)^{\lambda_{i}}=x_{i}^{2},\; (-1)^{\lambda_{ij}}=[x_{i},x_{j}]\; \mbox{and}\;  (-1)^{\lambda_{ijk}}=(x_{i},x_{j},x_{k})$.
	
	%(e.g. $\{x\in L|[x,a]=(x,a,b)=1, \forall a,b\in L\}.$).
	%Here, $[x,y]$ denotes the commutator of $x$ and $y$, and $(x,y,z)$ denotes the \linebreak associator of $x,y$ and $z$.
	
	Here, $[x,y]$ denotes the commutator of $x$ and $y$, and $(x,y,z)$ denotes the 
	associator of $x,y$ and $z$. The center of a loop $L$, denoted by $\mathcal{Z}(L)$, is the set 
	\[
	\mathcal{Z}(L) = \{x \in L \mid [x,y] = (x,y,z) = 1, \ \forall y,z \in L\}.
	\]
	%%%%%
	
	Once a generating set $X$ is fixed, we obtain a unique characteristic vector $\lambda_X(L)$, which depends on the choice of the elements in $X$ and on the order of its elements. Different minimal generating sets may produce different characteristic vectors for the same code loop, as illustrated by the following example: 
	
	Let $L = \mathbb{Z}_2 \times \mathbb{Z}_4 = \{0,1\} \times \{0,1,2,3\}$ and let $Z = \{(0,0), (0,2)\}$.
	Then $Z \leq \mathcal{Z}(L)$ (since $L$ is an abelian group),  and $L/Z$ is the Klein group, so $L$ is a code loop of rank $2$.
	We can consider various (ordered) minimal generating subsets of $L$.
	With $X = (a,b) = ((0,1), (1,0))$ we obtain $a+a = (0,2)$ and $b+b = (0,0)$,
	while with $X' = (a', b') = ((0,1), (1,1))$ we obtain $a'+a' = (0,2)$ and $b'+b' = (0,2)$. Note that $b$ has order $2$, while $b'$ has order $4$, so the characteristic vectors associated with each generating set are different.
	
	The fact that the characteristic vector uniquely determines a code loop was established in \cite{AR}. By definition of a characteristic vector, the number of possible such vectors is $2^m$, where $m = n + \frac{n(n-1)}{2} + \frac{n(n-1)(n-2)}{6}$. In \cite{AR}, the authors constructed  a free loop $\mathcal{F}_n$ in the variety generated by code loops and showed how to establish a correspondence between a vector with coordinates in $\mathbb{F}_2$ and a nonassociative code loop.  However, different vectors may determine the same nonassociative code loop up to isomorphism. Moreover, using the natural action of $GL_n(2)$ on the set of characteristic vectors, it was proved that the orbits of this action determine the isomorphism classes of nonassociative code loops. Therefore, there is a bijective correspondence between the set of  isomorphism classes of nonassociative code loops of rank $n$ and the set of $GL_n(2)$-orbits of characteristic vectors of such loops (see \cite[Corollary~2.11, p.~172]{AR} and \cite[Proposition~2.12, p.~172]{AR}).
	
	%%%%

	We can determine a representation $V$ of $L$ just by knowing the characteristic vector $\lambda(L)$ associated with $L$. Indeed, the method consists of finding a basis of $V$ with vectors $v_{1},\dots,v_{n}$ corresponding to $x_{1},\dots,x_{n}$, respectively. The squares, commutators and associators of the generators of $L$ will determine the respective squares, commutators and associators of the basis elements of $V$ and, from Equations (\ref{eq1}), (\ref{eq2}) and (\ref{eq3}), we get the weights of $v_{i}$, $v_{i}\cap v_{j}$ and $v_{i} \cap v_{j} \cap v_{k}$,  that is, $|v_{i}|$, $|v_{i}\cap v_{j}|$ and $|v_{i} \cap v_{j} \cap v_{k}|$, for each $i,j,k \in \left\lbrace 1,\dots,n\right\rbrace $. To determine the possible vectors of a basis of $V$, it is enough to analyze these weights, as we will see in the next section. We usually use the following notation: $t_i=|v_i|$, $t_{ij}=|v_i \cap v_j|$ and $t_{ijk}=|v_{i}\cap v_{j}\cap v_{k}|$.
	
	%%%
	
	%%%
	
	Now, we assume that $L$ is a nonassociative code loop of rank $3$ with generators $a,b,c$. Then the associator $(a,b,c)$ is equal to $-1$, otherwise, by Moufang Theorem, $L$ would be a group. The characteristic  vector associated to $L$ is given by $\lambda(L) = (\lambda_1,...,\lambda_6)$ (or simply $\lambda(L) = (\lambda_1...\lambda_6)$) where $a^{2}=(-1)^{\lambda_1}$, $b^{2}=(-1)^{\lambda_2}$, $c^{2}=(-1)^{\lambda_3}$, $\left[a,b\right]=(-1)^{\lambda_4}$, $\left[a,c\right]=(-1)^{\lambda_5}$ and $\left[b,c\right]=(-1)^{\lambda_6}.$ The omitted seventh coordinate is equal to one. 
	
	%%%
	As discussed above, we know that two code loops with the same characteristic vector are isomorphic. This follows directly from \cite[Corollary 2.11, p.172]{AR} and \cite[Proposition 2.12, p.172]{AR}. To illustrate and justify this in a concrete case, consider the nonassociative code loops of rank~$3$. 
	By definition of the characteristic vector, and since we are considering nonassociative loops of this rank, there are exactly $64$ such vectors corresponding to these loops. According to \cite[Proposition 3.2, p.175]{AR}, these vectors are divided into $5$ equivalence classes (orbits) under the natural action of $GL_3(2)$. Each orbit corresponds to a unique nonassociative code loop given by Theorem~\ref{th2.1} below, where each listed characteristic vector is a representative of the corresponding orbit.
	
	%%%

	%%%%
	
	% Next, we present the Classification Theorem of code loops of rank $3$.

	\begin{theorem}[\cite{AR}]\label{th2.1}
		Consider $C_{1}^{3},...,C_{5}^{3}$ the code loops with the following characteristic vectors:
		\begin{eqnarray*}
			\begin{array}{lllll}
				\lambda(C_{1}^{3})=(111111), &\,&\lambda(C_{2}^{3})=(000000),&\,&\lambda(C_{3}^{3})=(000111),\\
				&&&&\\
				\lambda(C_{4}^{3})=(110000),&\,&\lambda(C_{5}^{3})=(100000).&\,&\\
			\end{array}
		\end{eqnarray*}
		
		Then any two loops from the list $\left\{C_{1}^{3},...,C_{5}^{3}\right\}$  are not isomorphic and every nonassociative
		code loop of rank $3$ is isomorphic to one of this list.
	\end{theorem}

	From the above classification, we can now describe the congruence classes of 
	$t_i$, $t_{ij}$, and $t_{ijk}$ for each loop, which encode the intersection 
	properties of the basis vectors in a representation. 
	The values listed in the next remark correspond to the case of code loops of rank~$3$.

	\begin{remark}\label{rem2.2}
		According to Theorem \ref{th2.1} and Equations (\ref{eq1}), (\ref{eq2}) and (\ref{eq3}), any representation of the loops  $C_{1}^{3},...,C_{5}^{3}$ has a basis $\left\lbrace v_{1},v_{2},v_{3}\right\rbrace $ that satisfies the following properties, respectively:
		
		\begin{enumerate}
			\item[1.] $t_1\equiv t_2\equiv t_3\equiv4\;(\mbox{mod} \; 8)$, $t_{12}\equiv t_{13}\equiv t_{23}\equiv2\;(\mbox{mod}\;4)$;
			\item[2.] $t_{1}\equiv t_{2}\equiv t_{3}\equiv0\;(\mbox{mod} \; 8)$, $t_{12}\equiv t_{13}\equiv t_{23}\equiv0\;(\mbox{mod}\;4)$;
			\item[3.] $t_{1}\equiv t_{2}\equiv t_{3}\equiv0\;(\mbox{mod} \; 8)$, $t_{12}\equiv t_{13}\equiv t_{13}\equiv2\;(\mbox{mod}\;\;4)$;
			\item[4.] $t_{1}\equiv t_{2}\equiv4\;(\mbox{mod}\; 8)$, $t_{3}\equiv0\;(\mbox{mod}\;8)$, $t_{12}\equiv t_{13}\equiv t_{23}\equiv0\;(\mbox{mod}\;4)$;
			\item[5.] $t_{1}\equiv4\;(\mbox{mod}\; 8)$, $t_{2}\equiv t_{3}\equiv0\;(\mbox{mod}\;8)$, $t_{12}\equiv t_{13}\equiv t_{23}\equiv0\;(\mbox{mod}\;4)$.
		\end{enumerate}
		For any representation, since $(v_{1},v_{2},v_{3})=-1$, we also have that $t_{123}\equiv{1\;}(\mbox{mod}{\;2})$. 
		Therefore it doesn't happen $v_{j}\subset v_{k}$, $j\neq k$ or $v_{j}\cap v_{k} = \emptyset$. If $v_{1}\subset v_{2}$, for example, we would have $t_{123}\equiv t_{13}\equiv {0\;}(\mbox{mod}{\;2}),$ which would be absurd because $t_{123}$ is odd. If $v_{2}\cap v_{3} = \emptyset$, for example, we would have $t_{123}\equiv t_{1}\equiv {0\;}(\mbox{mod}{\;4})$, which also can not happen.

	\end{remark}

	Let $L$ be a nonassociative code loop of rank $4$. By Lemma 3.4 of \cite{AR}, we can assume that $L$ has a generating set $X=\left\{a,b,c,d\right\}$  such that the nucleus of $L$ is generated by $d$, that is,  $N(L)=\mathbb{F}_{2}d$. Since $L$ is  nonassociative, we have $(a,b,c)=-1$.  By definition, the characteristic vector associated to $L$ has 14 coordinates. However, since  $(a,b,d)=(a,c,d)=(b,c,d)=1$ and $(a,b,c)=-1$, 
	we can omit the coordinates corresponding to the values of these associators and thus consider the characteristic vector of $L$  of the form $\lambda_{X}(L)=(\lambda_{1},\dots,\lambda_{10})$ (or simply $\lambda(L)=(\lambda_{1}\dots\lambda_{10})$), where
	
	%  $$a^{2}=(-1)^{\lambda_1}, b^{2}=~(-1)^{\lambda_2},c^{2}=(-1)^{\lambda_3}, d^{2}=(-1)^{\lambda_4},$$
	% $$ \left[a,b\right]=(-1)^{\lambda_5}, \left[a,c\right]=~(-1)^{\lambda_6}, \left[a,d\right]=(-1)^{\lambda_7},$$
	% $$\left[b,c\right]=~(-1)^{\lambda_8}, \left[b,d\right]=(-1)^{\lambda_9},\left[c,d\right]=(-1)^{\lambda_{10}}.$$
	
	\[
	\begin{aligned}
		a^{2} &= (-1)^{\lambda_1}, & b^{2} &= (-1)^{\lambda_2}, & c^{2} &= (-1)^{\lambda_3}, & d^{2} &= (-1)^{\lambda_4}, \\
		\left[a,b\right] &= (-1)^{\lambda_5}, & \left[a,c\right] &= (-1)^{\lambda_6}, & \left[a,d\right] &= (-1)^{\lambda_7}, \\
		\left[b,c\right] &= (-1)^{\lambda_8}, & \left[b,d\right] &= (-1)^{\lambda_9}, & \left[c,d\right] &= (-1)^{\lambda_{10}}.
	\end{aligned}
	\]

	%\vspace{0.1in}

	Similarly to the rank~$3$ case, the set of characteristic vectors of nonassociative code loops of rank~$4$ can be partitioned into orbits under the natural action of $GL_4(2)$, each corresponding to a unique loop up to isomorphism~\cite{8}.
	Theorem~\ref{th2.3} below lists one representative from each orbit and thus provides the classification of code loops of rank~$4$.
	
	\begin{theorem}[\cite{AR}]\label{th2.3}
		Consider $C_{1}^{4},...,C_{16}^{4}$ the code loops with the following characteristic vectors.
		Every nonassociative code loop of rank $4$ is isomorphic to one from the list. Moreover, no two of those loops are isomorphic to each other.
		\begin{table}[H]
			\centering
			{\begin{tabular}{llll|lllll|l}
					&  &  & $L$ & $\lambda(L)$ &  &  &  & $L$ & $\lambda(L)$ \\
					\cline{4-5}\cline{9-10}
					&  &  & $C^{4}_{1}$ & $(1110110100)$ &  &  &  & $C^{4}_{9}$ & $(0100001000)$ \\
					&  &  & $C^{4}_{2}$ & $(0000000000)$ &  &  &  & $C^{4}_{10}$ & $(0001111000)$ \\
					&  &  & $C^{4}_{3}$ & $(0000110100)$ &  &  &  & $C^{4}_{11}$ & $(0001001000)$ \\
					&  &  & $C^{4}_{4}$ & $(0010100000)$ &  &  &  & $C^{4}_{12}$ & $(0000001100)$ \\
					&  &  & $C^{4}_{5}$ & $(0000010100)$ &  &  &  & $C^{4}_{13}$ & $(0110111100)$ \\
					&  &  & $C^{4}_{6}$ & $(1111110100)$ &  &  &  & $C^{4}_{14}$ & $(0001001100)$ \\
					&  &  & $C^{4}_{7}$ & $(0001000000)$ &  &  &  & $C^{4}_{15}$ & $(1001001100)$ \\
					&  &  & $C^{4}_{8}$ & $(0000001000)$ &  &  &  & $C^{4}_{16}$ & $(0001111100)$ \\
			\end{tabular}}
			\label{tabvc4}
		\end{table}
		
	\end{theorem}

	The next remark lists the congruence classes of $t_i$, $t_{ij}$, and $t_{ijk}$ that will be used to calculate representations of code loops of rank~$4$. 
	
	\begin{remark}\label{rem2.4}
		As a consequence of the Theorem \ref{th2.3} we will obtain, for each code loop, representations with basis $\left\lbrace v_{1},v_{2},v_{3},v_{4}\right\rbrace $ that satisfies the following properties:
		
		\begin{itemize}
			\item If $\lambda_{i}=0$, then  $t_i\equiv 0\;(\mbox{mod} \; 8)$. Otherwise, we have $t_i\equiv 4\;(\mbox{mod} \; 8)$.
			\item If $\lambda_{ij}=0$, then  $t_{ij}\equiv 0\;(\mbox{mod} \; 4)$. Otherwise, we have $t_{ij}\equiv 2\;(\mbox{mod} \; 4)$.
			\item $t_{123}\equiv{1\;}(\mbox{mod}{\;2})$ and $t_{ij4}\equiv{0\;}(\mbox{mod}{\;2}) .$
		\end{itemize} 
		
	\end{remark}
	
	% \begin{remark}\label{rem2.final}
		% From the characteristic vector we do not obtain the absolute values of the weights 
		% $t_i$ and $t_{ij}$, but only their congruence classes: 
		% $t_i \pmod{8}$ from~(\ref{eq1}) and $t_{ij} \pmod{4}$ from~(\ref{eq2}). 
		% The exact values of these weights are analyzed in the next sections, where we consider all possible values (up to a fixed bound) compatible with these congruences, in order to determine the minimal representations. In other words, the characteristic vector is used to extract information about the 
		% vectors in $V$ and their pairwise and triple intersections, and all possible 
		% absolute values of the weights compatible with these congruences are then analyzed 
		% to obtain the minimal representations.
		% The possible congruence classes for $t_i$, $t_{ij}$ and $t_{ijk}$ for the code loops of rank $3$ and $4$  are listed in Remarks~\ref{rem2.2} and~\ref{rem2.4}.
		
		% \end{remark}
	
	\begin{remark}\label{rem2.final}
		From the characteristic vector we do not obtain the absolute values of the weights 
		$t_i$ and $t_{ij}$, but only their congruence classes: 
		$t_i \pmod{8}$ from~(\ref{eq1}) and $t_{ij} \pmod{4}$ from~(\ref{eq2}). 
		These congruences provide information about the vectors in a representation $V$ and 
		the properties of their intersections. The possible congruence classes for 
		$t_i$, $t_{ij}$ and $t_{ijk}$ for the code loops of rank $3$ and $4$ are listed in 
		Remarks~\ref{rem2.2} and~\ref{rem2.4}. 
		The exact values of these weights are analyzed in the next sections, where we 
		consider all possibilities (up to a fixed bound) compatible with the above 
		congruences, in order to determine the minimal representations.
	\end{remark}
	
	\section{Representations of code loops}

	% We prove by Theorems \ref{th3.4} and \ref{theorem3.6} below that, there are representations of nonassociatives code loops of rank $3$ and $4$ such that the degree of each representation is the smallest possible.
	
	Let $L$ be a code loop of rank $n$ with generators $B=\left\lbrace x_1,\cdots, x_n\right\rbrace $ and $V= \break Span_{\mathbb{F}_{2}}\left\lbrace v_{1},\cdots, v_{n}\right\rbrace $ a representation of $L$ with degree $m$. As in \cite{AR} we identify the $\mathbb{F}_{2}-$space  $\mathbb{F}_{2}^m$ with $P_{m}=\left\lbrace  \sigma \subseteq I_m \right\rbrace  $, then $v_i \subseteq I_m$. Let suppose that $\displaystyle\bigcup_{i=1}^{n}v_{i}=I_m$. Then we have an equivalence on $I_{m}: i \sim j \iff \left\lbrace i \in v_{p} \iff j \in v_{p} \right\rbrace $. Consider the corresponding partition of $I_m$ in equivalence classes: $I_m= \displaystyle\bigcup_{k=1}^{t}I_{m}^{k}$.

	The same partition we can obtain in the following way. Let $\sigma \subseteq I_n$ and 
	\begin{equation}\label{lemmaequivalenceclasses}
		v^{\sigma}=\left( \displaystyle\bigcap_{i \in \sigma}v_{i}\right)\setminus \left(\displaystyle\bigcup_{j \notin \sigma}v_{j}\right).
	\end{equation}

	\begin{lemma}\label{lemma3.1}
		In notation above, we have:
		
		% \begin{enumerate}
			% 	\item [i)] $v^\sigma \cap v^{\tau}=\emptyset$, \, $\sigma \neq \tau $;
			% 	\item [ii)] $\displaystyle\bigcup_{\sigma \subseteq I_{n}}v^{\sigma}=I_{m}$;
			% 	\item [iii)] $\left\lbrace v^{\sigma} \neq \emptyset \, | \,\sigma \subseteq I_{n} \right\rbrace = \left\lbrace I_{m}^{k}\,|\,k=1,\cdots,t\right\rbrace  $.
			% \end{enumerate}
		
		\begin{enumerate}[label=\roman*)]
			\item  $v^\sigma \cap v^{\tau}=\emptyset$, \, $\sigma \neq \tau $;
			\item  $\displaystyle\bigcup_{\sigma \subseteq I_{n}}v^{\sigma}=I_{m}$;
			\item  $\left\lbrace v^{\sigma} \neq \emptyset \, | \,\sigma \subseteq I_{n} \right\rbrace = \left\lbrace I_{m}^{k}\,|\,k=1,\cdots,t\right\rbrace  $.
		\end{enumerate}
	\end{lemma}
	\begin{proof}%\leavevmode\\[-6pt] 
		\par \noindent
		\begin{enumerate}
			\item [i)] Let $i \in v^{\sigma} \cap v^{\tau}$, and $j\in \sigma \setminus \tau$. Hence, by (\ref{lemmaequivalenceclasses}), $i \in v_{k}$, for all $k \in \sigma$. Since $j \neq \tau$, then $i \notin v_{j}$ by (\ref{lemmaequivalenceclasses}). As $j \in \sigma$ we get a contradiction.
			
			\item [ii)]  Let $i \in I_{n}$, and define $\sigma=\left\lbrace k \in I_{n}\,|\,i \in v_{k} \right\rbrace $. Then for $q \notin \sigma$ we have $i \notin v_{q}$, and $i \in v^{\sigma}$.
			
			\item [iii)] Let $i \sim j$ and $i \in v^{\sigma}$, then $i \in v_{k}$, $k \in \sigma$, hence $j \in v_{k}$. Since $i \in v^{\sigma}$, $i \notin v_{p}$, $p \notin \sigma$. As $i \sim j$, $j \notin v_{p}$. Hence $j \in v^{\sigma}$. It means that $I_{m}^{k} \subseteq v^{\sigma}$, if $i,j \in I_{m}^{k}$. Let $k \notin I_{m}^{k}$, $k \in v^{\sigma}$, then there exists $v_{q}$ such that $i \notin v_{q}$, and $j \notin v_{q}$, $k \in v_{q}$. By $k \in v^{\sigma}$, $k \in v_{q}$ and (\ref{lemmaequivalenceclasses}) we get that $q \in \sigma$. On the other hand, since $i \notin v_{q}$ and $i \in v^{\sigma}$ hence $q \notin \sigma$. Contradiction.

		\end{enumerate}
	\end{proof}

	%Lemma 2
	\begin{lemma}\label{ind.generators}The partition $I_{m}=\displaystyle\bigcup_{k=1}^{t}I_{m}^{k}$ does not dependent on the choise of the system of generators $v_{1},\cdots,v_{n}$.
	\end{lemma}
	\begin{proof} We have $L=V\cup -V$, where $V=Span_{\mathbb{F}_{2}}\left\lbrace v_{1},\cdots,v_{n}\right\rbrace $. Recall that in $V$ the operation is $v+w = \left(v\setminus w \right) \cup \left(w\setminus v \right) $. Let $v \in V$ and $i \in v\cap I_{m}^{k}$. We need to prove that $I_{m}^{k}\subseteq v$. Suppose that $j \in I_{m}^{k}\setminus v$. Let $v=\displaystyle\sum_{k \in \sigma \subseteq I_{k}}v_{k}$. If $|\sigma|=1$, then $v=v_{k}$, hence $i \in v$, $j \notin v$ and $i\not\sim j$. Contradiction with $i,j \in I_{m}^{k}$.
		
		Let $|\sigma|=s$ and for all $|\tau|< s$ we have $i \in v_{\tau}=\displaystyle\sum_{k \in \tau}v_{k}$ if and only if $j \in v_{\tau}$.
		
		We have $v_{\sigma}=v_{\tau}+v_{l}$, $\tau = \sigma \setminus l$, $|\tau|<s$. Let $i \in v_{\sigma}$ and $i\in v_{\tau}$, then $i \notin v_{l}$, since $v_{\sigma}=\left( v_{\tau}\setminus v_{l} \right) \cup \left(v_{l}\setminus v_{\tau} \right) $. By $i \sim j$ we get $j \notin v_{l}$ and $j \in v_{\tau}$, hence $j \in v_{\sigma}$. Let $i \in v_{\sigma}$, $i \notin v_{\tau}$, then $i \in v_{l}$. By induction $j \in v_{l}$ and $j \notin v_{\tau}$. Hence $j \in v_{\sigma}$.
	\end{proof}

	We will consider only representations such that, for any equivalence class $X$, we have $|X| < 8$. The choice is arbitrary and it only serves the purpose of making the enumeration of all reduced representations tractable. We call these representations by \textbf{reduced representations}. 
	
	\begin{definition} A reduced representation $V$ is \textbf{minimal} if whenever $L(V)$ is isomorphic to $L(W)$ then $\deg(V) \leq \deg(W)$, where $W$ is another representation.
	\end{definition}
	
	By (\ref{lemmaequivalenceclasses}), we will get all the equivalence classes on $I_{m}$ and we will show that it's possible to find all reduced and minimal representations of degree $m$ for code loops of rank $3$ and $4$ using this classes. To do this, it is necessary to find the cardinality of this equivalence classes and show how this values can be obtained just by knowing the weights, $t_{i}, t_{ij}$ and $t_{ijk}$, with $i,j,k \in \left\lbrace 1,2,3,4\right\rbrace $.  For each choice of these weights, there is a different representation. This choice must be in accordance with the properties mentioned in the  Remark \ref{rem2.2} and \ref{rem2.4}.

	\begin{definition}
		Let $V$ be a $m$-degree representation of a code loop $L$ and  $X_{1},...,X_{r}$ the equivalence classes on $I_m$. We define the \textbf{type} of $V$ as the vector $(|X_{1}|,...,|X_{r}|)$ such that $|X_{1}| \leq |X_{2}| \leq ...\leq |X_{r}|$. When necessary, we can write $(|X_{1}|...|X_{r}|)$.
	\end{definition} 
	
	\begin{definition}
		Let $V_{1}$ and $V_{2}$ be doubly even binary codes of ${\mathbb F}_{2}^{m}$. We say that $V_{1}$ and $V_{2}$ are  equivalent codes if and only if there is a bijection $\varphi \in S_{m}$ such that $V_{1}^{\varphi}=V_{2}$.
	\end{definition}
	
	%%%%%%%
	%%%%%%%%%%%%%%%%%%%%%%%%%%%%%%%%%%%%%%%%%%%%%%%%%%%%%%%%%seção de representações reduzidas de posto 3
	\subsection{Code Loops of rank $3$}\label{sec3.1}
	
	Let $C_{i}^{3}$, $i=1,\dots,5,$ be the nonassociative code loops of rank $3$ and $\lambda(C_{i}^{3})$ the associated characteristic vector  as shown in Section \ref{section2}. We also denote the corresponding representations by $V_{i}^{3}$ and assume that $V_{i}^{3} = Span_{\mathbb{F}_{2}}\left\lbrace v_{1},v_{2},v_{3}\right\rbrace$.

	Let $m$ be the degree of $V_{i}^{3}$ and $I_{m}=\{1,2,\dots,m\}$, such that $v_1, v_2, v_3 \subset I_{m}$ and $v_{1}\cup v_{2}\cup v_{3} = I_m$. According to Lemma \ref{lemma3.1}, the sets bellow determine all equivalence classes on $I_{m}$ (see Fig. \ref{figura1}):
	\[X_{123}=v_{1}\cap v_{2}\cap v_{3},\,\, X_{12}=(v_{1}\cap v_{2})\setminus v_{3}, \,\,X_{13}=(v_{1}\cap v_{3})\setminus v_{2},\,\,X_{23}=(v_{2}\cap v_{3})\setminus v_{1},\] \[X_{1}=v_{1}\setminus(v_{2}\cup v_{3}),\,\, X_{2}=v_{2}\setminus(v_{1}\cup v_{3}),\,\,  X_{3}=v_{3}\setminus(v_{1}\cup v_{2}).\]

	%\vspace{-0.7cm}
	%\begin{figure}[H]
	%	\centering
	%	\includegraphics[scale=0.5]{fig1}
	%	\caption{Equivalence classes on $I_{m}$ for code loops of rank $3$}\label{figura1}
	%\end{figure}
	
	\begin{figure}[h!]
		\centering
		\includegraphics[width=0.4\linewidth]{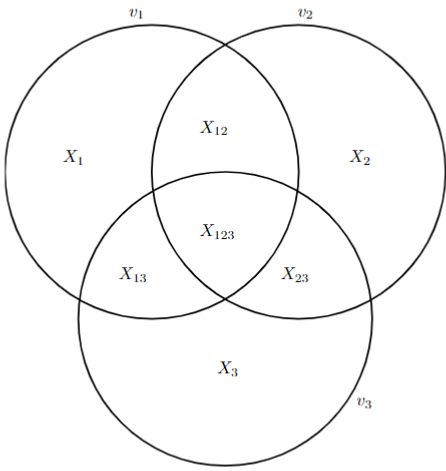}
		\caption{Equivalence classes on $I_{m}$ for code loops of rank $3$}
		\label{figura1}
	\end{figure}

	Now, note that the generators of $V_{i}^{3}$ depend on the equivalence classes since we can write them as follows:
	\[
	v_{1}=X_{123}\cup X_{12}\cup X_{13} \cup X_{1},
	v_{2}=X_{123}\cup X_{12}\cup X_{23} \cup X_{2}, 
	v_{3}=X_{123}\cup X_{13}\cup X_{23} \cup X_{3}.
	\]

	So if we find the equivalence classes on $I_m$, we will determine the generators of the representation. As the elements of the representation depend on the weights of $v_i$, $v_{i}\cap v_{j}$ $(i,j \in \left\lbrace 1,2,3 \right\rbrace )$ and $v_{1}\cap v_{2}\cap v_{3}$, we can establish a relationship between the cardinalities of the classes and these weights.

	Denote the cardinalities of the sets $X_{123}$, $X_{ij}$ and $X_{i}$ by $x_{123}$, $x_{ij}$ and $x_{i}$,  respectively; $i,j=1,2,3$, and assume that $x_{123}=t_{123}$. These cardinalities and the weights are related through the following system in the variables $x_{12}$, $x_{13}$, $x_{23}$, $x_{1}$, $x_{2}$ and $x_{3}$:

	\begin{equation}\label{eq.3.2}\left\lbrace 
		\begin{array}{l}
			t_{123}+x_{12}=t_{12}\\
			t_{123}+x_{13}=t_{13}\\
			t_{123}+x_{23}=t_{23}\\
			t_{123}+x_{12}+x_{13}+x_1=t_{1}\\
			t_{123}+x_{12}+x_{23}+x_2=t_{2}\\
			t_{123}+x_{13}+x_{23}+x_3=t_{3}\\
		\end{array}
		\right.
	\end{equation}
	
	Then the solutions are given by:

	\begin{equation}\label{}\left\lbrace
		\begin{array}{l}
			x_{12}=t_{12} - t_{123}\\
			x_{13}=t_{13} - t_{123}\\
			x_{23}=t_{23} - t_{123}\\
			x_{1}=t_{1} - t_{12} - t_{13} + t_{123}\\
			x_{2}=t_{2} - t_{12} - t_{23} + t_{123}\\
			x_{3}=t_{3} - t_{13} - t_{23} + t_{123}\\
		\end{array}
		\right.
	\end{equation}

	Therefore, to determine a representation of a code loop of rank 3, first  choose values for $t_{123}$, $t_{ij}$ and $t_{i}$, $i,j=1,2,3$; after look for the solution for the Linear System (\ref{eq.3.2}). Since we want equivalence classes whose cardinality is strictly less than 8, we consider  solutions $v=(x_{12}, x_{13}, x_{23}, x_1, x_2, x_3)$ such that $0 \leq x_{i} \leq 7$ and $0 \leq x_{ij} \leq 7$, with $i,j=1,2, 3$,

	\begin{remark} It follows from 7 equivalence classes of at most 7 elements that:
		\begin{itemize}
			\item $t_{123}=1,3,5$ or $7$;
			\item $t_{ij}=2,4,6,8,10,12$ or $14$; 
			\item $t_{i}=4,8,12,16,20,24$ or $28$;
			\item $7\leq m \leq 49$.
		\end{itemize}
		
	\end{remark}

	Note that the degree $m$ and the type of $V_{i}^{3}$ are obtained from these soluções.  In fact, $m=x_{123}+x_{12}+x_{13}+x_{23}+x_{1}+x_{2}+x_{3}$ and the type is a vector determined by putting the coordinates of $(x_{123}, x_{12}, x_{13}, x_{23}, x_1, x_2, x_3)$  in ascending order.

	After finding the system's solutions, that is, the cardinalities of the equivalence classes over $I_m$, it will be necessary to label its elements. By choice, we label the elements as follows: 
	%  \vspace{-0.4in}
	
	% \small{	\begin{table}[hh!]
			% 		$X_{123}=\{1,\cdots,t_{123}\},$   \\ $X_{12}=\{t_{123}+1,\cdots,t_{123}+x_{12}\},$  \\
			% 		$X_{13}=\{t_{123}+x_{12}+1,\cdots,t_{123}+x_{12}+x_{13}\},$  \\	
			% 		$X_{1}=\{t_{123}+x_{12}+x_{13}+1,\cdots,t_{123}+x_{12}+x_{13}+x_{1}\},$  \\			
			% 		$X_{23}=\{t_{123}+x_{12}+x_{13}+x_{1}+1,\cdots,t_{123}+x_{12}+x_{13}+x_{1}+x_{23}\},$ \\			
			% 		$X_{2}=\{t_{123}+x_{12}+x_{13}+x_{23}+x_{1}+1,\cdots,t_{123}+x_{12}+x_{13}+x_{23}+x_{1}+x_{2}\},$  \\			
			% 		$X_{3}=\{t_{123}+x_{12}+x_{13}+x_{23}+x_{1}+x_{2}+1,\cdots,t_{123}+x_{12}+x_{13}+x_{23}+x_{1}+x_{2}+x_{3}\}.$ 			
			% \end{table}}
	
	\begin{small}
		\begin{align*}
			X_{123} &= \{1,\cdots,t_{123}\},\\
			X_{12}  &= \{t_{123}+1,\cdots,t_{123}+x_{12}\},\\
			X_{13}  &= \{t_{123}+x_{12}+1,\cdots,t_{123}+x_{12}+x_{13}\},\\
			X_{1}   &= \{t_{123}+x_{12}+x_{13}+1,\cdots,t_{123}+x_{12}+x_{13}+x_{1}\},\\
			X_{23}  &= \{t_{123}+x_{12}+x_{13}+x_{1}+1,\cdots,t_{123}+x_{12}+x_{13}+x_{1}+x_{23}\},\\
			X_{2}   &= \{t_{123}+x_{12}+x_{13}+x_{23}+x_{1}+1,\cdots,t_{123}+x_{12}+x_{13}+x_{23}+x_{1}+x_{2}\},\\
			X_{3}   &= \{t_{123}+x_{12}+x_{13}+x_{23}+x_{1}+x_{2}+1,\cdots,t_{123}+x_{12}+x_{13}+x_{23}+x_{1}+x_{2}+x_{3}\}.
		\end{align*}
	\end{small}
	
	%\vspace{-0.4in}
	
	%%%%%%%%%%%%%%%%%%%%%%%%%%%%%END\vspace{-0.4cm}
	\subsection{Code Loops of Rank 4}\label{section3.2}

	Let $V$ a $m$-degree reduced representation with basis $\left\lbrace v_1,v_2,v_3,v_4\right\rbrace $ corresponding to a code loop of rank $4$. Here we also use the notation $t$ or $t_{1234}$ to denote the weight of the 4 generators, that is, $t =t_{1234} = |v_{1}\cap v_{2}\cap v_{3}\cap v_{4}|$.
	
	As we know, the weights of the generators and their intersections are determined through the characteristic vector associated with the code loop. So, for our first step, take a characteristic vector associated with a code loop of rank $4$.   
	Since $t_{123}\equiv{1\;}(\mbox{mod}\;2)$, $t_{ij4}\equiv{0\;}(\mbox{mod}{\;2})$, $t_{ij}\equiv{0\;}(\mbox{mod}{\;2})$ and $t_{i}\equiv{0\;}(\mbox{mod}{\;4})$ and $V$ is reduced, then we will get $1\leq t_{123}\leq 13$, $0\leq t_{ij4}\leq 14$, $2\leq t_{ij}\leq 28$ $(i,j=1,2,3; \,i\neq j)$, $0\leq t_{i4}\leq 28$ and $4\leq t_{i}\leq 56$.

	\begin{remark}\label{rem3.7}
		Consider $I_{m}=\{1,2,\dots,m\}$, such that $v_1, v_2, v_3, v_4 \subset I_{m}$ and \linebreak $v_{1}\cup v_{2}\cup v_{3}\cup v_{4} = I_m$. Acoording to Lemma \ref{lemma3.1}, the equivalence classes on $I_{m}=\{1,2,\cdots,m\}$ are obtained by calculating the following sets (see Fig. \ref{figura2}):
		\begin{flushleft}
			$X_{1234}=v_1\cap v_2 \cap v_3 \cap v_4,\,\, \,\,\,X_{123}= (v_1\cap v_2 \cap v_3)\setminus v_4, \,\,\,\,X_{124}=(v_1\cap v_2 \cap v_4)\setminus v_3,$\\ 
			$X_{134}=(v_1\cap v_3 \cap v_4)\setminus v_2,\,\,\,\,X_{234}=(v_2\cap v_3 \cap v_4)\setminus v_1,\,\,\,\,X_{12}=(v_1\cap v_2)\setminus (v_3 \cup v_4),$ 
			$X_{13}=(v_1\cap v_3)\setminus (v_2 \cup v_4),\,\,X_{14}=(v_1\cap v_4)\setminus (v_2 \cup v_3),\,\,X_{23}=(v_2\cap v_3)\setminus (v_1 \cup v_4),$
			$X_{24}=(v_2\cap v_4)\setminus (v_1 \cup v_3),\,\,X_{34}=(v_3\cap v_4)\setminus (v_1 \cup v_2),\,\,X_{1}=v_1\setminus (v_2 \cup v_3 \cup v_4),$	$X_{2}=v_2\setminus (v_1\cup v_3 \cup v_4),\,\,\,\,\,\,\,\,\,X_{3}=v_3\setminus (v_1 \cup v_2 \cup v_4),\,\,\,\,\,\,\,\,\,X_{4}=v_4\setminus (v_1 \cup v_2 \cup v_3).$	
		\end{flushleft}
	\end{remark}
	\begin{figure}[h!]
		\centering
		\definecolor{uuuuuu}{rgb}{0.26666666666666666,0.26666666666666666,0.26666666666666666}
		\definecolor{ffffff}{rgb}{1,1,1}
		\begin{tikzpicture}[line cap=round,line join=round,>=triangle 45,x=1.5cm,y=0.8cm,scale=1.8]
			\path[use as bounding box] (0,0) rectangle (4,4);  %ajuda a centralizar como quero
			\clip(-0.5,-0.3) rectangle (8.0,5.6);
			\fill[line width=0.8pt,color=ffffff] (4,0) -- (4,4) -- (0,4) -- (0,0) -- cycle;
			\draw [green,line width=0.7pt] (1,4)-- (1,0);
			\draw [orange,line width=0.7pt] (2,4)-- (2,0);
			\draw [green,line width=0.7pt] (3,4)-- (3,0);
			\draw [blue,line width=0.7pt] (0,3)-- (4,3);
			\draw [red,line width=0.7pt] (0,2)-- (4,2);
			\draw [blue,line width=0.7pt] (0,1)-- (4,1);
			\draw [black,line width=0.7pt] (0,4)-- (4,4);
			\draw [orange,line width=0.7pt] (4,4)-- (4,0);
			\draw [red,line width=0.7pt] (4,0)-- (0,0);
			\draw [black,line width=0.7pt] (0,0)-- (0,4);
			\draw (2.3,1.65) node[anchor=north west] {$X_{1234}$};
			\draw (1.3,1.65) node[anchor=north west] {$X_{123}$};
			\draw (2.3,2.65) node[anchor=north west] {$X_{124}$};
			\draw (1.3,2.65) node[anchor=north west] {$X_{12}$};
			\draw (0.3,2.65) node[anchor=north west] {$X_{1}$};
			\draw (0.3,1.65) node[anchor=north west] {$X_{13}$};
			\draw (3.3,1.65) node[anchor=north west] {$X_{134}$};
			\draw (1.3,3.65) node[anchor=north west] {$X_{2}$};
			\draw (3.3,2.65) node[anchor=north west] {$X_{14}$};
			\draw (0.3,0.65) node[anchor=north west] {$X_{3}$};
			\draw (1.3,0.65) node[anchor=north west] {$X_{23}$};
			\draw (2.3,0.65) node[anchor=north west] {$X_{234}$};
			\draw (3.3,0.65) node[anchor=north west] {$X_{34}$};
			\draw (2.3,3.65) node[anchor=north west] {$X_{24}$};
			\draw (3.3,3.65) node[anchor=north west] {$X_{4}$};
			\draw (-0.2,0.65) node[anchor=north west] {$v_{3}$};
			\draw (-0.2,2.65) node[anchor=north west] {$v_{1}$};
			\draw (-0.50,1.65) node[anchor=north west] {$v_{1}\cap v_{3}$};
			\draw (1.3,4.35) node[anchor=north west] {$v_{2}$};
			\draw (2.2,4.35) node[anchor=north west] {$v_{2}\cap v_{4}$};
			\draw (3.3,4.35) node[anchor=north west] {$v_{4}$};			
			%\draw (1,3.690766245907773) node[anchor=north west] {1};
			%\draw (2.2811463617400642,3.392008280090975) node[anchor=north west] {2};
			\begin{scriptsize}
				\draw [fill=uuuuuu] (0,4) circle (0.1pt);
				\draw [fill=uuuuuu] (4,4) circle (0.1pt);
				\draw [fill=uuuuuu] (4,0) circle (0.1pt);
				\draw [fill=uuuuuu] (0,0) circle (0.1pt);
			\end{scriptsize}
		\end{tikzpicture}
		\caption{Equivalence classes on $I_{m}$ for code loops of rank $4$}\label{figura2}
	\end{figure}

	As done in previous section, denote by $x_{1234}$, $x_{ijk}$, $x_{ij}$ or $x_{i}$; $i,j,k=1,2,3,4$; the cardinality of this sets. Also assume $t_{1234}=x_{1234}$. The second step to finding reduced representations is to  choose values for $t_{1234}$, $t_{ijk}$, $t_{ij}$ and $t_{i}$, with $i,j,k=1,2,3,4$, that satisfy the properties presented in the Remark \ref{rem2.4}.

	In this case, we can also relate these weights and the cardinalities of the sets listed in the Remark \ref{rem3.7} through the following linear system of $14$ equations in the variables $x_{123}$, $x_{124}$, $x_{134}$, $x_{234}$, $x_{12}$, $x_{13}$, $x_{14}$, $x_{23}$, $x_{24}$, $x_{34}$, $x_{1}$, $x_{2}$, $x_{3}$ and $x_{4}$:
	
	\begin{eqnarray}\label{system3.4}\left\lbrace 
		\begin{array}{l}
			t+x_{123}=t_{123}\\
			t+x_{124}=t_{124}\\
			t+x_{134}=t_{134}\\
			t+x_{234}=t_{234}\\
			t+x_{123}+x_{124}+x_{12}=t_{12}\\
			t+x_{123}+x_{134}+x_{13}=t_{13}\\
			t+x_{124}+x_{134}+x_{14}=t_{14}\\
			t+x_{123}+x_{234}+x_{23}=t_{23}\\
			t+x_{124}+x_{234}+x_{24}=t_{24}\\
			t+x_{134}+x_{234}+x_{34}=t_{34}\\
			t+x_{123}+x_{12}+x_{124}+x_{134}+x_{13}+x_{14}+x_1=t_{1}\\
			t+x_{123}+x_{12}+x_{124}+x_{234}+x_{23}+x_{24}+x_2=t_{2}\\
			t+x_{123}+x_{13}+x_{134}+x_{234}+x_{23}+x_{34}+x_3=t_{3}\\
			t+x_{124}+x_{134}+x_{234}+x_{14}+x_{24}+x_{34}+x_4=t_{4}\\	
		\end{array}
		\right.
	\end{eqnarray}
	
	The third step is to solve this system. Solving it, we obtain: 
	
	\begin{equation*}\label{}\left\lbrace 
		\begin{array}{l}
			x_{123}=t_{123}-t\\
			x_{124}=t_{124}-t\\
			x_{134}=t_{134}-t\\
			x_{234}=t_{234}-t\\
			x_{12}=t_{12}-t_{123}-t_{124}+t\\
			x_{13}=t_{13}-t_{123}-t_{134}+t\\
			x_{14}=t_{14}-t_{124}-t_{134}+t\\
			x_{23}=t_{23}-t_{123}-t_{234}+t\\
			x_{24}=t_{24}-t_{124}-t_{234}+t\\
			x_{34}=t_{34}-t_{134}-t_{234}+t\\
			x_{1}=t_{1}-t_{12}-t_{13}+t_{123}-t_{14}+t_{124}+t_{134}-t\\
			x_{2}=t_{2}-t_{12}-t_{23}+t_{123}-t_{24}+t_{234}+t_{124}-t\\
			x_{3}=t_{3}-t_{13}-t_{23}+t_{123}-t_{34}+t_{234}+t_{134}-t\\
			x_{4}=t_{4}-t_{14}-t_{24}+t_{124}-t_{34}+t_{134}+t_{234}-t\\	
		\end{array}
		\right.
	\end{equation*}

	Remember that the reduced representations will be get for all those solutions that satisfies the following conditions, for all $i,j,k \in \{1,2,3,4\}$:  
	\begin{equation*}
		0\leq x_{i}\leq 7,\,\,\, 0\leq x_{ij}\leq 7,\,\,\,0\leq x_{ijk}\leq 7, \,\,\, 0\leq x_{1234}\leq 7.
	\end{equation*}

	With those solutions, the sets $X_{1234},X_{123}$, $X_{124}$, $X_{134}$, $X_{234}$, $X_{12}$, $X_{13}$, $X_{14}$, $X_{23}$, $X_{24}$, $X_{34}$, $X_{1}$, $X_{2}$, $X_{3}$ and $X_{4}$ are determined following the labeling chosen in the Table \ref{tb6}. The order chosen to determine these sets is arbitrary. Note that only non-empty sets determine the equivalence classes on $I_{m}$, where $m$ is the degree of $V$. At last, find the generators $v_1, v_2, v_3$ and $v_4$ of $V$ writing:
	%\vspace{-0.2cm}
	\begin{eqnarray*}
		v_1 &=& X_{1234}\cup X_{123}\cup X_{124}\cup X_{134}\cup X_{12} \cup X_{13} \cup X_{14} \cup X_1,\\
		v_2 &=& X_{1234}\cup X_{123}\cup X_{124}\cup X_{234}\cup X_{12} \cup X_{23} \cup X_{24} \cup X_2,\\
		v_3 &=& X_{1234}\cup X_{123}\cup X_{134}\cup X_{234}\cup X_{13} \cup X_{23} \cup X_{34} \cup X_3,\\
		v_4 &=& X_{1234}\cup X_{124}\cup X_{134}\cup X_{234}\cup X_{14} \cup X_{24} \cup X_{34} \cup X_4.
	\end{eqnarray*}
	
	As an illustration, we will present a reduced representation, denoted by $V_{16}^4$, for the code loop with characteristic vector $(0001111100)$. We will use the following notation:  
	\begin{eqnarray*}
		u=(t,t_{123},t_{124},t_{134},t_{234},t_{12},t_{13},t_{14},t_{23},t_{24},t_{34},t_{1},t_{2},t_{3})\\  v=(x_{123},x_{124},x_{134},x_{234},x_{12},x_{13},x_{14},x_{23},x_{24},x_{34},x_{1},x_{2},x_{3})
	\end{eqnarray*}
	
	If  $u=\left[ 0,1,0,2,0,2,6,2,6,0,4,8,8,16,4\right]$, then   $v=(1,0,2,0,1,3,0,5,0,2,1,1,3,0)$ is the solution of the System (\ref{system3.4}). Thus the generators of $V_{16}^4$ are given by:
	\begin{eqnarray*}
		v_1&=&\left\lbrace 1,2,3,4,5,6,7,8 \right\rbrace , \\
		v_2&=&\left\lbrace 1,4,9,10,11,12,13,14 \right\rbrace,\\
		v_3&=&\left\lbrace 1,2,3,5,6,7,9,10,11,12,13,15,16,17,18,19 \right\rbrace ,\\
		v_4&=&\left\lbrace 2,3,15,16\right\rbrace .
	\end{eqnarray*}

	%	\vspace{-0.7cm}
	
	\begin{table}[H]
		\centering
		\scalefont{0.6}
		\caption{Labeling elements of equivalence classes on $I_{m}$. \label{tb6}} 
		{\begin{tabular}{|l|l|l|} \hline
				$X_{1234}$&$\left\lbrace \;\right\rbrace$ \;\mbox{if}\; $x_{1234}=0$ &  \\ 
				&$\left\lbrace 1,\dots,x_{1234}\right\rbrace$ \;\mbox{if}\; $x_{1234}\neq 0$&  \\ \hline
				$X_{123}$ &$\left\lbrace \;\right\rbrace$ \;\mbox{if}\; $x_{123}=0$ &$n_{123}=x_{1234}+x_{123}$\\
				& $\left\lbrace x_{1234}+1,\dots,n_{123}  \right\rbrace$ &   \\ \hline
				$X_{124}$ & $\left\lbrace \; \right\rbrace$ \;\mbox{if}\; $x_{124}=0$ & $n_{124}=n_{123}+x_{124}$\\ 
				&$\left\lbrace n_{123}+1, \dots, n_{124} \right\rbrace$ \;\mbox{if}\; $x_{124}\neq 0$&\\ \hline
				$X_{12}$ & $\left\lbrace \; \right\rbrace$ \;\mbox{if}\; $x_{12}=0$ & $n_{12}=n_{124}+x_{12}$ \\
				&$\left\lbrace n_{134}+1, \dots, n_{12} \right\rbrace$ & \\ \hline
				
				$X_{134}$ & $\left\lbrace \; \right\rbrace$ \;\mbox{if}\; $x_{134}=0$ & $n_{134}=n_{12}+x_{134}$\\ 
				&$\left\lbrace n_{12}+1, \dots, n_{134} \right\rbrace$ \;\mbox{if}\; $x_{134}\neq 0$ &\\ \hline
				
				$X_{13}$ &$\left\lbrace \; \right\rbrace$ \;\mbox{if}\; $x_{13}=0$ & $n_{13}=n_{134}+x_{13}$ \\
				&$\left\lbrace n_{134}+1, \dots, n_{13} \right\rbrace$ & \\ \hline
				
				$X_{14}$ & $\left\lbrace \; \right\rbrace$ \;\mbox{if}\; $x_{14}=0$ & $n_{14}=n_{13}+x_{14}$\\ 
				& $\left\lbrace n_{13}+1, \dots, n_{14} \right\rbrace$ \;\mbox{if}\; $x_{14}\neq 0$ & \\ \hline
				
				$X_{1}$& $\left\lbrace \; \right\rbrace$ \;\mbox{if}\; $x_{1}=0$ & $n_1 = n_{14}+x_1$ \\
				&$\left\lbrace n_{14}+1, \dots, n_1 \right\rbrace$ & \\ \hline
				
				$X_{234}$ & $\left\lbrace \; \right\rbrace$ \;\mbox{if}\; $x_{234}=0$  & $n_{234}=n_1 + x_{234}$\\ 
				&$\left\lbrace n_{1}+1, \dots, n_{234} \right\rbrace$ \;\mbox{if}\; $x_{234}\neq 0$ & \\ \hline
				
				$X_{23}$& $\left\lbrace \; \right\rbrace$ \;\mbox{if}\; $x_{23}=0$   &  $n_{23}=n_{234} + x_{23}$\\
				&$\left\lbrace n_{234}+1, \dots, n_{23} \right\rbrace$ &  \\ \hline
				
				$X_{24}$ & $\left\lbrace \; \right\rbrace$ \;\mbox{if}\; $x_{24}= 0$ & $n_{24}=n_{23} + x_{24}$\\ 
				&$\left\lbrace n_{23}+1, \dots, n_{24} \right\rbrace$ \;\mbox{if}\; $x_{24}\neq 0$ &\\ \hline
				
				$X_{2}$& $\left\lbrace \; \right\rbrace$ \;\mbox{if}\; $x_{2}= 0$  & $n_{2}=n_{24}+x_{2}$  \\
				&$\left\lbrace n_{24}+1, \dots, n_{2} \right\rbrace$ & \\ \hline
				
				$X_{34}$&$\left\lbrace \; \right\rbrace$ \;\mbox{if}\; $x_{34}= 0$ & $n_{34}=n_{2} + x_{34}$ \\ 
				&$\left\lbrace n_{2}+1, \dots, n_{34} \right\rbrace$ \;\mbox{if}\; $x_{34}\neq 0$ & \\ \hline
				
				$X_{3}$ & $\left\lbrace \; \right\rbrace$ \;\mbox{if}\; $x_{3}= 0$   &  $n_{3}=n_{34} + x_{3}$ \\
				&$\left\lbrace n_{34}+1, \dots, n_{3} \right\rbrace$ & \\ \hline
				
				$X_{4}$ & $\left\lbrace \; \right\rbrace$ \;\mbox{if}\; $x_{4}= 0$ & $n_{4}=n_{3} + x_{4}$\\ 
				&$\left\lbrace n_{3}+1, \dots, n_{4} \right\rbrace$ \;\mbox{if}\; $x_{4}\neq 0$ &\\ \hline
		\end{tabular}}
	\end{table}

	%%%%%%%%%%%%%%%%%
	
	\section{Minimal Representations}
	
	Below, we present in the theorems \ref{th3.4} and \ref{theorem3.6} the representations of nonassociative code loops of rank $3$ and $4$ where the degree of each representation is minimal and we determine the degrees of these representations. Here we use the notation $\left\langle v_{1},v_{2},\cdots, v_{n}\right\rangle$ to denote the vector subspace of the space $\mathbb{F}_{2}^{m}$ generated by $v_{1},v_{2},\cdots,v_{n}$. We also use the notations introduced in Section \ref{section3.2}.

	We define the length of the vector $v_i + v_j$ to be the cardinality of the set $v_i \cup v_j $, that is, the length of $v_i + v_j$ is the number $l_{ij}=t_i +t_j - t_{ij}$. Analogously we define the length of the vector $v_i +v_j +v_k$ to be the cardinality of the set $v_i \cup v_j \cup v_k$, that is, the length of $v_i +v_j +v_k$ is the number $l_{ijk}=t_i +t_j +t_k - (t_{ij} + t_{ik} + t_{jk}) + t_{ijk}$.

	\begin{theorem}\label{th3.4}
		The code loops $C_{1}^{3},\dots,C_{5}^{3}$ have the following minimal representations $V_{1},\dots,V_{5}$, which are given by
		
		\vspace{0.05in}
		\noindent
		$V_{1}=\left\langle (1,2,3,4),(1,2,5,6),(1,3,5,7)\right\rangle,$\\
		$V_{2}=\left\langle (1,2,3,4,5,6,7,8),(1,2,3,4,9,10,11,12),(1,5,6,7,9,10,11,13)\right\rangle,$\\
		$V_{3}=\left\langle (1,2,3,4,5,6,7,8),(1,2,3,4,5,6,9,10),(1,2,3,4,5,7,9,11)\right\rangle,$\\
		$V_{4}=\left\langle (1,2,3,4),(1,2,5-14),(1,3,5,6,7,8,9,10,11,15,16,17)\right\rangle,$\\
		$V_{5}=\left\langle (1-12),(1-8,13,14,15,16),(1,2,3,4,5,9,10,11,13,14,15,17)\right\rangle.$
	\end{theorem}
	
	%%%%%%%%%%%%%%%DEMONSTRATION THEOREM CODE LOOP OF RANK 3
	\begin{proof}
		We consider the ${\mathbb F}_{2}$-subspaces $V_{1},\dots,V_{5}$ of ${{\mathbb F}_{2}^{7}},{{\mathbb F}_{2}^{13}},{{\mathbb F}_{2}^{11}},{{\mathbb F}_{2}^{17}},{{\mathbb F}_{2}^{17}},$ respectively, as above. First we will see that each space $V_{i}$  is a representation of $C_{i}^{3}$, that is,  a doubly even code of ${{\mathbb F}_{2}^{m}}$, for some $m$ and that $C_{i}^{3} \simeq L(V_{i})$, $i=1,\dots,5$. 
		
		Since the elements of $V_{5}$ are $v_{0}=0$, $v_{1}=(1-12)$,   $v_{2}=(1-8,13-16)$, $v_{3}=(1-5,9-11,13-15,17)$, $v_{4}=v_{1}+v_{2}=(9-16)$, $v_{5}=v_{1}+v_{3}=(6,7,8,12,13,14,15,17)$, $v_{6}=v_{2}+v_{3}=(6-11,16,17)$ and $v_{7}=v_{1}+v_{2}+v_{3}=(1-5,12,16,17)$, we see, clearly, that all the vectors have weight with multiplicity  $4$ and the weight of the intersection of each pair of vectors is even. Thus, $V_{5}$ is a doubly even code. Analogously, we prove that $V_{i}$, $i=1,\dots,4$ is a doubly even code.
		
		Now, the isomorphism $C_{i}^{3}\simeq L(V_{i})$ follows directly from Theorem \ref{th2.1} of Classification of Code Loops of rank $3$. First, we just need to calculate the characteristic vector associated to $L(V_{i})$ and note that it corresponds to the same code loop $C_{i}^{3}$ (see [\cite{AR}, Proposition 3.2]).
		
		As example, we calculate the characteristic vector associate to \linebreak $L(V_{5})=\{1,-1\}\times V_{5}$. We have $ v_{i}^{2}=(-1)^{\frac{|v_{i}|}{4}}=-1$ and $ \left[v_{i},v_{j}\right]=(-1)^{\frac{|v_{i}\cap v_{j}|}{2}}=1,$ for $i,j=1,2,3$. Hence $\lambda(L(V_{5}))=(111000)$. Since this vectors produces the code loop $C_{5}^{3}$, by Theorem \ref{th2.1} we conclude that $C_{5}^{3}\simeq L(V_{5}).$
		
		We are going to demonstrate now that each $V_{i}$, $i=1,\dots,5$, up to isomorphism, is the unique minimal representation of the code loop $C_{i}^{3}$. We consider $X=\{a,b,c\}$ a set of generators for $C_{i}^{3}$ such that $\lambda = \lambda_{X}(C_{i}^{3})$ is the corresponding characteristic vector, and we assume that $V$ is a minimal representation of $C_{i}^{3}$, where $v,w,u$ are the elements of the basis of $V$ which corresponds to $a,b,c$. We use the notation $t=|v \cap w\cap u|$. Remember that $t \equiv 1\;(\mbox{mod}\;2)$, because $(a,b,c)=-1$.

		{\itshape Case} $i=1$: In this case suppose that $\mbox{deg}(V) \leq 7$. The characteristic vector is $\lambda = (111111)$, then $|v|\equiv |w|\equiv |u|\equiv 4\; (\mbox{mod}\;8)$. We can suppose that $v=(1,2,3,4)$, but $[a,b]=-1$, then $|v \cap w| \equiv 2 \;(\mbox{mod}\;4)$, and hence, $|v\cap w|=2$. Analogously, we obtain $|v\cap u|=2$. Let $w=(1,2,5,6)$, hence we also obtain $|w\cap u|=2$ and then, $t=1$. Hence, $u=(1,3,5,7)$. Therefore, $V=V_{1}$.

		{\itshape Case} $i=2$: The characteristic vector is $\lambda = (000000)$, then $|v|\equiv |w|\equiv |u|\equiv 0\; (\mbox{mod}\;8)$. Since $\mbox{deg}(V) \leq 13$, then $|v|=8$. Analogously, $|w|=|u|=8$. Let $v=(1-8)$. As $|v \cap w| \equiv 0 \;(\mbox{mod}\;4)$, hence $|v\cap w|=4$. Analogously, $|v\cap u|=|w\cap u|=4$.
		
		Let $w=(1-4,9-12)$. We have two possibilities for $t$: $t=1$ or $t=~3$. Case $t=3$, we will have $\mbox{deg}(V) = 15$, a contradiction. Hence, $t=1$ and $u=(1,5-7,9-11,13)$. Therefore, $V=V_{2}$.
		
		{\itshape Case} $i=3$: We consider $\lambda = (000111)$ and $\mbox{deg}(V) \leq 11$. Analogously to the previous case, we have $|v|=|w|=|u|=8$. In this case, we have $|v\cap w|=|v\cap u|=|w\cap u|=6$. We suppose $v=(1-8)$ and $w=(1-6,9,10)$. We have three possibilities for $t$: $1,3$ or $5$. Case $t\leq 3$ we will have $|v|> 8$, a contradiction. Hence, $t=5$ and we can assume $u=(1-5,7,9,11)$. Therefore, $V=V_{3}$.
		
		{\itshape Case} $i=4$: In this case, consider the caracteristic vector $\lambda = (111110)$. Since $|v|\equiv |w|\equiv |u|\equiv 4 \;(\mbox{mod}\;8)$ and $|w\cap u|\equiv 0 \;(\mbox{mod}\; 4)$ then $|w|=|u|=12$ and $|w\cap u|=8$. Let $v=(1,2,3,4)$, so $|v\cap w|= |v\cap u| = 2$ and $t=1$. Then we can assume $w=(1,2,5-14)$ and  $u=(1,3,5-11,15-17)$. Therefore, $V = V_{4}$. 
		
		Now, we suppose $v=(1-12)$. If $|v\cap w| \leq 6$ then $\mbox{deg}(V) \geq 18.$ In fact, we have $|v\cap w|=2,6$ or $10$. If $|v\cap w|=2$ we will have $|v+w|=20$, a contradiction. If $|v\cap w|=6$ we will have $|v+w|=12$ and hence, $\mbox{deg}(V) \;\geq 18$, a contradiction. Therefore, $|v\cap w| =10$. Analogously we obtain $|v\cap u|=10$. Without loss of generality, we suppose $w=(1,2,5-14)$. Since $t\equiv 1\;(\mbox{mod}\;2)$, then the possibilities for $t$ are $1,3$ or $7$. In any case we will have $|v| \geq 13,$ which is a contradiction. Then, there is not this last possibility for $v$.

		{\itshape Case} $i=5$: Consider the characteristic vector given by $\lambda = (111000)$, so we have $|v|\equiv |w|\equiv |u|\equiv 4 \;(\mbox{mod}\;8)$.  Assuming $v=(1,2,3,4)$ and since $[a,b]=1$, we will have $|v \cap w| \equiv 0 \;(\mbox{mod}\;4)$. Case $v\cap w \neq \emptyset$ we will have $|v\cap w|=4$, and hence $v \subset w$, which is a contradiction, because this give us $t \equiv 0 \;(\mbox{mod}\;2)$, which does not occur. The case $v\cap w = \emptyset$ also does not occur, since we must have nonempty intersection between $v, w$ and $u$.
		
		Therefore $|v|\geq 12$. Analogously, we prove that $|w|\geq 12$ and $|u|\geq 12$.
		
		The representation $V$ is minimal and $\mbox{deg}(V) \leq 17$, so we obtain $|v|=|w|=|u|=12$. Without loss of generality, we suppose $v=(1-12)$ and $|w\cap v| = 4$ or $8$.  If $|v\cap w|=4$, then $|v+w|=16$ and hence, $\mbox{deg}(V) \geq |v\cap w|+|v+w|=20,$ which is a contradiction.
		
		Therefore, $v\cap w = (1-8)$ and $w=(1-8,13-16).$ Analogously, we have $|v\cap u|=|w\cap u|=8.$
		
		If $t\leq 3$, then $|v|\geq |v\cap w\cap u|+|(v\cap w)\backslash
		(v\cap w\cap u)|+|(v\cap u)\backslash
		(v\cap w\cap u)|\geq 3+5+5=13$, a contradiction. Therefore, $t\geq 5.$
		
		If $t=7$, then $\mbox{deg}(V) \geq 19$. To prove this, we consider $u=(i_{1},\dots,i_{12})$, where $i_{1},\dots,i_{7} \in v\cap w\cap u$, but since $v\cap w = (1-8)$ then $i_{1},\dots,i_{7} \in v\cap w$. Suppose $u=(1-7,i_{8},\dots,i_{12})$. Since $|u\cap v|=8$ and $|u\cap w|=8$, so $i_{8} \in \left\{9,10,11,12\right\}$ whereas $i_{9} \in \left\{13,14,15,16\right\}$. We choose $i_{8}=10$ and $i_{9}=13.$ Hence $i_{j} \notin \left\{9,11,12,14,15,16\right\}$, $j=10,11,12$. Hence, a possibility for $u$ it will be $u=(1-7,10,13,17-19)$ so that $\mbox{deg}(V) \geq 19$, a contradiction with the fact that $V$ is a minimal representation. Thus, $t=5$ and  $u=~(1-5,9-11,13-15,17)$. Therefore, $V=V_{5}$.
		
	\end{proof}
	%
	%
	
	%%%%%%%%%%%%%%%%%%%%

	\begin{corollary}
		Each minimal representation of the code loops $C_{1}^{3},\dots,C_{5}^{3}$ has the following types, respectively:
		$$(1111111),(1111333),(1111115),(1111337),(1113335).$$
	\end{corollary}
	
	In order to present the Representation Theorem for nonassociative code loops of rank $4$, note that, according to Theorem \ref{th2.3}, we have exactly $16$ code loops of rank $4$, namely, $C_{1}^{4},C_{2}^{4},\dots,C_{16}^{4}.$ For each $C_{i}^{4}$, $i=1,\dots,16,$ we have to find $V_{i} \subseteq {\mathbb{F}_{2}^{m}}$ doubly even code of minimal degree $m$ such that $ C_{i}^{4}\cong L(V_{i}).$

	\begin{theorem}\label{theorem3.6}
		The code loops $C_{1}^{4},\dots,C_{16}^{4}$ have the following  minimal representations $V_{1},\dots,V_{16}$, which are given by:
		
		\vspace{0.05in}
		\noindent
		$V_{1}=\left\langle (1,2,3,4),(1,2,5,6),(1,3,5,7),(1-8)\right\rangle,$\\
		$V_{2}=\left\langle (1-8),(1-4,9-12),(1,5,6,7,9,10,11,13),(1,2,5,8,9,12,13,14)\right\rangle,$\\
		$V_{3}=\left\langle (1-8),(1-6,9,10),(1,2,3,4,5,7,9,11),(1,6-12)\right\rangle,$\\
		$V_{4}=\left\langle (1-8),(1-6,9,10),(1,2,3,7,9,11-17),(1,4,7,8,9,10,11,18)\right\rangle,$\\
		$V_{5}=\left\langle (1-8),(1,2,3,4,9,10,11,12),(1,5,9,13-17),(1,2,5,6,9,10,13,18)\right\rangle,$\\
		$V_{6}=\left\langle (1,2,3,4),(1,2,5,6),(1,3,5,7),(8,9,10,11)\right\rangle,$\\
		$V_{7}=\left\langle (1-8),(1,2,3,4,9,10,11,12),(1,2,3,5,9,13,14,15),(1-16)\right\rangle,$\\
		$V_{8}=\left\langle (1-8),(1,2,3,4,9-12),(1,2,3,5,9,13,14,15),(1,2,9,10,13,14,16,17)\right\rangle,$\\
		$V_{9}=\left\langle (1-8),(1,2,3,4,9-16),(1,5,6,7,9,10,11,17),(5,6,9,10,12,13,18,19)\right\rangle,$\\
		$V_{10}=\left\langle (1-8),(1,2,9-14),(1,3,9,10,11,15,16,17),(1-6,9-18)\right\rangle,$\\
		$V_{11}=\left\langle (1-8),(1,2,3,4,9-12),(1,2,3,5,9,13,14,15),(6,7,16,17)\right\rangle,$\\
		$V_{12}=\left\langle (1-8),(1,2,3,4,9-12),(1,2,3,5,9,10,11,13),(1,2,9,10,14,15,16,17)\right\rangle,$\\
		$V_{13}=\left\langle (1-8),(1,2,9,10),(1,3,9,11),(4,5,12,13,14,15,16,17)\right\rangle,$\\
		$V_{14}=\left\langle (1-8),(1,2,3,4,9-12),(1,2,3,5,9,10,11,13),(1,2,9,10)\right\rangle,$\\
		$V_{15}=\left\langle (1-12),(1,2,3,4,13-16),(1,2,3,5,13,14,15,17),(1,2,13,14)\right\rangle,$\\
		$V_{16}=\left\langle (1-8),(1,2,9,10,11,12,13,14),(1,3,9,10,11,12,13,15),(4,5,16,17)\right\rangle.$
	\end{theorem}
	
	%%%%%%%%%%%%%%%DEMONSTRATION THEOREM CODE LOOP OF RANK 4
	\begin{proof}
		Each $V_{i}$, $i=1,\dots,16$, is clearly a doubly even binary code. Now, to prove that $C_{i}^{4}\simeq L(V_{i})$ we just need to find the characteristic vector associated to $L(V_{i})$ and apply the Theorem \ref{th2.3} (Classification of code loop of rank 4). Therefore, $V_{i}$ is a representation of  $C_{i}^{4}$, $i=1,\dots,16$.
		
		Now, we are going to prove, up to isomorphism, that $V_{i}$ is the unique minimal representation of $C_{i}^{4}$, $i=1,\dots,16$.  We suppose that $V=\left\langle v_{1},v_{2},v_{3},v_{4}\right\rangle $ is a minimal representation of $C_{i}^{4}$. To determine $V$ we calculate $t_i$, $t_{ij}$, $t_{ijk}$ and $t=t_{1234}$ using the following properties:
		\begin{equation}\label{properties4.1}\left. 
			\begin{array}{l}
				t\leq t_{ij4}\leq t_{ij}-t_{123}+t \,\, (i,j=1,2,3;\,\, i \neq j)\\
				t_{124}+t_{134}-t \leq t_{14} \leq t_{1}-t_{12}-t_{13}+t_{123}+t_{124}+t_{134}-t\\
				t_{124}+t_{234}-t \leq t_{24} \leq t_{2}-t_{12}-t_{23}+t_{123}+t_{124}+t_{234}-t\\
				t_{134}+t_{234}-t \leq t_{34} \leq t_{3}-t_{13}-t_{23}+t_{123}+t_{134}+t_{234}-t\\
				t_4 \geq t_{14}+t_{24}+t_{34}-t_{124}-t_{134}-t_{234}+t\\
			\end{array}
			\right.
		\end{equation}
		Let's prove the unicity of each $V_i$, for all $i=1,\cdots,16$, up to isomorphism.	 For a better visualization of these representations, see the diagrams presented in the \ref{diagrams}.
		
		%\subsection*{Unicity of $V_{1}$}
		% \, - \, $\lambda = (1110110100)$}
	%\noindent
	\vspace{0.1cm}	 	 	
	{\itshape Case} $i=1$:	In this case, the characteristic vector is $\lambda_1 = (1110110100)$.
	Since $\mbox{deg}(V) \leq 8$,  $t_1\equiv  t_2\equiv t_3\equiv 4\;(\mbox{mod}\;8)$ and $t_{12}\equiv  t_{13}\equiv t_{23}\equiv 2 \;(\mbox{mod}\;4)$, then $t_1 = t_2 = t_3 =4$, $t_{12}=t_{13}=t_{23}=2$ and $t_{123}=1$. So we can assume $v_{1}=(1,2,3,4)$, $v_{2}=(1,2,5,6)$ and $v_{3}=(1,3,5,7)$.
	
	Next, we find $v_4$ using the fact that $t_4\equiv 0\;(\mbox{mod}\;8)$ and	$t_{14}\equiv  t_{24}\equiv t_{34}\equiv 0 \;(\mbox{mod}\;4)$. It is necessary to analyse two possible values for $t$: $0$ and $1$.

	Case $t=0$, we have $t_{ij4}=0$ and thus, $t_{i4}=0$, for $i,j=1,2,3$ and $i\neq j$, according to Properties (\ref{properties4.1}). It means that $\left\lbrace 1,\dots,7\right\rbrace \cap v_4 = \emptyset$. Therefore,  $\mbox{deg}(V) > 7+t_4\geq 15$ and then, we don't have minimal reduced representation in this case. Case $t=1$, we have $t_{ij4}=2$ and thus, $t_{i4}=4$, that is, $v_{i}\subset v_{4}$, $i=1,2,3$. Then, since $V$ is minimal, we obtain $t_4 = 8$ and we can assume $v_{4}=(1-8)$. Therefore $V=V_{1}$.

	%	  \subsection*{Unicity of $V_{2}$}
	%\, - \, $\lambda = (0000000000)$}

%\indent 
\vspace{0.1cm}
{\itshape Case} $i=2$: In this case, the characteristic vector is $\lambda_2 = (0000000000)$, so  $t_i \equiv 0\;(\mbox{mod}\;8)$ and $t_{ij}\equiv 0 \;(\mbox{mod}\;4)$ ($i\neq j$). Since $\mbox{deg}(V) \leq 14$, then $t_i = 8$ and $t_{ij}=4$; $i,j = 1,2,3$. Indeed, it's clear that $t_1 = 8$ and $t_{12}=t_{13}=4$. If  we assume that $t_2 = 16$ or $t_3 = 16$, we will have $l_{12}=|v_{1}\cup v_{2} | = 20 > \mbox{deg}(V)$ and $l_{13}=|v_{1}\cup v_{3} | = 20 > \mbox{deg}(V)$, which contradicts the minimality of the degree of $V$. Thus $t_2 = t_3 = 8$ and $t_{23}=4$. We also obtain that  $t_{123}=1$, since if $t_{123}=3$ gives us $l_{123}=|v_1 \cup v_2 \cup v_3|=15 > \mbox{deg}(V)$.

Therefore: $v_1 = (1-8)$, $v_2 = (1,2,3,4,9,10,11,12)$ and $v_3 = (1,5,6,7,9,10,11,13).$

To find $v_4$, we notice that $l_{123}=13$ and $\{1,\cdots,13\}\not\subseteq v_4$, because otherwise, $t_4 \geq 16$ and $\mbox{deg}(V)>14$. Thus $t_4 = 8$ and $v_4$ has 7 elements in the set $\{1,\cdots,13\}$ and a new element, that is, $v_4 = (i_1, \cdots, i_7, i_8 = 14), i_{j}\in \{1,\cdots, 13\}, j\in\{1,\cdots,7\}$. Furthermore, we need to find the weights $t, t_{i4}$ and $t_{ij4}$ that satisfy the following equation:
\begin{eqnarray}\label{eq4.2}
	t_{14} + t_{24} + t_{34} - t_{124} - t_{134} - t_{234} + t + 1 = 8.
\end{eqnarray}

We have two possibilities for $t$: $0$ or $1$. Accordingly with Properties (\ref{properties4.1}) and the minimality of $V$, we conclude that there is a unique solution to the Equation (\ref{eq4.2}):
$$t=1, \,\,t_{ij4} = 2\,\, \mbox{and}\,\, t_{i4} = 4.$$

This way we get $v_4 = (1,2,5,8,9,12,13,14)$ and $V=V_2$.

% \subsection*{Unicity of $V_{3}$}
%\, - \, $\lambda = (0000110100)$}

%\indent 
\vspace{0.1cm}
{\itshape Case} $i=3$: $\lambda_3 = (0000110100)$; $t_1\equiv  t_2\equiv t_3\equiv t_4 \equiv 0\;(\mbox{mod}\;8)$, $t_{12}\equiv  t_{13}\equiv t_{23}\equiv 2 \;(\mbox{mod}\;4)$ and $t_{14}\equiv  t_{24}\equiv t_{34}\equiv 0 \;(\mbox{mod}\;4)$.

Since $\mbox{deg}(V)\leq 12$, we have that $t_1 =8$, $t_{12} = 2$ or $6$ and $t_{13} = 2$ or $6$. We assume that $v_1 = (1-8)$. Case $t_{12} = 2$, we could consider, for instance, $v_2 = (1,2,9,10,11,12,13,14)$, but this would give us $\mbox{deg}(V)\geq 14$. Then, $t_{12} = 6$. As we want to find $V$ with the smallest possible degree, we have that $t_2 = 8$ e we can assume $v_2 = (1,2,3,4,5,6,9,10)$. In a similar way, we get $t_{13} = 6$, $t_3 = 8$ and $t_{23} = 6$.

Using the fact that $t_1 > t_{12} + t_{13} - t_{123}$, we obtain that $t_{123} = 5.$ We consider $v_3 = (1,2,3,4,5,7,9,11)$. So $v_1 \cup v_2 \cup v_3 = \{1,\cdots,11\}$.

To find $v_4$, first we note that $\{1,\cdots,11\} \not\subseteq v_4$, $\{1,\cdots,11\} \cap v_4 \neq \emptyset$ and $t_4 = 8$, because otherwise we would have $\mbox{deg}(V)> 12$. This way, we only have two possibilities for $v_4$: 
\begin{enumerate}
\item $v_4 = (i_1, \cdots,i_7, i_8 = 12)$, with $i_1, \cdots, i_7 \in \{1,\cdots,11\} $;
\item $v_4 = (i_1, \cdots, i_8)$, with $i_1, \cdots, i_8 \in \{1,\cdots,11\} $. 

\end{enumerate}

The next step is to determine the weights $t, t_{ij4}$ and $t_{i4}$ $(i,j = 1,2, 3)$ that satisfy the Properties (\ref{properties4.1}), the minimality of $V$ and one of the following equations: 
\begin{eqnarray}\label{equation4.3}
t_{14}+t_{24}+t_{34}-t_{124}-t_{134}-t_{234}+t + 1 = 8 \\ 
t_{14}+t_{24}+t_{34}-t_{124}-t_{134}-t_{234}+t = 8	\label{eq4.4}
\end{eqnarray}

We have six possibilities for $t$: $0, 1, 2, 3, 4$ or $5$. According to Table \ref{tab.V3}, we do not have a representation in cases $t= 0, 2, 3$ and $5$. For $t=0$, 	the Equations (\ref{equation4.3}) and (\ref{eq4.4}) has no solutions. For $t=2, 3$ and $t=5$, we have a contradiction related to the weigths $t_{i4}$. For $t=4$, there is a non-minimal representation and for $t=1$, we obtain a solution to the Equation (\ref{equation4.3}), so we can consider $v_4 = (1,6,7,8,9,10,11,12)$ and, thus, $V=V_3$. 

\vspace{-0.3cm}

%%%%%%% TABLE V3
\begin{table}[H]
\centering
\caption{$V_{3}$.\label{tab.V3}}
{\scalefont{0.8}\begin{tabular}{|l|l|l|}
		\hline
		$t=0$: $0 \leq t_{ij4} \leq 1 \Rightarrow t_{ij4} =0$ & $t=2$: $2 \leq t_{ij4} \leq 3 \Rightarrow t_{ij4} =2$                                                                                                                                                          &$t=4$: $4 \leq t_{ij4} \leq 5 \Rightarrow t_{ij4} =4$    \\ \hline \hline
		\begin{tabular}[c]{@{}l@{}} $0\leq t_{i4} \leq 1 \Rightarrow t_{i4} = 0$\end{tabular} & \begin{tabular}[c]{@{}l@{}} 		$2\leq t_{i4} \leq 3$ \end{tabular}                              & \begin{tabular}[c]{@{}l@{}} $4\leq t_{i4} \leq 5 \Rightarrow t_{i4} = 4$\end{tabular}                                                   \\ \hline \hline
		$t=1$: $1 \leq t_{ij4} \leq 2 \Rightarrow t_{ij4} =2$                                                                                                                                                                 & $t=3$: $3 \leq t_{ij4} \leq 4 \Rightarrow t_{ij4} =4$                                                                                                                                                          & $t=5$: $5 \leq t_{ij4} \leq 6 \Rightarrow t_{ij4} =6$                                                                                                                                                                                                                   \\ \hline \hline
		\begin{tabular}[c]{@{}l@{}} $3\leq t_{i4} \leq 5 \Rightarrow t_{i4} = 4$\end{tabular} & \begin{tabular}[c]{@{}l@{}} 		$5\leq t_{i4} \leq 6$ \end{tabular} & \begin{tabular}[c]{@{}l@{}} $7\leq t_{i4} \leq 8 \Rightarrow t_{i4} = 8$ 
			% because $v_i \neq v_4$, $i=1,2,3$)
		\end{tabular} \\ \hline
\end{tabular}}
\end{table}

% \subsection*{Unicity of $V_{4}$}
% \, - \, $\lambda = (0010100000)$}

%\indent 

%\vspace{-0.5cm}
{\itshape Case} $i=4$: $\lambda_4 = (0010100000)$; $t_1\equiv  t_2\equiv t_4\equiv 0\;(\mbox{mod}\;8)$; $t_3\equiv 4\;(\mbox{mod}\;8)$; $t_{12}\equiv 2 \;(\mbox{mod}\;4)$; $t_{13}\equiv t_{14}\equiv  t_{23}\equiv t_{24}\equiv t_{34}\equiv 0 \;(\mbox{mod}\;4)$.

%\begin{description}
%	\item[$\lambda_4 = (0010100000)$:] $t_1\equiv  t_2\equiv t_4\equiv 0\;(\mbox{mod}\;8)$; $t_3\equiv 4\;(\mbox{mod}\;8)$; $t_{12}\equiv 2 \;(\mbox{mod}\;4)$; $t_{13}\equiv t_{14}\equiv  t_{23}\equiv t_{24}\equiv t_{34}\equiv 0 \;(\mbox{mod}\;4)$.
%\end{description}

First, we note that $t_{12} \leq 6$. For instance, if $t_{12} = 10$, we will have $t_{1} \geq 16$ and $t_2 \geq 16$, and so $l_{12}=|v_{1}\cup v_{2}|\geq 22 $  and thus will be greater than the degree of $V$. Since $t_{12} \leq 6$ and $l_{12} \leq \;\mbox{deg}(V)$, so we have $t_1 = t_2 = 8$ and so $t_{13} = t_{23} = 4$ and $t_{123} = 1 \, \mbox{or} \, 3$. We can consider $v_{1} = (1,2,3,4,5,6,7,8)$. 

Case $t_{12} = 2$, we will have $t_{123} = 1$. Since $t_3 > t_{13} + t_{23} - t_{123}$, we will obtain $t_3 \geq 12$. This gives us that $\mbox{deg}(V) > 18$, because $l_{123} = t_3 + 7 \geq 19$. Therefore $t_{12} = 6$ and, we can take $v_2 = (1,2,3,4,5,6,9,10)$. In this case we obtain $t_{123} = 3$, because $t_{12} + t_{13} - t_{123} < 8$. As we have  $t_3 > t_{13} + t_{23} - t_{123} $, that is, $t_3 > 5$, and $l_{123} < \mbox{deg}(V)$, then $t_3 = 12$. We consider $v_3 = (1,2,3,7,9,11,12,13,14,15,16,17)$. 

Now, to determine the generator $v_4$, note que $v_1 \cup v_2 \cup v_3 = \left\lbrace 1,\dots, 17 \right\rbrace \not\subseteq v_4$. So we have two possibilities  for $v_4$: $v_4 = (i_1, \dots, i_{15}, 18)$ or $v_4 = (i_1, \dots, i_7, 18)$, with $i_j \in \left\lbrace 1,\dots, 17\right\rbrace $, $j=1,\dots,15$. We need to find the values $t_{i4}, t_{ij4}$ and $t$ such that: 
\begin{equation}\label{eq4.3}
t_4 = t_{14}+t_{24}+t_{34}-t_{124}-t_{134}-t_{234}+t + 1.
\end{equation}

We have four possibilities for $t$: $0,1,2$ or $3$. Accordingly with Properties (\ref{properties4.1}) and the minimality of $V$, we conclude that there is a unique solution to the Equation (\ref{eq4.3}):
$$t=1, \,\,t_{ij4} = 2,\,\, t_{i4} = 4\,\, \mbox{and}\,\, t_4 = 8.$$

This way we get $v_4 = (1,4,7,8,9,10,11,18)$ and $V=V_4$.

% \subsection*{Unicity of $V_{5}$}
%\, - \, $\lambda = (0000010100)$}

%\indent 
\vspace{0.2cm}
{\itshape Case} $i=5$: $\lambda_5 = (0000010100)$; $t_1\equiv  t_2\equiv t_3\equiv t_4 \equiv 0\;(\mbox{mod}\;8)$; $t_{12}\equiv  t_{14}\equiv t_{24}\equiv t_{34} \equiv 0 \;(\mbox{mod}\;4)$; $t_{13}\equiv  t_{23}\equiv 2 \;(\mbox{mod}\;4)$.

%\begin{description}
%	\item[$\lambda_5 = (0000010100)$:] $t_1\equiv  t_2\equiv t_3\equiv t_4 \equiv 0\;(\mbox{mod}\;8)$; $t_{12}\equiv  t_{14}\equiv t_{24}\equiv t_{34} \equiv 0 \;(\mbox{mod}\;4)$; $t_{13}\equiv  t_{23}\equiv 2 \;(\mbox{mod}\;4)$.
%\end{description}

Since $l_{13} < 18$ and $l_{23} < 18$, we obtain $t_{1} = t_{2} = t_{3} = 8$, $t_{13} \leq 6$ and $t_{23} \leq 6$. If $t_{13} = 10$, for instance, we will have $t_{1} \geq 16$, $t_{3} \geq 16$ and $l_{13} \geq 22 \geq \mbox{deg}(V)$, which contradicts the minimality of the degree of $V$. We also have $t_{12} = 4$ and so, $t_{123} = 1$ or $3$. As $t_{1} > t_{12} + t_{13} - t_{123}$, $t_{2} > t_{12} + t_{23} - t_{123}$ and $t_{3} > t_{13} + t_{23} - t_{123}$, that is, $t_{13} - t_{123} < 4$, $t_{23} - t_{123} < 4$ and $t_{23} + t_{13} < t_{3} + t_{123}$, we obtain $t_{123} = 1$ and $t_{13} = t_{23} =2$. Thus, it can be considered $v_{1} = (1-8)$, $v_{2} = (1,2,3,4,9,10,11,12)$ and $v_{3} = (1,5,9,13,14,15,16, 17)$.

Note that $v_{1}\cup v_{2}\cup v_{3} = \{1,\cdots,17\} \not\subseteq v_{4}$, then $t_{4} \leq 16$. Thus $v_{4} = (i_{1},\cdots,i_{15},18)$ or $v_{4} = (i_{1},\cdots,i_{7},18)$ where $i_{j} \in \{1,\cdots,17\}$. By (\ref{properties4.1}), this means that the weights $t_{i4}, t_{ij4}$ and $t$ need to satisfy one of the following equations:
\begin{eqnarray}
t_{14} + t_{24} + t_{34} - t_{124} - t_{134} - t_{234} + t = 15\\
t_{14} + t_{24} + t_{34} - t_{124} - t_{134} - t_{234} + t = 7
\end{eqnarray}

%%%%%  v4

According to Table (\ref{tab.V5}), there is a solution only in the case where $t_{124}= t_{134} = t_{234} = 2$ and  $t_{14} = t_{24} = t_{34} = 4$. So consider $v_4 = (1,2,5,6,9,10,13,18)$ to get $V=V_5$.

\vspace{-0.4cm}
%%%%TABLE V5
\begin{table}[H]

\centering \caption{$V_{5}$. \label{tab.V5}}
{\scalefont{0.8}\begin{tabular}{|l|l|}
\hline
$t=0$     & $t=1$                                                                                                                                                                                         \\ \hline \hline
\begin{tabular}[c]{@{}l@{}}$0 \leq t_{124} \leq 3 \Rightarrow t_{124} =0$ or $2$\\ 			$0 \leq t_{134} \leq 1 \Rightarrow t_{134} =0$\\ 			$0 \leq t_{234} \leq 1 \Rightarrow t_{234} =0$\end{tabular} & \begin{tabular}[c]{@{}l@{}}$2 \leq t_{124} \leq 4 \Rightarrow t_{124} =2$ or $4$\\ $1 \leq t_{134} \leq 2 \Rightarrow t_{134} =2$\\ $1 \leq t_{234} \leq 2 \Rightarrow t_{234} =2$\end{tabular} \\ \hline
$t_{124}= t_{134} = t_{234} = 0$                                                                                                                                                                       & $t_{124}= t_{134} = t_{234} = 2$                                                                                                                                                                \\ \hline \hline
\begin{tabular}[c]{@{}l@{}}$0 \leq t_{14} \leq 3 \Rightarrow t_{14} =0$\\ $0\leq t_{24} \leq 3 \Rightarrow t_{24} = 0$\\ $0\leq t_{34} \leq 5 \Rightarrow t_{34} = 0$ or $4$\end{tabular}              & \begin{tabular}[c]{@{}l@{}}$3 \leq t_{14} \leq 6 \Rightarrow t_{14} =4$\\ $3\leq t_{24} \leq 6 \Rightarrow t_{24} = 4$\\ $3\leq t_{34} \leq 8 \Rightarrow t_{34} = 4$ or $8$\end{tabular}       \\ \hline
$t_{124}= 2$; $t_{134} = t_{234} = 0$                                                                                                                                                               & $t_{124}= 4$; $t_{134} = t_{234} = 2$                                                                                                                                                        \\ \hline \hline
\begin{tabular}[c]{@{}l@{}}$2 \leq t_{14} \leq 5 \Rightarrow t_{14} =4$\\ $2\leq t_{24} \leq 5 \Rightarrow t_{24} = 4$\\ $0\leq t_{34} \leq 5 \Rightarrow t_{34} = 0$ or $4$\end{tabular}              & \begin{tabular}[c]{@{}l@{}}$5 \leq t_{14} \leq 8 \Rightarrow t_{14} =8$\\ $5\leq t_{24} \leq 8 \Rightarrow t_{24} = 8$\\ $3\leq t_{34} \leq 8 \Rightarrow t_{34} = 4$ or $8$\end{tabular}       \\ \hline 
\end{tabular}}
\end{table}

%%%%%%

% \subsection*{Unicity of $V_{6}$}
%\, - \, $\lambda = (1111110100)$}

%\indent 
%\vspace{-0.4cm}
{\itshape Case} $i=6$: $\lambda_6 = (1111110100)$; $t_1\equiv  t_2\equiv t_3\equiv t_4 \equiv 4\;(\mbox{mod}\;8)$; $t_{12}\equiv  t_{13}\equiv t_{23}\equiv 2 \;(\mbox{mod}\;4)$; $t_{14}\equiv  t_{24}\equiv t_{34}\equiv 0 \;(\mbox{mod}\;4)$.

%	 	\begin{description}
%	 		\item[$\lambda_6 = (1111110100)$:] $t_1\equiv  t_2\equiv t_3\equiv t_4 \equiv 4\;(\mbox{mod}\;8)$; $t_{12}\equiv  t_{13}\equiv t_{23}\equiv 2 \;(\mbox{mod}\;4)$; $t_{14}\equiv  t_{24}\equiv t_{34}\equiv 0 \;(\mbox{mod}\;4)$.
%	 	\end{description}
%	 	
As $\mbox{deg}(V)\leq 11$, then $t_{1} = 4$, which implies $t_{12} = t_{13} = 2$ and $t_{123} =1$. Furthermore, the fact that $l_{12}<11$ and $l_{13} < 11$ gives us	$t_{2} = t_{3} = 4$ and $t_{23} =2$. Thus, consider $v_{1} = (1,2,3,4)$, $v_{2} = (1,2,5,6)$ and $v_{3} = (1,3,5,7)$. To find $v_{4}$, note that the minimality of the degree of V implies that $\{1,\cdots,7\} \not\subseteq v_{4}$, and so, $t_{4} = 4$. Thus, $t_{i4} = 0$, and this way, $t = t_{ij4} = 0$. Therefore, $\{1,\cdots,7\} \cap v_{4} = \emptyset$, which implies $v_{4} = (8,9,10,11)$ and $V=V_{6}$.

%%%%%  v4

%	 \subsection*{Unicity of $V_{7}$}
%\, - \, $\lambda = (0001000000)$}

%\indent 
\vspace{0.2cm}
{\itshape Case} $i=7$: $\lambda_7 = (0001000000)$; $t_1\equiv  t_2\equiv t_3 \equiv 0\;(\mbox{mod}\;8)$; $t_4 \equiv 4\;(\mbox{mod}\;8)$; $t_{12}\equiv  t_{13}\equiv 0 \;(\mbox{mod}\;4)$; $t_{23} \equiv t_{14}\equiv  t_{24}\equiv t_{34}\equiv 0 \;(\mbox{mod}\;4)$.

%\begin{description}
%	\item[$\lambda_7 = (0001000000)$:] $t_1\equiv  t_2\equiv t_3 \equiv 0\;(\mbox{mod}\;8)$; $t_4 \equiv 4\;(\mbox{mod}\;8)$; $t_{12}\equiv  t_{13}\equiv 0 \;(\mbox{mod}\;4)$; $t_{23} \equiv t_{14}\equiv  t_{24}\equiv t_{34}\equiv 0 \;(\mbox{mod}\;4)$.
%\end{description}

As $l_{ij}  < 16$, we have $t_{i} = 8$ and $t_{ij} = 4$, for $i,j = 1,2,3$. If, for instance, $t_{1} = t_{2} = 16$ and $t_{12} = 12$, we will have $l_{12} = 20$. Suppose $v_{1} = (1-8)$ and $v_{2} = (,2,3,4,9,10,11,12)$. 

Note that there are two possibitilies for $t_{123}$: $1$ or $3$. Both cases were analyzed and the results are shown in the Table \ref{tab.V7}.

We assume $v_3 = (1,5,6,7,9,10,11,13)$ in the first case and  $v_3 = (1,2,3,5,9,13,14,15)$ in the second case. By item (1.1.1.), we have $\{1,\cdots,13\}\cap v_4 = \emptyset$. If $t_4 = 4 $ and $v_4 = (14,15,16,17)$, for instance, we would obtain a representation with a degree greater than 16. In items (1.1.2.), (1.2.1.), (1.2.2.), (2.1.), (2.2.) and (2.3.) we also get non-minimal representations. In items (1.1.3.), (1.2.3.) and cases similar to these, we have a contradiction  with the hypothesis $t_{14} \equiv 0\;(\mbox{mod}\;4)$. Finally, the minimal representation is determined in the last case because  $t_{14} + t_{24} + t_{34} - t_{124} - t_{134} - t_{234} + t = 15$. Then, if $t_4 = 16$, we can consider $v_4 = (1,\cdots,16)$ and therefore, $V=V_7$.

%%%%Table V7
\vspace{-0.1in}
\begin{table}[H]

\centering \caption{$V_{7}$. \label{tab.V7}}
{\scalefont{0.7}\begin{tabular}{|ll|}
\hline
\multicolumn{2}{|l|}{\textbf{Case 1:} $t_{123} = 1$}                                                     \\ \hline \hline
\multicolumn{1}{|l|}{1.1. $t=0$: $0 \leq t_{ij4} \leq 3 \Rightarrow t_{ij4} =2$ or $4$ }                                                                                                                                                           & 1.2. $t=1$: $1 \leq t_{ij4} \leq 4 \Rightarrow t_{ij4} =2$ or $4$                                                                                                                                                                                            \\ \hline
%			\multicolumn{1}{|l|}{$0 \leq t_{ij4} \leq 3 \Rightarrow t_{ij4} =2$ or $4$}                                                                                                                  & $1 \leq t_{ij4} \leq 4 \Rightarrow t_{ij4} =2$ or $4$                                                                                                                                                  \\ \hline
\multicolumn{1}{|l|}{\begin{tabular}[c]{@{}l@{}}1.1.1. $t_{124}= t_{134} = t_{234} = 0$\\ $\Rightarrow 0 \leq t_{i4} \leq 1 \Rightarrow t_{i4} =0$\end{tabular}}                             & \begin{tabular}[c]{@{}l@{}}1.2.1. $t_{124}= t_{134} = t_{234} = 2$\\ $\Rightarrow 3 \leq t_{i4} \leq 4 \Rightarrow t_{i4} =4$ \\ \end{tabular} \\ \hline
\multicolumn{1}{|l|}{\begin{tabular}[c]{@{}l@{}}1.1.2. $t_{124}= t_{134} = t_{234} = 2$\\ $\Rightarrow 4 \leq t_{i4} \leq 5 \Rightarrow t_{i4} =4$\end{tabular}} & \begin{tabular}[c]{@{}l@{}}1.2.2. $t_{124}= t_{134} = t_{234} = 4$\\ $\Rightarrow 7 \leq t_{i4} \leq 8 \Rightarrow t_{i4} =8$\\ \end{tabular}  \\ \hline
\multicolumn{1}{|l|}{\begin{tabular}[c]{@{}l@{}}1.1.3. $t_{124}= 2$; $t_{134} = t_{234} = 0$\\ $\Rightarrow 2 \leq t_{14} \leq 3$\end{tabular}}                                           & \begin{tabular}[c]{@{}l@{}}1.2.3. $t_{124}= 4$; $t_{134} = t_{234} = 2$\\ $\Rightarrow 5 \leq t_{14} \leq 6$\end{tabular}                                                                           \\ \hline \hline
\multicolumn{2}{|l|}{\textbf{Case 2:} $t_{123} = 3$}                                                \\ \hline \hline
\multicolumn{1}{|l|}{2.1. $t=0$: $0 \leq t_{ij4} \leq 1 \Rightarrow t_{ij4} =0$}                                                                                                                                                             & 2.3. $t=2$: $2 \leq t_{ij4} \leq 3 \Rightarrow t_{ij4} =2$                                                                                                                                                                                             \\ \hline
\multicolumn{1}{|l|}{\begin{tabular}[c]{@{}l@{}} $0\leq t_{i4} \leq 3 \Rightarrow t_{i4} = 0$\end{tabular}}                                  & \begin{tabular}[c]{@{}l@{}} $2\leq t_{i4} \leq 5 \Rightarrow t_{i4} = 4$\end{tabular}                                                                  \\ \hline
\multicolumn{1}{|l|}{2.2. $t=1$: $1 \leq t_{ij4} \leq 2 \Rightarrow t_{ij4} =2$}                                                                                                                                                             & 2.4. $t=3$: $3 \leq t_{ij4} \leq 4 \Rightarrow t_{ij4} =4$                                                                                                                                                                                             \\ \hline
\multicolumn{1}{|l|}{\begin{tabular}[c]{@{}l@{}} $3\leq t_{i4} \leq 6 \Rightarrow t_{i4} = 4$\end{tabular}}                                  & \begin{tabular}[c]{@{}l@{}} $5\leq t_{i4} \leq 8 \Rightarrow t_{i4} = 8$\end{tabular}                                                                  \\ \hline
\end{tabular}
}
\end{table}

%\vspace{-0.7cm}	

%	 \subsection*{Unicity of $V_{8}$}
%\, - \, $\lambda = (0000001000)$}

%	\indent 
%\vspace{0.2cm}
%\newpage
{\itshape Case} $i=8$: $\lambda_8 = (0000001000)$;  $t_1\equiv  t_2\equiv t_3\equiv t_4 \equiv 0\;(\mbox{mod}\;8)$; $t_{12}\equiv  t_{13}\equiv t_{23}\equiv 0 \;(\mbox{mod}\;4)$; $t_{24}\equiv  t_{34}\equiv 0 \;(\mbox{mod}\;4)$; $t_{14}\equiv 2 \;(\mbox{mod}\;4)$.

%	\begin{description}
%		\item[$\lambda_8 = (0000001000)$:] $t_1\equiv  t_2\equiv t_3\equiv t_4 \equiv 0\;(\mbox{mod}\;8)$; $t_{12}\equiv  t_{13}\equiv t_{23}\equiv 0 \;(\mbox{mod}\;4)$; $t_{24}\equiv  t_{34}\equiv 0 \;(\mbox{mod}\;4)$; $t_{14}\equiv 2 \;(\mbox{mod}\;4)$.
%	\end{description}

%	
As $l_{ij} < 17$, we need to have $t_{i}=8$ and $t_{ij}=4$, for $i,j = 1,2,3$. Assume $v_{1} = (1-8)$ and $v_{2} = (1,2,3,4,9,10,11,12)$. There two cases to analyse (see Table \ref{tab.V8}).

In case 1,  we assume that $v_3 = (1,5,6,7,9,10,11,13)$. In (1.1.1.), if 	$0 \leq t_{i4} \leq 1$ then $t_{i4} =0$ and thus, $\{1,\cdots,13\}\cap v_4 = \emptyset$. Soon $t_4 \geq 8$ and $\mbox{deg}(V) \geq 21$, that is, we will obtain a non-minimal representation. In other cases, we find contradictions related to the weights $t_{i4}$. The other cases are analyzed in a similar way. On the other hand, in case 2, we assume $v_3 = (1,2,3,5,9,13,14,15)$.
If $t=0$, we will have $t_{14} + t_{24} + t_{34} - t_{124} - t_{134} - t_{234} + t = 2$. Then, if $t_4 = 8$ and if we consider $v_4 = (i_1, i_2, 16,\cdots,21)$, with $i_1, i_2 \in \{1,\cdots,15\}$ and $i_1, i_2 \in v_1 \cap v_4$, we will have a non-minimal representation.
If $t=1$ and $t=3$, we will also find  non-minimal representations with degree equal to $21$ and $18$, respectively.  If $t=2$, then $t_4 \geq t_{14} + t_{24} + t_{34} - t_{124} - t_{134} - t_{234} + t = 6$. If we assume $t_4 = 8$ and $v_4 = (1,2,9,10,13,14,16,17)$, we obtain $V= V_8$.

%%%%

%%%%Table V8	

%%
\vspace{-0.3cm}
\begin{table}[H]
\centering
\caption{$V_{8}$. \label{tab.V8}}
{\scalefont{0.8}
\begin{tabular}{|ll|}
\hline 
\multicolumn{2}{|l|}{\textbf{Case 1:} $t_{123} = 1$}                                                                                                                                                                 \\ \hline \hline
\multicolumn{1}{|l|}{1.1. $t=0$: $0 \leq t_{ij4} \leq 3 \Rightarrow t_{ij4} =0$ or $2$}                                  & 1.2. $t=1$: $1 \leq t_{124} \leq 4 \Rightarrow t_{124} =2$ or $4$                \\ \hline
\multicolumn{1}{|l|}{1.1.1. $t_{124}= t_{134} = t_{234} = 0$ $\Rightarrow$ $0 \leq t_{i4} \leq 1 \Rightarrow t_{i4} =0$} & 1.2.1. $t_{124}= t_{134} = t_{234} = 2$ $\Rightarrow$ $3 \leq t_{14} \leq 4$     \\ \hline
\multicolumn{1}{|l|}{1.1.2. $t_{124}= 2$; $t_{134} = t_{234} = 0$ $\Rightarrow$ $2\leq t_{24} \leq 3$}                   & 1.2.2. $t_{124}= 2$; $t_{134} = t_{234} = 4$ $\Rightarrow$ $5\leq t_{24} \leq 6$ \\ \hline
\multicolumn{1}{|l|}{1.1.3. $t_{124}= t_{134} = t_{234} = 2$ $\Rightarrow$ $4\leq t_{14} \leq 5$}                        & 1.2.3. $t_{124} = t_{134} = t_{234} = 4$ $\Rightarrow$ $7\leq t_{14} \leq 8$     \\ \hline \hline
\multicolumn{2}{|l|}{\textbf{Case 2:} $t_{123} = 3$}                                                                                                                                                                 \\ \hline \hline
\multicolumn{1}{|l|}{2.1. $t=0$: $0 \leq t_{ij4} \leq 1 \Rightarrow t_{ij4} =0$}                                         & 2.3. $t=2$: $2 \leq t_{ij4} \leq 3 \Rightarrow t_{ij4} =2$                       \\ \hline
\multicolumn{1}{|l|}{$0\leq t_{i4} \leq 3 \Rightarrow t_{14} = 2$; $t_{24} = t_{34} = 0$}                                & $2\leq t_{i4} \leq 5 \Rightarrow t_{14} = 2$; $t_{24} = t_{34} = 4$              \\ \hline
\multicolumn{1}{|l|}{2.2. $t=1$: $1 \leq t_{ij4} \leq 2 \Rightarrow t_{ij4} =2$}                                         & 2.4. $t=3$: $3 \leq t_{ij4} \leq 4 \Rightarrow t_{ij4} =4$                       \\ \hline
\multicolumn{1}{|l|}{$3\leq t_{i4} \leq 6 \Rightarrow t_{14} = 6$; $t_{24} = t_{34} = 4$}                                & $5\leq t_{i4} \leq 8 \Rightarrow t_{14} = 6$; $t_{24} = t_{34} = 8$              \\ \hline
\end{tabular}
}
\end{table}

%	 \subsection*{Unicity of $V_{9}$}
%\, - \, $\lambda = (0100001000)$}

%	\indent 

\vspace{-0.1cm}
{\itshape Case} $i=9$: $\lambda_9 = (0100001000)$; $t_1\equiv t_3\equiv t_4 \equiv 0\;(\mbox{mod}\;8)$; $t_2 \equiv 4\;(\mbox{mod}\;8)$; $t_{14}\equiv  2 \;(\mbox{mod}\;4)$; $t_{12}\equiv  t_{13}\equiv t_{23}\equiv t_{24} \equiv t_{34}\equiv 0 \;(\mbox{mod}\;4)$.

%	\begin{description}
%		\item[$\lambda_9 = (0100001000)$:] $t_1\equiv t_3\equiv t_4 \equiv 0\;(\mbox{mod}\;8)$; $t_2 \equiv 4\;(\mbox{mod}\;8)$; $t_{14}\equiv  2 \;(\mbox{mod}\;4)$; $t_{12}\equiv  t_{13}\equiv t_{23}\equiv t_{24} \equiv t_{34}\equiv 0 \;(\mbox{mod}\;4)$.
%	\end{description}

We affirm that $t_{1} = 8$, $t_{2} = 12$ and $t_{12} = 4$. Indeed, first note that $l_{12} < 19$ implies $t_{1} \leq 16$. If $t_{1} = 8$, it is clear that $t_{12} = 4$ and $t_{2} > 4$. Since $l_{12} < 19$, so $t_{2} < 15$, which implies that $t_{2} = 12$. If $t_{1} = 16$, we will have $t_{12} = 4, 8$ or $12$. But, in order to have $l_{12} < 19$, we should have $t_{2} - t_{12} < 3$, which does not happen for any value of  $t_{2}$ and $t_{12}$. Let $v_{1} = (1-8)$ and $v_{2} = (1,2,3,4,9,10,11,12,13,14,15,16)$. Note that $t_{13} = 4$ and $t_{3} = 8$, because $t_{1} = 8$ and $l_{13} < 19$. So, $t_{23} = 4$.

%%%% v4
Now, let's find $v_{3}$ and $v_4$. The possible cases for the remaining weights are shown in Table \ref{tab.V9}. In case 1, we assume $v_3 = (1,5,6,7,9,10,11,17).$ 
For items (1.1.1.), (1.1.2.) - (1.1.7.), (1.2.1.) - (1.2.4.), (1.2.7.) and (1.2.8), we have contradictions related to the weights $t_{i4}$. In items (1.1.3.), (1.2.5.) and (1.2.6.), we obtain representations of degree equal to $21$, $20$ and $22$, respectively.
In item (1.1.8.), since we have $t_{14} + t_{24} + t_{34} - t_{124} - t_{134} - t_{234} + t = 6$,  we can consider  $v_4 = (i_1, \cdots, i_{6}, 18,19)$, with $\{i_1,\cdots, i_{6}\} \subset \{1,\cdots,17\}$, and, therefore, $V=V_9$. In case 2, we can consider $v_3 = (1,2,3,5,9,17,18,19)$. In this case, if there is another representation of degree 19, we must have $v_4 \subset \{1,\cdots, 19\}$. Soon $t_4 = 8$ or $t_4 = 16$, and $t_{14} + t_{24} + t_{34} - t_{124} - t_{134} - t_{234} + t = 8$ or $t_{14} + t_{24} + t_{34} - t_{124} - t_{134} - t_{234} + t = 16$. After analyzing all the possibilities for $t$, we see that these equations have no solution.

\vspace{-0.2cm}
%%%%TABLE V9
%\scalefont{0.6}
\begin{table}[H] %!hbt

\centering \caption{$V_{9}$. \label{tab.V9}}
{\scalefont{0.8}
\begin{tabular}{|ll|}
\hline
\multicolumn{2}{|l|}{\textbf{Case 1:} $t_{123} = 1$}                                                                                                                                                                                                                                                                                                                                                                                                                                                        \\ \hline \hline
\multicolumn{1}{|l|}{1.1. $t=0$: $0 \leq t_{ij4} \leq 3 \Rightarrow t_{ij4} =0$ or $2$}                                                                                                                                                                     & 1.2. $t=1$: $1 \leq t_{ij4} \leq 4 \Rightarrow t_{ij4} =2$ or $4$                                                                                                                                                                    \\ \hline 
\multicolumn{1}{|l|}{1.1.1. $t_{124}= t_{134} = t_{234} = 0 \Rightarrow  0 \leq t_{14} \leq 1$}                                                                                                                                                             & 1.2.1. $t_{124}= t_{134} = t_{234} = 2$ $\Rightarrow 3 \leq t_{14} \leq 4$                                                                                                                                                           \\ \hline
\multicolumn{1}{|l|}{1.1.2. $t_{124}= t_{134} = t_{234} = 2 \Rightarrow 4 \leq t_{14} \leq 5 $}                                                                                                                                                             & 1.2.2. $t_{124}= t_{134} = t_{234} = 4$ $\Rightarrow 7 \leq t_{14} \leq 8$                                                                                                                                                           \\ \hline
\multicolumn{1}{|l|}{\begin{tabular}[c]{@{}l@{}}1.1.3. $t_{124}= 2$; $t_{134} = t_{234} = 0$\\  $\Rightarrow 2 \leq t_{14} \leq 3 \Rightarrow t_{14} =2;$  $2\leq t_{24} \leq 7\Rightarrow t_{24} = 4;$ \\ $0\leq t_{34} \leq 1 \Rightarrow t_{34} = 0$\end{tabular}}  & 1.2.3. $t_{124}= t_{134} = 2$; $t_{234} = 4$ $\Rightarrow 3 \leq t_{14} \leq 4$                                                                                                                                                      \\ \hline
\multicolumn{1}{|l|}{1.1.4. $t_{124}= t_{134} = 2$; $t_{234} =  0$ $\Rightarrow 4 \leq t_{14} \leq 5$}                                                                                                                                                      & 1.2.4. $t_{124}= t_{134} = 4$; $t_{234} = 2$ $\Rightarrow 7 \leq t_{14} \leq 8$                                                                                                                                                      \\ \hline
\multicolumn{1}{|l|}{1.1.5. $t_{124}= t_{234} = 2$; $t_{134} = 0$ $\Rightarrow 2\leq t_{34} \leq 3 $}                                                                                                                                                       & \begin{tabular}[c]{@{}l@{}}1.2.5. $t_{124}= 2$; $t_{134} = t_{234} = 4$ \\ $ \Rightarrow 5 \leq t_{14} \leq 6 \Rightarrow t_{14} =6;$ $5\leq t_{24} \leq 10 \Rightarrow t_{24} = 8;$\\  $7\leq t_{34} \leq 8 \Rightarrow t_{34} = 8$\end{tabular} \\ \hline
\multicolumn{1}{|l|}{1.1.6. $t_{124}= t_{134} = 0$; $t_{234} = 2$ $\Rightarrow 0 \leq t_{14} \leq 1$}                                                                                                                                                       & \begin{tabular}[c]{@{}l@{}}1.2.6. $t_{124}= 4$; $t_{134} = t_{234} = 2$\\  $\Rightarrow 5 \leq t_{14} \leq 6 \Rightarrow t_{14} =6;$ $5\leq t_{24} \leq 10 \Rightarrow t_{24} = 8;$ \\ $3\leq t_{34} \leq 4 \Rightarrow t_{34} = 4$\end{tabular} \\ \hline
\multicolumn{1}{|l|}{1.1.7. $t_{124}= t_{234} = 0$; $t_{134} =  2$ $\Rightarrow 2\leq t_{34} \leq 3$}                                                                                                                                                       & 1.2.7. $t_{124}= t_{234} = 4$; $t_{134} =  2$ $\Rightarrow 7 \leq t_{14} \leq 8$                                                                                                                                                     \\ \hline
\multicolumn{1}{|l|}{\begin{tabular}[c]{@{}l@{}}1.1.8. $t_{124}= 0$; $t_{134} = t_{234} = 2$ \\ $ \Rightarrow  2 \leq t_{14} \leq 3 \Rightarrow t_{14} =2;$ $2\leq t_{24} \leq 5 \Rightarrow t_{24} = 4;$  \\ $4\leq t_{34} \leq 5 \Rightarrow t_{34} = 4$\end{tabular}} & 1.2.8. $t_{124}= t_{234} = 2$; $t_{134} = 4$ $\Rightarrow 5 \leq t_{34} \leq 6$                                                                                                                                                      \\ \hline \hline
\multicolumn{2}{|l|}{\textbf{Case 2:} $t_{123} = 3$}                                                                                                                                                                                                                                                                                                                                                                                                                                                        \\ \hline \hline
\multicolumn{1}{|l|}{2.1. $t=0$: $0 \leq t_{ij4} \leq 1 \Rightarrow t_{ij4} =0$}                                                                                                                                                                            & 2.3. $t=2$: $2 \leq t_{ij4} \leq 3 \Rightarrow t_{ij4} =2$                                                                                                                                                                           \\ \hline 
\multicolumn{1}{|l|}{\begin{tabular}[c]{@{}l@{}}$ 0\leq t_{14} \leq 3 \Rightarrow t_{14} = 2$;  $0\leq t_{34} \leq 3 \Rightarrow t_{34} = 0$;\\ $0\leq t_{24} \leq 7 \Rightarrow t_{24} = 0$ or $4$\end{tabular}}                                           & \begin{tabular}[c]{@{}l@{}}$2\leq t_{14} \leq 5 \Rightarrow t_{14} = 2$; $2\leq t_{34} \leq 5 \Rightarrow t_{34} = 4$;\\ $2\leq t_{24} \leq 9 \Rightarrow t_{24} = 4$ or $8$\end{tabular}                                            \\ \hline
\multicolumn{1}{|l|}{2.2. $t=1$: $1 \leq t_{ij4} \leq 2 \Rightarrow t_{ij4} =2$}                                                                                                                                                                            & 2.4. $t=3$: $3 \leq t_{ij4} \leq 4 \Rightarrow t_{ij4} =4$                                                                                                                                                                           \\ \hline
\multicolumn{1}{|l|}{\begin{tabular}[c]{@{}l@{}}$3\leq t_{14} \leq 6 \Rightarrow t_{14} = 6$; $3\leq t_{34} \leq 6 \Rightarrow t_{34} = 4$;\\ $3\leq t_{24} \leq 10 \Rightarrow t_{24} = 4$ or $8$\end{tabular}}                                            & \begin{tabular}[c]{@{}l@{}}$ 5\leq t_{14} \leq 8 \Rightarrow t_{14} = 6$;  $5\leq t_{34} \leq 8 \Rightarrow t_{34} = 8$;\\ $5\leq t_{24} \leq 12 \Rightarrow t_{24} = 8$ or $12$\end{tabular}                                        \\ \hline
\end{tabular}
}
\end{table}

%\vspace{-0.2cm}

%	 \subsection*{Unicity of $V_{10}$}
%\, - \, $\lambda = (0001111000)$}

%\indent 

%\vspace{-0.5cm}
{\itshape Case} $i=10$: $\lambda_{10} = (0001111000)$; $t_1\equiv  t_2\equiv t_3 \equiv 0\;(\mbox{mod}\;8)$; $t_4 \equiv 4\;(\mbox{mod}\;8)$; $t_{12}\equiv  t_{13}\equiv 2 \;(\mbox{mod}\;4)$; $t_{14}\equiv 2 \;(\mbox{mod}\;4)$; $t_{23}\equiv  t_{24}\equiv t_{34}\equiv 0 \;(\mbox{mod}\;4)$.

%\begin{description}
%	\item[$\lambda_{10} = (0001111000)$:] $t_1\equiv  t_2\equiv t_3 \equiv 0\;(\mbox{mod}\;8)$; $t_4 \equiv 4\;(\mbox{mod}\;8)$; $t_{12}\equiv  t_{13}\equiv 2 \;(\mbox{mod}\;4)$; $t_{14}\equiv 2 \;(\mbox{mod}\;4)$; $t_{23}\equiv  t_{24}\equiv t_{34}\equiv 0 \;(\mbox{mod}\;4)$.
%\end{description}

Let's prove that $t_{1} = t_{2} = 8$ and $t_{12} =2$. If $t_{1} = 8$, then $t_{12} = 2$ or $6$. Case $t_{12} =6$, $l_{12} < 19$ implies $t_{2} = 8$ or $16$. There isn't representations in this case. Indeed, if $t_{2} = 8$ and $l_{23} < 19$, then $t_{23} = 4$ and $t_{3} = 8$. As $t_{12} + t_{23} - t_{123} < t_{2}$, so $t_{123} = 3$, which implies $t_{13} > 2$. This is an absurd, because we need to have $t_{12} + t_{13} - t_{123} < t_{1} = 8$. The analysis for $t_{2} = 16$ in this case is analogous. Therefore, $t_{12} = 2$. Furthemore, $l_{12} < 19$ implies $t_{2} = 8$. Assume $v_{1} = (1-8)$ and $v_{2} = (1,2,9,10,11,12,13,14)$. If $t_{1} = 16$, then $t_{2} - t_{12} < 3$, which only has a solution if $t_{2} = 8$ and $t_{12} = 6$. There is also no representation in this case.

By $t_{12} = 2$, $t_{2} = 8$ and $l_{23} < 19$, note that $t_{123} = 1$, $t_{23} = 4$ and $t_{3} = 8$. We use the fact that $t_{13} + t_{23} - t_{123} < t_{3}$ to obtain $t_{13} = 2$. Then, we can assume $v_{3} = (1,3,9,10,11,15,16,17)$.

%%% v4
To find $v_4$, we anlyze the possibilities for the weights $t, t_{i4}$ and $t_{ij4}$, presented in Table \ref{tab.V10}.  In items (1.2.), (2.1.) and (2.2.), we have representations of degree 	$21$, $20$ and $22$, respectively. As last, in item (1.1.), we have $t_{14} + t_{24} + t_{34} - t_{124} - t_{134} - t_{234} + t = 2$, so $t_4 = 4$ and we can consider $v_4 = (4,5,18,19)$. Therefore, $V= V_{10}$.

\vspace{-0.8cm}
%%%%Table V10
%\scalefont{0.6}

\begin{table}[H]
\centering
\caption{$V_{10}$. \label{tab.V10}}
{\scalefont{0.8}
\begin{tabular}{|l|l|}
\hline
1.	$t=0$                                                                                                                                                                                                                                    & 2. $t=1$                                                                                                                                                                                                                                    \\ \hline \hline
\begin{tabular}[c]{@{}l@{}}$0 \leq t_{124} \leq 1 \Rightarrow t_{124} =0$\\ $0 \leq t_{134} \leq 1 \Rightarrow t_{134} =0$\\ $0 \leq t_{234} \leq 3 \Rightarrow t_{234} =0$ or $2$\end{tabular}                                          & \begin{tabular}[c]{@{}l@{}}$1 \leq t_{124} \leq 2 \Rightarrow t_{124} =2$\\ $1 \leq t_{134} \leq 2 \Rightarrow t_{134} =2$\\ $1 \leq t_{234} \leq 4 \Rightarrow t_{234} =2$ or $4$\end{tabular}                                          \\ \hline
\begin{tabular}[c]{@{}l@{}} 1.1. $t_{124}= t_{134} = t_{234} = 0$ \\ $\Rightarrow 0 \leq t_{14} \leq 5 \Rightarrow t_{14} =2$; $0\leq t_{24} \leq 3 \Rightarrow t_{24} = 0$;\\  $0\leq t_{34} \leq 3 \Rightarrow t_{34} = 0$\end{tabular}      & \begin{tabular}[c]{@{}l@{}}2.1. $t_{124}= t_{134} = t_{234} = 2$\\ $ \Rightarrow 3 \leq t_{14} \leq 8 \Rightarrow t_{14} =6$; $3\leq t_{24} \leq 6 \Rightarrow t_{24} = 4$; \\ $3\leq t_{34} \leq 6 \Rightarrow t_{34} = 4$\end{tabular}      \\ \hline
\begin{tabular}[c]{@{}l@{}}1.2. $t_{124}= t_{134} = 0$; $t_{234} = 2$ \\ $\Rightarrow 0 \leq t_{14} \leq 5 \Rightarrow t_{14} =2$; $2\leq t_{24} \leq 5 \Rightarrow t_{24} = 4$; \\ $2\leq t_{34} \leq 5 \Rightarrow t_{34} = 4$\end{tabular} & \begin{tabular}[c]{@{}l@{}} 2.2. $t_{124}= t_{134} = 2$; $t_{234} = 4$ \\ $\Rightarrow 3 \leq t_{14} \leq 8 \Rightarrow t_{14} =6$; $5\leq t_{24} \leq 8 \Rightarrow t_{24} = 8$; \\ $5\leq t_{34} \leq 8 \Rightarrow t_{34} = 8$\end{tabular} \\ \hline
\end{tabular}
}
\end{table}

%	\vspace{-0.3cm}
%%	 \subsection*{Unicity of $V_{11}$}
%\, - \, $\lambda = (0001001000)$}

%	\indent 

\vspace{-0.3cm}
{\itshape Case} $i=11$: $\lambda_{11} = (0001001000)$; $t_1\equiv t_2\equiv t_3 \equiv 0\;(\mbox{mod}\;8)$; $t_4 \equiv 4\;(\mbox{mod}\;8)$; $t_{14}\equiv  2 \;(\mbox{mod}\;4)$; $t_{12}\equiv  t_{13}\equiv t_{23}\equiv t_{24} \equiv t_{34}\equiv 0 \;(\mbox{mod}\;4)$.

%	\begin{description}
%		\item[$\lambda_{11} = (0001001000)$:] $t_1\equiv t_2\equiv t_3 \equiv 0\;(\mbox{mod}\;8)$; $t_4 \equiv 4\;(\mbox{mod}\;8)$; $t_{14}\equiv  2 \;(\mbox{mod}\;4)$; $t_{12}\equiv  t_{13}\equiv t_{23}\equiv t_{24} \equiv t_{34}\equiv 0 \;(\mbox{mod}\;4)$.
%	\end{description}
%	
Firt, note that $t_{1} < 16$. If $t_{1} = 16$, for instance, we will have $t_{12} = 4,8$ or $12$. Case $t_{12} = 4$: since $l_{12} < 17$ so $t_{2} < 5$, which is an absurd. Case $t_{12} = 8$ or $12$: $t_{2} \geq 16$, which give us $l_{12} > 17$, which can't happen either. Therefore, $t_{1} = 8$ and, this way $t_{12} = t_{13} = 4$. By $l_{12} < 17$, we get $t_{2} = 8$. Analogously, $t_{3} = 8$ and $t_{23} = 4$. Assume $v_{1} = (1-8)$ and $v_{2} = (1,2,3,4,9,10,11,12)$.

%%%% v4  $t_{14} + t_{24} + t_{34} - t_{124} - t_{134} - t_{234} + t = 8$, $t_4 = 12$

Let's find $v_{3}$ and $v_{4}$. According to Table \ref{tab.V11}, there are two cases to analyze: $t_{123} = 1$ and $t_{123} = 3$. In case 1, we don't have representations. In case 2, we can consider $v_3 = (1,2,3,5,9,13,14,15)$. In item (2.2.) we also dont't have a representation. In last two cases (2.3. and 2.4.) we have non-minimal representations. But if $t=0$, that is, in item (2.1.), we have $t_{14} + t_{24} + t_{34} - t_{124} - t_{134} - t_{234} + t = 2$. Then if $t_4 = 4$ and $v_4 = (6,7,16,17)$ we obtain $V=V_{11}$. If $t_4 = 12$, we will obtain a representation with degree greater than $17$.

\vspace{-0.6cm}
%%%%TABLE V11
%\scalefont{0.6}
\begin{table}[H]
\centering
\caption{$V_{11}$. \label{tab.V11}}
{\scalefont{0.8}
\begin{tabular}{|ll|}
\hline
\multicolumn{2}{|l|}{\textbf{Case 1:} $t_{123} = 1$}                                                                                                                                                                                                                                                                                                                                                                                                                                                                                                                                                                                                                                                                                                                                                                                                                                                                                                                                                                                                                                                                                                                                                                                                                                                                                    \\ \hline \hline
\multicolumn{1}{|l|}{1.1. $t=0$: $0 \leq t_{ij4} \leq 3 \Rightarrow t_{ij4} =0$ or $2$}                                                                                                                                                                                                                                                                                                                                                                                                                                                                                                                                                                                  & 1.2. $t=1$: $1 \leq t_{ij4} \leq 4 \Rightarrow t_{ij4} =2$ or $4$                                                                                                                                                                                                                                                                                                                                                                                                                                                                                                                                                                                   \\ \hline
\multicolumn{1}{|l|}{\begin{tabular}[c]{@{}l@{}}$t_{124}= t_{134} = t_{234} = 0$ $\Rightarrow 0 \leq t_{14} \leq 1$\\ $t_{124}= t_{134} = t_{234} = 2$ $\Rightarrow 4 \leq t_{14} \leq 5$\\ $t_{124}= 2$; $t_{134} = t_{234} = 0$ $\Rightarrow 2\leq t_{24} \leq 3$\\ $t_{124}= t_{134} = 2$; $t_{234} =  0$ $\Rightarrow 4 \leq t_{14} \leq 5$\\ $t_{124}= t_{234} = 2$; $t_{134} = 0$ $\Rightarrow 2\leq t_{34} \leq 3 $\\ $t_{124}= t_{134} = 0$; $t_{234} = 2$ $\Rightarrow 0 \leq t_{14} \leq 1$\\ $t_{124}= t_{234} = 0$; $t_{134} =  2$ $\Rightarrow 2\leq t_{34} \leq 3$\\ $t_{124}= 0$; $t_{134} = t_{234} = 2$ $\Rightarrow 2\leq t_{24} \leq 3$\end{tabular}} & \begin{tabular}[c]{@{}l@{}}$t_{124}= t_{134} = t_{234} = 2$ $\Rightarrow 3 \leq t_{14} \leq 4$\\ $t_{124}= t_{134} = t_{234} = 4$ $\Rightarrow 7 \leq t_{14} \leq 8$\\ $t_{124}= t_{134} = 2$; $t_{234} = 4$ $\Rightarrow 3 \leq t_{14} \leq 4$\\ $t_{124}= t_{134} = 4$; $t_{234} = 2$ $\Rightarrow 7 \leq t_{14} \leq 8$\\ $t_{124}= 2$; $t_{134} = t_{234} = 4$ $\Rightarrow 5\leq t_{24} \leq 6$\\ $t_{124}= 4$; $t_{134} = t_{234} = 2$ $\Rightarrow 5\leq t_{24} \leq 6$\\ $t_{124}= t_{234} = 4$ ; $t_{134} =  2$ $\Rightarrow 5 \leq t_{34} \leq 6$\\ $t_{124}= t_{234} = 2$; $t_{134} = 4$ $\Rightarrow 5 \leq t_{34} \leq 6$\end{tabular} \\ \hline \hline
\multicolumn{2}{|l|}{\textbf{Case 2:} $t_{123} = 3$}                                                                                                                                                                                                                                                                                                                                                                                                                                                                                                                                                                                                                                                                                                                                                                                                                                                                                                                                                                                                                                                                                                                                                                                                                                                                                    \\ \hline \hline
\multicolumn{1}{|l|}{2.1. $t=0$: $0 \leq t_{ij4} \leq 1 \Rightarrow t_{ij4} =0$}                                                                                                                                                                                                                                                                                                                                                                                                                                                                                                                                                                                         & 2.3. $t=2$: $2 \leq t_{ij4} \leq 3 \Rightarrow t_{ij4} =2$                                                                                                                                                                                                                                                                                                                                                                                                                                                                                                                                                                                          \\ \hline
\multicolumn{1}{|l|}{\begin{tabular}[c]{@{}l@{}}$0\leq t_{14} \leq 3 \Rightarrow t_{14} = 2$; $0\leq t_{24} \leq 3 \Rightarrow t_{24} = 0$;\\ $0\leq t_{34} \leq 3 \Rightarrow t_{34} = 0$\end{tabular}}                                                                                                                                                                                                                                                                                                                                                                                                                                                                 & \begin{tabular}[c]{@{}l@{}}$2\leq t_{14} \leq 5 \Rightarrow t_{14} = 2$; $2\leq t_{24} \leq 5 \Rightarrow t_{24} = 4$;\\ $2\leq t_{34} \leq 5 \Rightarrow t_{34} = 4$\end{tabular}                                                                                                                                                                                                                                                                                                                                                                                                                                                                  \\ \hline
\multicolumn{1}{|l|}{2.2. $t=1$: $1 \leq t_{ij4} \leq 2 \Rightarrow t_{ij4} =2$}                                                                                                                                                                                                                                                                                                                                                                                                                                                                                                                                                                                         & 2.4. $t=3$: $3 \leq t_{ij4} \leq 4 \Rightarrow t_{ij4} =4$                                                                                                                                                                                                                                                                                                                                                                                                                                                                                                                                                                                          \\ \hline
\multicolumn{1}{|l|}{$3\leq t_{24} \leq 6$}                                                                                                                                                                                                                                                                                                                                                                                                                                                                                                                                                                                                                              & \begin{tabular}[c]{@{}l@{}}$5\leq t_{14} \leq 8 \Rightarrow t_{14} = 6$; $5\leq t_{24} \leq 8 \Rightarrow t_{24} = 8$;\\ $5\leq t_{34} \leq 8 \Rightarrow t_{34} = 8$\end{tabular}                                                                                                                                                                                                                                                                                                                                                                                                                                                                  \\ \hline
\end{tabular}
}
\end{table}

%	 \subsection*{Unicity of $V_{12}$}
%\, - \, $\lambda = (0000001100)$}

%	\indent 
\vspace{-0.2cm}
{\itshape Case} $i=12$: $\lambda_{12} = (0000001100)$; $t_1\equiv  t_2\equiv t_3\equiv t_4 \equiv 0\;(\mbox{mod}\;8)$; $t_{12}\equiv  t_{13}\equiv t_{24}\equiv 0 \;(\mbox{mod}\;4)$; $t_{34}\equiv 0 \;(\mbox{mod}\;4)$; $t_{14}\equiv t_{23}\equiv 2 \;(\mbox{mod}\;4)$.

%	\begin{description}
%		\item[$\lambda_{12} = (0000001100)$:] $t_1\equiv  t_2\equiv t_3\equiv t_4 \equiv 0\;(\mbox{mod}\;8)$; $t_{12}\equiv  t_{13}\equiv t_{24}\equiv 0 \;(\mbox{mod}\;4)$; $t_{34}\equiv 0 \;(\mbox{mod}\;4)$; $t_{14}\equiv t_{23}\equiv 2 \;(\mbox{mod}\;4)$.
%	\end{description}
%	
Similarly to what was done in the previous case, we obtain $t_{1} = t_{2} = t_{3} = 8$ and $t_{12} = t_{13} = 4$. Let $v_{1} = (1-8)$ and $v_{2} = (1,2,3,4,9,10,11,12)$. Note that $t_{23} = 2$ or $6$. 
Furthermore, since $t_{12} + t_{23} - t_{123} < 8$, then $t_{23} - t_{123} < 4$. Therefore there is two cases to analyze (see Table \ref{tab.V12}).

Note that it is only possible to find representations in the following cases: (1.1.2.), (1.1.3.), (1.2.2.), (1.2.3.), (2.1.1.), (2.2.1.), (2.4.2.) and (2.3.2.). In the last item we find the minimal representation we want. Indeed:   
In case 1, consider $v_3 = (1,5,6,7,9,13,\break 14,15)$. For (1.1.2.): $t_{14} + t_{24} + t_{34} - t_{124} - t_{134} - t_{234} + t = 4$. If, for instance, we consider $v_4 = (i_1, \cdots, i_4, 16,17,18,19)$, with $\{i_1,\cdots,i_4\} \subset \{1,\cdots,15\}$, we will have a non-minimal representation with degree equal to $19$. For (1.2.2.): $t_{14} + t_{24} + t_{34} - t_{124} - t_{134} - t_{234} + t = 11$. We also obtain a non-minimal representation. For items (1.1.3.) and (1.2.3.): the analysis is analogous to the previous item.

In case 2, let $v_3 = (1,2,3,5,9,10,11,13)$. For (2.1.1.), (2.2.1.) and (2.4.2.):  we obtain a non-minimal representation of degree equal to $19$, $20$ and $18$, respectively. For 
(2.3.2.): $t_{14} + t_{24} + t_{34} - t_{124} - t_{134} - t_{234} + t = 4$ and $t_4 = 8$. If we consider $v_4 = (1,2,9,10,14,15,16,17)$ we get $V= V_{12}$.

\vspace{-0.2cm}
%%%%%%%%TABLE V12
%\scalefont{0.6}

\begin{table}[H]
\centering
\caption{$V_{12}$. \label{tab.V12}}
{\scalefont{0.8}
\begin{tabular}{|ll|}
\hline
\multicolumn{2}{|l|}{\textbf{Case 1:} $t_{123} = 1$ and $t_{23} = 2$}                                                                                                                                                                                                                                                                                                                                                                                                                                                          \\ \hline \hline
\multicolumn{1}{|l|}{1.1. $t=0$}                                                                                                                                                                                                                                     & 1.2. $t=1$                                                                                                                                                                                                                                     \\ \hline
\multicolumn{1}{|l|}{\begin{tabular}[c]{@{}l@{}}$0 \leq t_{124} \leq 3 \Rightarrow t_{124} =0$ or $2$\\ $0 \leq t_{134} \leq 3 \Rightarrow t_{134} =0$ or $2$\\ $0 \leq t_{234} \leq 1 \Rightarrow t_{234} =0$\end{tabular}}                                         & \begin{tabular}[c]{@{}l@{}}$1 \leq t_{124} \leq 4 \Rightarrow t_{124} =2$ or $4$\\ $1 \leq t_{134} \leq 4 \Rightarrow t_{134} =2$ or $4$\\ $1 \leq t_{234} \leq 2 \Rightarrow t_{234} =2$\end{tabular}                                         \\ \hline
\multicolumn{1}{|l|}{1.1.1. $t_{124}= t_{134} = t_{234} = 0$ $\Rightarrow 0 \leq t_{14} \leq 1$}                                                                                                                                                                     & 1.2.1. $t_{124}= t_{134} = t_{234} = 2$ $\Rightarrow 3 \leq t_{14} \leq 4$                                                                                                                                                                     \\ \hline
\multicolumn{1}{|l|}{\begin{tabular}[c]{@{}l@{}}1.1.2. $t_{124}= t_{234} = 0$; $t_{134} =  2$\\ $\Rightarrow 2 \leq t_{14} \leq 3 \Rightarrow t_{14} =2$; $0\leq t_{24} \leq 3 \Rightarrow t_{24} = 0;$\\ $2\leq t_{34} \leq 5 \Rightarrow t_{34} = 4$\end{tabular}} & \begin{tabular}[c]{@{}l@{}}1.2.2. $t_{124}= t_{234} = 2$; $t_{134} =  4$\\ $\Rightarrow 5 \leq t_{14} \leq 6 \Rightarrow t_{14} =6;$ $3\leq t_{24} \leq 6 \Rightarrow t_{24} = 4$;\\ $5\leq t_{34} \leq 8 \Rightarrow t_{34} = 8$\end{tabular} \\ \hline
\multicolumn{1}{|l|}{\begin{tabular}[c]{@{}l@{}}1.1.3. $t_{124}= 2 $; $t_{134} = t_{234} = 0$\\ $\Rightarrow 2 \leq t_{14} \leq 3 \Rightarrow t_{14} =2;$ $2\leq t_{24} \leq 5 \Rightarrow t_{24} = 4$;\\ $0\leq t_{34} \leq 3 \Rightarrow t_{34} = 0$\end{tabular}} & \begin{tabular}[c]{@{}l@{}}1.2.3.  $t_{124}= 4$; $t_{134} = t_{234} = 2$\\ $\Rightarrow 5 \leq t_{14} \leq 6 \Rightarrow t_{14} =6;$ $5\leq t_{24} \leq 8 \Rightarrow t_{24} = 8$;\\ $3\leq t_{34} \leq 6 \Rightarrow t_{34} = 4$\end{tabular} \\ \hline
\multicolumn{1}{|l|}{1.1.4. $t_{124}= t_{134} = 2$; $t_{234} = 0$ $\Rightarrow 4 \leq t_{14} \leq 5$}                                                                                                                                                                & 1.2.4. $t_{124}= t_{134} = 4$; $t_{234} = 2$ $\Rightarrow 7 \leq t_{14} \leq 8$                                                                                                                                                                \\ \hline \hline
\multicolumn{2}{|l|}{\textbf{Case 2:} $t_{123} = 3$ and $t_{23} = 6$}                                                                                                                                                                                                                                                                                                                                                                                                                                                          \\ \hline \hline
\multicolumn{1}{|l|}{2.1. $t=0$}                                                                                                                                                                                                                                     & 2.3. $t=2$                                                                                                                                                                                                                                     \\ \hline
\multicolumn{1}{|l|}{\begin{tabular}[c]{@{}l@{}}$0 \leq t_{124} \leq 1 \Rightarrow t_{124} =0$\\ $0 \leq t_{134} \leq 1 \Rightarrow t_{134} =0$\\ $0 \leq t_{234} \leq 3 \Rightarrow t_{234} =0$ or $2$\end{tabular}}                                                & \begin{tabular}[c]{@{}l@{}}$2 \leq t_{124} \leq 3 \Rightarrow t_{124} =2$\\ $2 \leq t_{134} \leq 3 \Rightarrow t_{134} =2$\\ $2 \leq t_{234} \leq 5 \Rightarrow t_{234} =2$ or $4$\end{tabular}                                                \\ \hline
\multicolumn{1}{|l|}{\begin{tabular}[c]{@{}l@{}}2..1.1. $t_{124}= t_{134} = t_{234} = 0$\\ $\Rightarrow 0 \leq t_{14} \leq 3 \Rightarrow t_{14} =2;$ $0\leq t_{24} \leq 1 \Rightarrow t_{24} = 0;$\\ $0\leq t_{34} \leq 1 \Rightarrow t_{34} = 0$\end{tabular}}      & 2.3.1. $t_{124}= t_{134} = t_{234} = 2$ $\Rightarrow 2\leq t_{24} \leq 3$                                                                                                                                                                      \\ \hline
\multicolumn{1}{|l|}{2.1.2. $t_{124}= t_{134} = 0$; $t_{234} = 2$ $\Rightarrow 2\leq t_{24} \leq 3$}                                                                                                                                                                 & \begin{tabular}[c]{@{}l@{}}2.3.2. $t_{124}= t_{134} = 2$; $t_{234} = 4$ \\ $\Rightarrow 2 \leq t_{14} \leq 5 \Rightarrow t_{14} =2$; $4\leq t_{24} \leq 5 \Rightarrow t_{24} = 4$;\\ $4\leq t_{34} \leq 5 \Rightarrow t_{34} = 4$\end{tabular} \\ \hline
\multicolumn{1}{|l|}{2.2. $t=1$}                                                                                                                                                                                                                                     & 2.4. $t=3$                                                                                                                                                                                                                                     \\ \hline
\multicolumn{1}{|l|}{\begin{tabular}[c]{@{}l@{}}$1 \leq t_{124} \leq 2 \Rightarrow t_{124} =2$\\ $1 \leq t_{134} \leq 2 \Rightarrow t_{134} =2$\\ $1 \leq t_{234} \leq 4 \Rightarrow t_{234} =2$ or $4$\end{tabular}}                                                & \begin{tabular}[c]{@{}l@{}}$3 \leq t_{124} \leq 4 \Rightarrow t_{124} =4$\\ $3 \leq t_{134} \leq 4 \Rightarrow t_{134} =4$\\ $3 \leq t_{234} \leq 6 \Rightarrow t_{234} =4$ or $6$\end{tabular}                                                \\ \hline
\multicolumn{1}{|l|}{\begin{tabular}[c]{@{}l@{}}2.2.1. $t_{124}= t_{134} = t_{234} = 2$ \\ $\Rightarrow 3 \leq t_{14} \leq 6 \Rightarrow t_{14} =6;$ $3\leq t_{24} \leq 4 \Rightarrow t_{24} = 4;$\\ $3\leq t_{34} \leq 4 \Rightarrow t_{34} = 4$\end{tabular}}      & 2.4.1. $t_{124}= t_{134} = t_{234} = 4$ $\Rightarrow 5\leq t_{24} \leq 6$                                                                                                                                                                      \\ \hline
\multicolumn{1}{|l|}{2.2.2. $t_{124}= t_{134} = 2$; $t_{234} = 4$ $\Rightarrow 5\leq t_{24} \leq 6$}                                                                                                                                                                 & \begin{tabular}[c]{@{}l@{}}2.4.2. $t_{124}= t_{134} = 4$; $t_{234} = 6$\\ $\Rightarrow 5 \leq t_{14} \leq 8 \Rightarrow t_{14} =6$; $7\leq t_{24} \leq 8 \Rightarrow t_{24} = 8;$\\ $7\leq t_{34} \leq 8 \Rightarrow t_{34} = 8$\end{tabular}  \\ \hline
\end{tabular}
}
\end{table}

%	 \subsection*{Unicity of $V_{13}$}
%\, - \, $\lambda = (0110111100)$}

%\indent 
%\vspace{-0.5cm}
{\itshape Case} $i=13$: $\lambda_{13} = (0110111100)$; $t_1\equiv  t_4 \equiv 0\;(\mbox{mod}\;8)$; $t_2\equiv  t_3 \equiv 4\;(\mbox{mod}\;8)$; $t_{24}\equiv  t_{34}\equiv 0 \;(\mbox{mod}\;4)$; $t_{12}\equiv  t_{13}\equiv t_{14} \equiv t_{23}\equiv 2 \;(\mbox{mod}\;4)$.

%\begin{description}
%	\item[$\lambda_{13} = (0110111100)$:] $t_1\equiv  t_4 \equiv 0\;(\mbox{mod}\;8)$; $t_2\equiv  t_3 \equiv 4\;(\mbox{mod}\;8)$; $t_{24}\equiv  t_{34}\equiv 0 \;(\mbox{mod}\;4)$; $t_{12}\equiv  t_{13}\equiv t_{14} \equiv t_{23}\equiv 2 \;(\mbox{mod}\;4)$.
%\end{description}

As $\mbox{deg}(V) = 17$, then $t_{1} = 8$ and so, $t_{12} = 2$ or $6$ and $t_{13} = 2$ or $6$. If $t_{1} = 16$, we will have $t_{2} - t_{12} < 1$, because $l_{12} < 17$, but this can't happen. Since $t_{1} = 8$ and $l_{12} < 17$, we must have $t_{2} - t_{12} < 9$. Therefore, there are  two cases to analyze: (1) $t_2 = 12$ and $t_{12} = 6$; (2)	$t_2 = 4$ and $t_{12} = 2$ (see Table \ref{tab.V13}). 
%%%% v4 $t_{14} + t_{24} + t_{34} - t_{124} - t_{134} - t_{234} + t = 8$, $t_4 = 12$
In case 1, there are two subcases to analyze, because we want  $l_{23} = t_2 + t_3 - t_{23}$ to be less than $17$, that is, $t_3 - t_{23} < 5$. By item (1.1.), we find $t_{13} = 6$ and $t_{123} = 5$, because $l_{13}=t_1 + t_3 -t_{13} < 17$ and $t_{12} + t_{13} - t_{123} <8$. So let $v_1 = (1-8)$, $v_2 = (1-6,9-14)$ and $v_3 = (1-5,7,9-13,15)$. It is only possible to find (non-minimal) representations in items (1.1.3.c) and (1.1.4.c). In any other item, we find some contradiction in the hypotheses related to the weights $t_{i4}$. In item (1.2.) we have $t_{13} = 2$, so $t_{123} = 1$, because $l_{13}=t_1 + t_3 -t_{13} < 17$ and $t_{12}$. Let $v_1 = (1-8)$, $v_2 = (1-6,9-14)$ and $v_3 = (1,7,9,15)$. There are non-minimal representations for items (1.2.1.c) and (1.2.2.c).

In case 2, we have $t_{23} = 2$ and $t_{123} = 1$. 	We want $l_{23} = t_2 + t_3 - t_{23}$ and $l_{13} = t_1 + t_3 - t_{13}$  to be smaller than $17$, then $t_3 = 4$ or $12$ and $t_3 - t_{13} < 9$. There are two possible solutions: (2.1.) $t_3 = 4$ and $t_{13} = 2$ and (2.2.) $t_3 = 12$ and $t_{13} = 6$. For item (2.1.), suppose $v_1 = (1-8)$, $v_2 = (1,2,9,10)$ and $v_3 = (1,3,9,11)$. By items (2.1.1.), (2.2.1.c) and (2.2.2.c), there are non-minimal representations. In item (2.1.2.), since we have $t_{14} + t_{24} + t_{34} - t_{124} - t_{134} - t_{234} + t = 2$, so $t_4 \geq 8$. If $t_4 = 8$ and $v_4 = (4,5,12-17)$, with $v_1 \cap v_4 = \{4,5\}$, we obtain $V=V_{13}$.

\begingroup        % ajuste fino de tamanho (pode tirar se não precisar)
\footnotesize
\setlength{\tabcolsep}{5pt}
\renewcommand{\arraystretch}{1.05}
\sloppy       % evita estouro em linhas muito longas de matemática
% centralizar um longtable que tem largura < \textwidth
\setlength{\LTleft}{\fill}
\setlength{\LTright}{\fill}
\begin{longtable}{|>{\raggedright\arraybackslash}p{.40\textwidth}|
>{\raggedright\arraybackslash}p{.40\textwidth}|}
%\scalefont{0.75} 
\caption{$V_{13}$.}\label{tab.V13}\\
\hline
\multicolumn{2}{|l|}{\textbf{Case 1:} $t_{2} = 12$ and $t_{12} = 6$} \\
\hline\hline
\endfirsthead

\hline
\multicolumn{2}{|l|}{\textit{Table \thetable\ (continued)}}\\
\hline
\multicolumn{2}{|l|}{\textbf{Case 1:} $t_{2} = 12$ and $t_{12} = 6$} \\
\hline\hline
\endhead

\hline
\multicolumn{2}{|r|}{\textit{Continues on the next page}}\\
\hline
\endfoot

\hline
\endlastfoot

\multicolumn{2}{|l|}{1.1. $t_{3} = 12$ and $t_{23} = 10$} \\ \hline

\multicolumn{2}{|l|}{%
\begin{tabular}[t]{@{}l@{}}
1.1.1.\ $t=0$:\\ $0 \leq t_{124} \leq 1 \Rightarrow t_{124}=0$; $0 \leq t_{134} \leq 1 \Rightarrow t_{134}=0$;  $0 \leq t_{234} \leq 5 \Rightarrow t_{234}=0, 2$ or $4$; $0 \leq t_{14} \leq 1$
\end{tabular}}\\

\multicolumn{2}{|l|}{%
\begin{tabular}[t]{@{}l@{}}
1.1.2.\ $t=1$:\\ $1 \leq t_{124} \leq 2 \Rightarrow t_{124}=2$; $1 \leq t_{134} \leq 2 \Rightarrow t_{134}=2$ $1 \leq t_{234} \leq 6 \Rightarrow t_{234}=2, 4$ or $6$; $3 \leq t_{14} \leq 4$
\end{tabular}}\\

\multicolumn{2}{|l|}{%
\begin{tabular}[t]{@{}l@{}}
1.1.3.\ $t=2$:\\ $2 \leq t_{124} \leq 3 \Rightarrow t_{124}=2$; $2 \leq t_{134} \leq 3 \Rightarrow t_{134}=2$; $2 \leq t_{234} \leq 7 \Rightarrow t_{234}=2, 4$ or $6$
\end{tabular}}\\

\multicolumn{2}{|l|}{%
\begin{tabular}[t]{@{}l@{}}
(a) $t_{124}= t_{134} = t_{234} = 2$ $\Rightarrow 2\leq t_{24} \leq 3$\\
(b) $t_{124}= t_{134} = 2$; $t_{234} = 6$ $\Rightarrow 6\leq t_{24} \leq 7$\\
(c)\ $t_{124}= t_{134} = 2$; $t_{234} = 4$: $2 \leq t_{14} \leq 3 \Rightarrow t_{14}=2$;\\
\quad $4 \leq t_{24} \leq 5 \Rightarrow t_{24}=4$; \, $4 \leq t_{34} \leq 5 \Rightarrow t_{34}=4$
\end{tabular}}\\

\multicolumn{2}{|l|}{%
\begin{tabular}[t]{@{}l@{}}
1.1.4.\ $t=3$: \\ $3\leq t_{124} \leq 4 \Rightarrow t_{124}=4$; $3 \leq t_{134} \leq 4 \Rightarrow t_{134}=4$; $3 \leq t_{234} \leq 8 \Rightarrow t_{234}=4, 6$ or $8$
\end{tabular}}\\

\multicolumn{2}{|l|}{%
\begin{tabular}[t]{@{}l@{}}
(a)\ $t_{124}= t_{134} = t_{234} = 4$ $\Rightarrow 5\leq t_{24} \leq 6$\\
(b)\ $t_{124}= t_{134} = 4$; $t_{234} = 8$ $\Rightarrow 9\leq t_{24} \leq 10$\\
(c)\ $t_{124}= t_{134} = 4$; $t_{234} = 6$: $5 \leq t_{14} \leq 6 \Rightarrow t_{14}=6$;\\
\quad $7 \leq t_{24} \leq 8 \Rightarrow t_{24}=8$;\, $7 \leq t_{34} \leq 8 \Rightarrow t_{34}=8$
\end{tabular}}\\

\multicolumn{2}{|l|}{%
\begin{tabular}[t]{@{}l@{}}
1.1.5.\ $t=4$:\\ $4 \leq t_{124} \leq 5 \Rightarrow t_{124}=2$ or $4$; $4 \leq t_{134} \leq 5 \Rightarrow t_{134}=2$; $4 \leq t_{234} \leq 9 \Rightarrow t_{234}=2$; $4 \leq t_{14} \leq 5$
\end{tabular}}\\

\multicolumn{2}{|l|}{%
\begin{tabular}[t]{@{}l@{}}
1.1.6.\ $t=5$:\\ $5 \leq t_{124} \leq 6 \Rightarrow t_{124}=2$ or $4$; $5 \leq t_{134} \leq 6 \Rightarrow t_{134}=2$; $5 \leq t_{234} \leq 10 \Rightarrow t_{234}=2$; $7 \leq t_{14} \leq 8$
\end{tabular}}\\ \hline
\multicolumn{2}{|l|}{1.2. $t_{3} = 4$ and $t_{23} = 2$} \\ \hline

\multicolumn{2}{|l|}{%
\begin{tabular}[t]{@{}l@{}}
1.2.1.\ $t=0$:\\ $0 \leq t_{124} \leq 5 \Rightarrow t_{124}=0, 2$ or $4$; $0 \leq t_{134} \leq 1 \Rightarrow t_{134}=0$;\, $0 \leq t_{234} \leq 1 \Rightarrow t_{234}=0$\\[2pt]
(a)\ $t_{124}= t_{134} = t_{234} = 0$ $\Rightarrow 0 \leq t_{14} \leq 1$\\
(b)\ $t_{124}= 4$; $t_{134} = t_{234} = 0$ $\Rightarrow 4 \leq t_{14} \leq 5$\\
(c)\ $t_{124}= 2$; $t_{134} = t_{234} = 0$: $2 \leq t_{14} \leq 3 \Rightarrow t_{14}=2$;\\
\quad $2 \leq t_{24} \leq 7 \Rightarrow t_{24}=4$;\, $0 \leq t_{34} \leq 1 \Rightarrow t_{34}=0$
\end{tabular}%
}\\

\multicolumn{2}{|l|}{%
\begin{tabular}[t]{@{}l@{}}
1.2.2.\ $t=1$:\\ $1 \leq t_{124} \leq 6 \Rightarrow t_{124}=2, 4$ or $6$; $1 \leq t_{134} \leq 2 \Rightarrow t_{134}=2$;\, $1 \leq t_{234} \leq 2 \Rightarrow t_{234}=2$
\end{tabular}%
}\\

\multicolumn{2}{|l|}{%
\begin{tabular}[t]{@{}l@{}}
(a)\ $t_{124}= t_{134} = t_{234} = 2$ $\Rightarrow 3 \leq t_{14} \leq 4$\\
(b)\ $t_{124}= 6$; $t_{134} = t_{234} = 2$ $\Rightarrow 7 \leq t_{14} \leq 8$\\
(c)\ $t_{124}= 4$; $t_{134} = t_{234} = 2$: $5 \leq t_{14} \leq 6 \Rightarrow t_{14}=6$;\\
\quad $5 \leq t_{24} \leq 10 \Rightarrow t_{24}=8$;\, $3 \leq t_{34} \leq 4 \Rightarrow t_{34}=4$
\end{tabular}%
}\\
\hline\hline

\multicolumn{2}{|l|}{\textbf{Case 2:} $t_{2} = 4$ and $t_{13} = 2$} \\ \hline\hline
\multicolumn{2}{|l|}{2.1. $t_{3}=4$ and $t_{13}=2$} \\ \hline
\multicolumn{2}{|l|}{\begin{tabular}[t]{@{}l@{}}2.1.1. $t=0$:\\ $0 \leq t_{ij4} \leq 1 \Rightarrow t_{ij4} =0$
$\Rightarrow 0 \leq t_{14} \leq 5 \Rightarrow t_{14} =2$; $0\leq t_{24} \leq 1 \Rightarrow t_{24} = 0$; $0\leq t_{34} \leq 1 \Rightarrow t_{34} = 0$
\end{tabular}} \\
\multicolumn{2}{|l|}{\begin{tabular}[t]{@{}l@{}}2.1.2. $t=1$: \\ $1 \leq t_{ij4} \leq 2 \Rightarrow t_{ij4} =2$
$\Rightarrow 3 \leq t_{14} \leq 8 \Rightarrow t_{14} =6$; $3\leq t_{24} \leq 4 \Rightarrow t_{24} = 4$; $3\leq t_{34} \leq 4 \Rightarrow t_{34} = 4$
\end{tabular}} \\ \hline

\multicolumn{2}{|l|}{2.2. $t_{3}=12$ and $t_{13}=6$} \\ \hline
% \multicolumn{2}{|l|}{2.2.1. $t=0$: $0 \leq t_{124} \leq 1 \Rightarrow t_{124} =0$; $0 \leq t_{134} \leq 5 \Rightarrow t_{134} =0, 2$ or $4$; $0 \leq t_{234} \leq 1 \Rightarrow t_{234} =0$} \\
\multicolumn{2}{|l|}{%
\begin{tabular}[t]{@{}l@{}}
2.2.1.\ $t=0$:\\ $0 \leq t_{124} \leq 1 \Rightarrow t_{124}=0$; $0 \leq t_{134} \leq 5 \Rightarrow t_{134}=0, 2$ or $4$;\, $0 \leq t_{234} \leq 1 \Rightarrow t_{234}=0$
\end{tabular}%
}\\

\multicolumn{2}{|l|}{%
\begin{tabular}[t]{@{}l@{}}
(a)\ $t_{124}= t_{134} = t_{234} = 0$ $\Rightarrow 0 \leq t_{14} \leq 1$\\
(b)\ $t_{124}=t_{234} = 0$; $t_{134} = 4$ $\Rightarrow 4 \leq t_{14} \leq 5$\\
(c)\ $t_{124}=t_{234} = 0$; $t_{134} = 2$: $2 \leq t_{14} \leq 3 \Rightarrow t_{14}=2$;\\
\quad $0 \leq t_{24} \leq 1 \Rightarrow t_{24}=0$;\, $2 \leq t_{34} \leq 7 \Rightarrow t_{34}=4$
\end{tabular}%
}\\

\multicolumn{2}{|l|}{%
\begin{tabular}[t]{@{}l@{}}
2.2.2.\ $t=1$:\\ $1 \leq t_{124} \leq 2 \Rightarrow t_{124}=2$; $1 \leq t_{134} \leq 6 \Rightarrow t_{134}=2, 4$ or $6$;\, $1 \leq t_{234} \leq 2 \Rightarrow t_{234}=2$
\end{tabular}%
}\\

\multicolumn{2}{|l|}{%
\begin{tabular}[t]{@{}l@{}}
(a)\ $t_{124}= t_{134} = t_{234} = 2$ $\Rightarrow 3 \leq t_{14} \leq 4$\\
(b)\ $t_{124}=t_{234} = 2$; $t_{134} = 6$ $\Rightarrow 8 \leq t_{14} \leq 9$\\
(c)\ $t_{124}=t_{234} = 2$; $t_{134} = 4$: $5 \leq t_{14} \leq 6 \Rightarrow t_{14}=6$;\\
\quad $3 \leq t_{24} \leq 4 \Rightarrow t_{24}=4$;\, $5 \leq t_{34} \leq 10 \Rightarrow t_{34}=8$
\end{tabular}%
}\\
\hline
\end{longtable}
\endgroup

%%%%%%%%%%%v14
%\vspace{-0.5cm}
{\itshape Case} $i=14$: $\lambda_{14} = (0001001100)$; $t_1\equiv  t_2 \equiv t_3 \equiv 0\;(\mbox{mod}\;8)$; $t_4 \equiv 4\;(\mbox{mod}\;8)$; $t_{14}\equiv  t_{23}\equiv 2 \;(\mbox{mod}\;4)$; $t_{12}\equiv  t_{13}\equiv t_{24} \equiv t_{34}\equiv 0 \;(\mbox{mod}\;4)$.

%	\begin{description}
%		\item[$\lambda_{14} = (0001001100)$:] $t_1\equiv  t_2 \equiv t_3 \equiv 0\;(\mbox{mod}\;8)$; $t_4 \equiv 4\;(\mbox{mod}\;8)$; $t_{14}\equiv  t_{23}\equiv 2 \;(\mbox{mod}\;4)$; $t_{12}\equiv  t_{13}\equiv t_{24} \equiv t_{34}\equiv 0 \;(\mbox{mod}\;4)$.
%	\end{description}

Since $\mbox{deg}(V) \leq 13$, then $t_{1} = t_{2} = t_{3} = 8$, $t_{12} = t_{13} =4$ and $t_{23} = 6$ ($l_{23} < 13$).	 Using the facts $t_{12} + t_{23} - t_{123} < 8$ and $t_{123} < t_{12} = t_{13}$, we get $t_{123} = 3$. So let $v_{1} = (1-8)$, $v_{2} = (1,2,3,4,9,10,11,12)$ and $v_{3} = (1,2,3,5,9,10,11,13)$.
%%% v4 $t_{14} + t_{24} + t_{34} - t_{124} - t_{134} - t_{234} + t = 8$, $t_4 = 12$

There are four possibilities for the weight $t$ (see Table \ref{tab.V14}). By items (1.1.), (2.1.) and (4.1.), there are non-minimal representations of degree equal to $15$,  $16$ and $14$, respectively. By item (3.1.), as $t_{14} + t_{24} + t_{34} - t_{124} - t_{134} - t_{234} + t = 4$, we can consider $t_4 = 4$ and $v_4 = (1,2,9,10)$. Thus, $V=V_{14}$. In other items there are no representations.

%\vspace{-0.4cm}
%%%%%%TABLE V14
%\scalefont{0.6}

\begin{table}[h!]
\centering
\caption{$V_{14}$. \label{tab.V14}}
{\scalefont{0.8}
\begin{tabular}{|l|l|}
\hline
1. $t=0$                                                                                                                                                                                                                               & 3. $t=2$                                                                                                                                                                                                                                    \\ \hline \hline
\begin{tabular}[c]{@{}l@{}}$0 \leq t_{124} \leq 1 \Rightarrow t_{124} =0$\\ $0 \leq t_{134} \leq 1 \Rightarrow t_{134} =0$\\ $0 \leq t_{234} \leq 3 \Rightarrow t_{234} =0$ or $2$\end{tabular}                                        & \begin{tabular}[c]{@{}l@{}}$2 \leq t_{124} \leq 3 \Rightarrow t_{124} =2$\\ $2 \leq t_{134} \leq 3 \Rightarrow t_{134} =2$\\ $2 \leq t_{234} \leq 5 \Rightarrow t_{234} =2$ or $4$\end{tabular}                                             \\ \hline
\begin{tabular}[c]{@{}l@{}}1.1. $t_{124}= t_{134} = t_{234} = 0$\\ $\Rightarrow 0 \leq t_{14} \leq 3 \Rightarrow t_{14} =2$; $0\leq t_{24} \leq 1 \Rightarrow t_{24} = 0$;\\ $0\leq t_{34} \leq 1 \Rightarrow t_{34} = 0$\end{tabular} & \begin{tabular}[c]{@{}l@{}}3.1. $t_{124}= t_{134} = 2$; $t_{234} = 4$\\ $\Rightarrow 2 \leq t_{14} \leq 5 \Rightarrow t_{14} =2$; $4\leq t_{24} \leq 5 \Rightarrow t_{24} = 4$;\\ $4\leq t_{34} \leq 5 \Rightarrow t_{34} = 4$\end{tabular} \\ \hline
1.2. $t_{124}= 2$; $t_{134} = t_{234} = 0$ $\Rightarrow 2\leq t_{24} \leq 3$                                                                                                                                                           & 3.2. $t_{124}= t_{134} = t_{234} = 2$ $\Rightarrow 2\leq t_{24} \leq 3$                                                                                                                                                                     \\ \hline \hline
2. $t=1$                                                                                                                                                                                                                               & 4. $t=3$                                                                                                                                                                                                                                    \\ \hline \hline
\begin{tabular}[c]{@{}l@{}}$1 \leq t_{124} \leq 2 \Rightarrow t_{124} =2$\\ $1 \leq t_{134} \leq 2 \Rightarrow t_{134} =2$\\ $1 \leq t_{234} \leq 4 \Rightarrow t_{234} =2$ or $4$\end{tabular}                                        & \begin{tabular}[c]{@{}l@{}}$3 \leq t_{124} \leq 4 \Rightarrow t_{124} =2$;\\ $3 \leq t_{134} \leq 4 \Rightarrow t_{134} =2$;\\ $3 \leq t_{234} \leq 6 \Rightarrow t_{234} =4$ or $6$\end{tabular}                                           \\ \hline
\begin{tabular}[c]{@{}l@{}}2.1. $t_{124}= t_{134} = t_{234} = 2$\\ $\Rightarrow 3 \leq t_{14} \leq 6 \Rightarrow t_{14} =6$; $3\leq t_{24} \leq 4 \Rightarrow t_{24} = 4$;\\ $3\leq t_{34} \leq 4 \Rightarrow t_{34} = 4$\end{tabular}  & \begin{tabular}[c]{@{}l@{}}4.1. $t_{124}= t_{134} = 4$; $t_{234} = 6$\\ $\Rightarrow 5 \leq t_{14} \leq 8 \Rightarrow t_{14} =6$; $7\leq t_{24} \leq 8 \Rightarrow t_{24} = 8$;\\ $7\leq t_{34} \leq 8 \Rightarrow t_{34} = 8$\end{tabular} \\ \hline
2.2. $t_{124}= t_{134} = 2$; $ t_{234} = 4$ $\Rightarrow 5\leq t_{24} \leq 6$                                                                                                                                                          & 4.2. $t_{124}= t_{134} = t_{234} = 4$ $\Rightarrow 5\leq t_{24} \leq 6$                                                                                                                                                                     \\ \hline
\end{tabular}
}
\end{table}

%	 \subsection*{Unicity of $V_{15}$}
%\, - \, $\lambda = (1001001100)$}

%	\indent 
%\newpage
{\itshape Case} $i=15$: $\lambda_{15} = (1001001100)$:] $t_1\equiv  t_4 \equiv 4\;(\mbox{mod}\;8)$; $t_2\equiv  t_3 \equiv 0\;(\mbox{mod}\;8)$; $t_{14}\equiv  t_{23}\equiv 2 \;(\mbox{mod}\;4)$; $t_{12}\equiv  t_{13}\equiv t_{24} \equiv t_{34}\equiv 0 \;(\mbox{mod}\;4)$.

%	\begin{description}
%		\item[$\lambda_{15} = (1001001100)$:] $t_1\equiv  t_4 \equiv 4\;(\mbox{mod}\;8)$; $t_2\equiv  t_3 \equiv 0\;(\mbox{mod}\;8)$; $t_{14}\equiv  t_{23}\equiv 2 \;(\mbox{mod}\;4)$; $t_{12}\equiv  t_{13}\equiv t_{24} \equiv t_{34}\equiv 0 \;(\mbox{mod}\;4)$.
%	\end{description}
%	

It is clear that $t_{1} = 12$. Besides, $t_{2} = 8$ and $t_{12} =4$, because $l_{12} < 17$, that is, $t_{2} - t_{12} < 5$. Assume $v_{1} = (1-12)$ and $v_{2} = (1,2,3,4,13,14,15,16)$. Analogously, we get $t_{3} = 8$ and $t_{13} = 4$.	We also have $t_{23} = 6$, and so $t_{123} =3$. If $t_{23} = 2$, then $t_{123} = 1$, which implies $l_{123} = 19 > \mbox{deg}(V)$.  Assume $v_{3} = (1,2,3,5,13,14,15,17)$. 

%%% v4 $t_{14} + t_{24} + t_{34} - t_{124} - t_{134} - t_{234} + t = 8$, $t_4 = 12$

As usual, our last step is to find $v_{4}$. According to Table \ref{tab.V15}, there are non-minimal representations by items (1.1.), (2.1.) and (4.1.). By item (3.1.), if $t_{14} = 6$, we will have a non-minimal representation. But if $t_{14} = 2$, then	$t_{14} + t_{24} + t_{34} - t_{124} - t_{134} - t_{234} + t = 4$. If $t_4 = 4$ and $v_4 = (1,2,13,14)$, we obtain $V=V_{15}$.	

\vspace{-0.5cm}
%%%%%TABLE V15
%\scalefont{0.6}
\begin{table}[H]
\centering
\caption{$V_{15}$. \label{tab.V15}}
{\scalefont{0.8}
\begin{tabular}{|l|l|}
\hline
1. $t=0$                                                                                                                                                                                                                                       & 3. $t=2$                                                                                                                                                                                                                                             \\ \hline \hline
\begin{tabular}[c]{@{}l@{}}$0 \leq t_{124} \leq 1 \Rightarrow t_{124} =0$\\ $0 \leq t_{134} \leq 1 \Rightarrow t_{134} =0$\\ $0 \leq t_{234} \leq 3 \Rightarrow t_{234} =0$ or $2$\end{tabular}                                                & \begin{tabular}[c]{@{}l@{}}$2 \leq t_{124} \leq 3 \Rightarrow t_{124} =2$\\ $2 \leq t_{134} \leq 3 \Rightarrow t_{134} =2$\\ $2 \leq t_{234} \leq 5 \Rightarrow t_{234} =2$ or $4$\end{tabular}                                                      \\ \hline
\begin{tabular}[c]{@{}l@{}}1.1. $t_{124}= t_{134} = t_{234} = 0$\\ $\Rightarrow 0 \leq t_{24} \leq 1 \Rightarrow t_{24} =0$; $0\leq t_{34} \leq 1 \Rightarrow t_{34} = 0$;\\ $0\leq t_{14} \leq 7 \Rightarrow t_{14} = 2$ or $6$;\end{tabular} & \begin{tabular}[c]{@{}l@{}}3.1. $t_{124}= t_{134} = 2$; $t_{234} = 4$\\ $\Rightarrow 4 \leq t_{24} \leq 5 \Rightarrow t_{24} =4$; $4\leq t_{34} \leq 5 \Rightarrow t_{34} = 4$;\\ $2\leq t_{14} \leq 9 \Rightarrow t_{14} = 2$ or $6$\end{tabular}   \\ \hline
1.2. $t_{124}= t_{134} = 0$; $t_{234} = 2$ $\Rightarrow 2\leq t_{24} \leq 3$                                                                                                                                                                   & 3.2. $t_{124}= t_{134} = t_{234} = 2$ $\Rightarrow 2\leq t_{24} \leq 3$                                                                                                                                                                              \\ \hline \hline
2. $t=1$                                                                                                                                                                                                                                       & 4. $t=3$                                                                                                                                                                                                                                             \\ \hline \hline
\begin{tabular}[c]{@{}l@{}}$1 \leq t_{124} \leq 2 \Rightarrow t_{124} =2$\\ $1 \leq t_{134} \leq 2 \Rightarrow t_{134} =2$\\ $1 \leq t_{234} \leq 4 \Rightarrow t_{234} =2$ or $4$\end{tabular}                                                & \begin{tabular}[c]{@{}l@{}}$3 \leq t_{124} \leq 4 \Rightarrow t_{124} =4$;\\ $3 \leq t_{134} \leq 4 \Rightarrow t_{134} =4$;\\ $3 \leq t_{234} \leq 6 \Rightarrow t_{234} =4$ or $6$\end{tabular}                                                    \\ \hline
\begin{tabular}[c]{@{}l@{}}2.1. $t_{124}= t_{134} = t_{234} = 2$\\  $3\leq t_{24} \leq 4 \Rightarrow t_{24} = 4$; $3\leq t_{34} \leq 4 \Rightarrow t_{34} = 4$;\\ $3 \leq t_{14} \leq 10 \Rightarrow t_{14} =6$ or $10$\end{tabular}           & \begin{tabular}[c]{@{}l@{}}4.1. $t_{124}= t_{134} = 4$; $t_{234} = 6$\\ $\Rightarrow 7 \leq t_{24} \leq 8 \Rightarrow t_{24} =8$; $7\leq t_{34} \leq 8 \Rightarrow t_{34} = 8$;\\ $5\leq t_{14} \leq 12 \Rightarrow t_{14} = 6$ or $10$\end{tabular} \\ \hline
2.2. $t_{124}= t_{134} =  t_{234} = 4$ $\Rightarrow 5\leq t_{24} \leq 6$                                                                                                                                                                       & 4.2. $t_{124}= t_{134} = t_{234} = 4$ $\Rightarrow 3\leq t_{24} \leq 6$                                                                                                                                                                              \\ \hline
\end{tabular}
}
\end{table}

%	

%	 \subsection*{Unicity of $V_{16}$}
%\, - \, $\lambda = (0001111100)$}

%\indent 
%\vspace{-0.5cm}
{\itshape Case} $i=16$:

\begin{description}
\item[$\lambda_{16} = (0001111100)$:] $t_1\equiv  t_2 \equiv t_3 \equiv 0\;(\mbox{mod}\;8)$; $t_4 \equiv 4\;(\mbox{mod}\;8)$; $t_{24}\equiv  t_{34}\equiv 0 \;(\mbox{mod}\;4)$; $t_{12}\equiv  t_{13}\equiv t_{23} \equiv t_{24}\equiv 2 \;(\mbox{mod}\;4)$.
\end{description}

We have $t_{12} \leq 6$, otherwise $t_1 \geq 16$, $t_2 \geq 16$ and $\mbox{deg}\,(V) \geq 22$.	Similarly, $t_{13} \leq 6$ and $t_{23} \leq 6$. Since $t_1 + t_2 - t_{12} \leq 17$ and $t_1 + t_3 - t_{13} \leq 17$, soon $t_1 = t_2 = t_3 = 8$. Let $v_1 = (1-8)$.

%%%%%%%%%%%%%%%%%%%%%%%%%%%%%%%%%%%%%%%%%%%%%%%%%%%%%%%%%%%%%%%%%%%%%%%%%%%%%%%%%%%%%%%%%%%%%%%%%%%%%%%%

%%%% v4 $t_{14} + t_{24} + t_{34} - t_{124} - t_{134} - t_{234} + t = 8$, $t_4 = 12$

Let's do the analysis in two cases: $t_{12} = 2$ and $t_{12} =6$.

In case 1, consider $v_2 = (1,2,9-14)$. In this case, $t_{123} = 1$. If $t_{13} = t_{23} =2$, we will have $\mbox{deg}\,(V)\geq 19$, because $l_{123} = 19$. If $t_{13} = t_{23} =6$, we will have $t_{3} > 11$, because $t_{13} + t_{23} - t_{123} = 11$. Thus we have only two subcases to analyze: (1.1.) $t_{13} = 2$ and $t_{23} = 6$; and (1.2.) $t_{13} = 6$ and $t_{23} = 2$. See Table \ref{tab.V16}.

For item (1.1.), suppose $v_3 = (1,3,9,10,11,12,13,15)$. By item (1.1.1.a), where $t=0$; $t_{124}= t_{134} = t_{234} = 0$ and $t_{14} = 2$, $t_{24} = t_{34} =0$, we obtain $t_{14} + t_{24} + t_{34} - t_{124} - t_{134} - t_{234} + t = 2$. If $t_4 = 4$, we can consider $v_{4} = (4,5,16,17)$. Therefore, $V= V_{16}$. Note that in any other item, if there is a representation, it has a degree greater than $17$, that is, it is not minimal. We found non-minimal representations through the analysis of the following items: (1.1.1.c), (1.1.2.a) and (1.1.2.c).

In case 2, we have $t_{123} = 1$ or $5$, since we need to have $t_{1} < t_{12} + t_{13} -t_{123}$. With an analysis to the previous case, we observe that we do not have representations in this case. For item (2.1.), use that $t_{1} = t_{2} =8$ to get $t_{13} = t_{23} =2$ and for item (2.2.), note that $t_{13} = t_{23} =6$.

%\newpage

%%%%TABLE V16  %%%
%\scalefont{0.6}
%\vspace{-0.1cm}
\begin{table}[hh!] \centering \caption{$V_{16}$. \label{tab.V16}}
{\scalefont{0.75}
\begin{tabular}{|ll|}
\hline
\multicolumn{2}{|l|}{\textbf{Case 1:} $t_{12} = 2$}                                                                                                                                                                                                                                                                                                                                                                                                                                                                                                                                                                                                                                                                                                                                                                                                                                                                                                                                                                                                                                             \\ \hline \hline
\multicolumn{2}{|l|}{1.1. $t_{13} = 2$ and $t_{23} = 6$}                                                                                                                                                                                                                                                                                                                                                                                                                                                                                                                                                                                                                                                                                                                                                                                                                                                                                                                                                                                                                               \\ \hline \hline
\multicolumn{1}{|l|}{1.1.1. $t=0$}                                                                                                                                                                                                                                                                                                                                                                                                                                                                                                                   & 1.1.2. $t=1$                                                                                                                                                                                                                                                                                                                                                                                                                                                                                                                    \\ \hline
\multicolumn{1}{|l|}{\begin{tabular}[c]{@{}l@{}}$0 \leq t_{124} \leq 1 \Rightarrow t_{124} =0$ \\ $0 \leq t_{134} \leq 1 \Rightarrow t_{134} =0$ \\ $0 \leq t_{234} \leq 5 \Rightarrow t_{234} =0, 2$ or $4$\end{tabular}}                                                                                                                                                                                                                                                                                                                           & \begin{tabular}[c]{@{}l@{}}$1 \leq t_{124} \leq 2 \Rightarrow t_{124} =2$ \\ $1 \leq t_{134} \leq 2 \Rightarrow t_{134} =2$ \\ $1 \leq t_{234} \leq 6 \Rightarrow t_{234} =2, 4$ or $6$\end{tabular}                                                                                                                                                                                                                                                                                                                            \\ \hline
\multicolumn{1}{|l|}{\begin{tabular}[c]{@{}l@{}}(a) $t_{124}= t_{134} = t_{234} = 0$ \\ $\Rightarrow 0 \leq t_{14} \leq 5 \Rightarrow t_{14} =2;$ $0\leq t_{24} \leq 1\Rightarrow t_{24} = 0;$\\ $0\leq t_{34} \leq 1 \Rightarrow t_{34} = 0$\\ \\ (b) $t_{124}= t_{134} = 0$; $t_{234} = 2$ $\Rightarrow 2\leq t_{24} \leq 3$\\ \\ (c) $t_{124}= t_{134} = 0$; $t_{234} = 4$\\ $\Rightarrow 0 \leq t_{14} \leq 5 \Rightarrow t_{14} =2;$ $4\leq t_{24} \leq 5 \Rightarrow t_{24} = 4$;\\ $4\leq t_{34} \leq 5 \Rightarrow t_{34} = 4$\end{tabular}} & \begin{tabular}[c]{@{}l@{}}(a) $t_{124}= t_{134} = t_{234} = 2$ \\ $\Rightarrow 3 \leq t_{14} \leq 8 \Rightarrow t_{14} =6$; $3\leq t_{24} \leq 4 \Rightarrow t_{24} = 4$;\\ $3\leq t_{34} \leq 4 \Rightarrow t_{34} = 4$\\ \\ (b) $t_{124}= t_{134} = 2$; $t_{234} = 4$ $\Rightarrow 5\leq t_{24} \leq 6$\\ \\ (c) $t_{124}= t_{134} = 2$; $t_{234} = 6$\\ $\Rightarrow 3 \leq t_{14} \leq 8 \Rightarrow t_{14} =6$; $7\leq t_{24} \leq 8 \Rightarrow t_{24} = 8$;\\ $7\leq t_{34} \leq 8 \Rightarrow t_{34} = 8$\end{tabular} \\ \hline
\multicolumn{2}{|l|}{1.2. $t_{13} = 6$ and $t_{23} = 2$}                                                                                                                                                                                                                                                                                                                                                                                                                                                                                                                                                                                                                                                                                                                                                                                                                                                                                                                                                                                                                               \\ \hline \hline
\multicolumn{1}{|l|}{1.2.1. $t=0$}                                                                                                                                                                                                                                                                                                                                                                                                                                                                                                                   & 1.2.2. $t=1$                                                                                                                                                                                                                                                                                                                                                                                                                                                                                                                    \\ \hline
\multicolumn{1}{|l|}{\begin{tabular}[c]{@{}l@{}}$0 \leq t_{124} \leq 1 \Rightarrow t_{124} =0$\\ $0 \leq t_{134} \leq 5 \Rightarrow t_{134} =0, 2$ or $4$\\ $0 \leq t_{234} \leq 1 \Rightarrow t_{234} =0$\end{tabular}}                                                                                                                                                                                                                                                                                                                             & \begin{tabular}[c]{@{}l@{}}$1 \leq t_{124} \leq 2 \Rightarrow t_{124} =2$\\ $1 \leq t_{134} \leq 6 \Rightarrow t_{134} =2, 4$ or $6$\\ $1 \leq t_{234} \leq 2 \Rightarrow t_{234} =2$\end{tabular}                                                                                                                                                                                                                                                                                                                              \\ \hline
\multicolumn{1}{|l|}{\begin{tabular}[c]{@{}l@{}}(a) $t_{124}= t_{134} = t_{234} = 0$ $\Rightarrow 0 \leq t_{14} \leq 1$ \\ \\ (b) $t_{124}= t_{234} = 0$; $t_{134} =  2$ $\Rightarrow 2 \leq t_{34} \leq 3$\\ \\ (c) $t_{124}= t_{234} = 0$; $t_{134} =  4$ $\Rightarrow 4 \leq t_{14} \leq 5$\end{tabular}}                                                                                                                                                                                                                                         & \begin{tabular}[c]{@{}l@{}}(a) $t_{124}= t_{134} = t_{234} = 2$ $\Rightarrow 3\leq t_{14} \leq 4$\\ \\ (b) $t_{124}= t_{234} = 2$; $t_{134} =  4$ $\Rightarrow 5\leq t_{34} \leq 6$\\ \\ (c) $t_{124}= t_{234} = 2$; $t_{134} =  6$ $\Rightarrow 7\leq t_{14} \leq 8$\end{tabular}                                                                                                                                                                                                                                              \\ \hline \hline
\multicolumn{2}{|l|}{\textbf{Case 2:} $t_{12} = 6$}                                                                                                                                                                                                                                                                                                                                                                                                                                                                                                                                                                                                                                                                                                                                                                                                                                                                                                                                                                                                                                             \\ \hline \hline
\multicolumn{2}{|l|}{2.1. $t_{123} = 1$}                                                                                                                                                                                                                                                                                                                                                                                                                                                                                                                                                                                                                                                                                                                                                                                                                                                                                                                                                                                                                                               \\ \hline \hline
\multicolumn{1}{|l|}{2.1.1. $t=0$}                                                                                                                                                                                                                                                                                                                                                                                                                                                                                                                   & 2.1.2. $t=1$                                                                                                                                                                                                                                                                                                                                                                                                                                                                                                                    \\ \hline
\multicolumn{1}{|l|}{\begin{tabular}[c]{@{}l@{}}$0 \leq t_{124} \leq 5 \Rightarrow t_{124} =0, 2$ or $4$\\ $0 \leq t_{134} \leq 1 \Rightarrow t_{134} =0$\\ $0 \leq t_{234} \leq 1 \Rightarrow t_{234} =0$\end{tabular}}                                                                                                                                                                                                                                                                                                                             & \begin{tabular}[c]{@{}l@{}}$1 \leq t_{124} \leq 6 \Rightarrow t_{124} =2, 4$ or $6$\\ $1 \leq t_{134} \leq 2 \Rightarrow t_{134} =2$\\ $1 \leq t_{234} \leq 2 \Rightarrow t_{234} =2$\end{tabular}                                                                                                                                                                                                                                                                                                                              \\ \hline
\multicolumn{1}{|l|}{\begin{tabular}[c]{@{}l@{}}(a) $t_{124}= t_{134} = t_{234} = 0$ $\Rightarrow 0 \leq t_{14} \leq 1$\\ \\ (b) $t_{124}= 2$; $t_{134} = t_{234} = 0$ $\Rightarrow 2\leq t_{24} \leq 3$\\ \\ (c) $t_{124}= 4$; $t_{134} = t_{234} = 0$ $\Rightarrow 4 \leq t_{14} \leq 5$\end{tabular}}                                                                                                                                                                                                                                             & \begin{tabular}[c]{@{}l@{}}(a) $t_{124}= t_{134} = t_{234} = 2$ $\Rightarrow 3\leq t_{14} \leq 4$\\ \\ (b) $t_{124}= 4$; $t_{134} = t_{234} = 2$ $\Rightarrow 5\leq t_{24} \leq 6$\\ \\ (c) $t_{124}= 6$; $t_{134} = t_{234} = 2$ $\Rightarrow 7 \leq t_{14} \leq 8$\end{tabular}                                                                                                                                                                                                                                               \\ \hline
\multicolumn{2}{|l|}{2.2. $t_{123}=5$}                                                                                                                                                                                                                                                                                                                                                                                                                                                                                                                                                                                                                                                                                                                                                                                                                                                                                                                                                                                                                                                 \\ \hline \hline
\multicolumn{1}{|l|}{\begin{tabular}[c]{@{}l@{}}(a) $t=0$: $0 \leq t_{ij4} \leq 1 \Rightarrow t_{ij4} =0$ $\Rightarrow 0 \leq t_{14} \leq 1$ \\ \\ (b) $t=1$: $1 \leq t_{ij4} \leq 2 \Rightarrow t_{ij4} =2$ $\Rightarrow 3 \leq t_{14} \leq 4$\\ \\ (c) $t=2$: $2 \leq t_{ij4} \leq 3 \Rightarrow t_{ij4} =2$ $\Rightarrow 2 \leq t_{24} \leq 3$\end{tabular}}                                                                                                                                                                                      & \begin{tabular}[c]{@{}l@{}}(d) $t=3$: $3 \leq t_{ij4} \leq 4 \Rightarrow t_{ij4} =4$ $\Rightarrow 5 \leq t_{24} \leq 6$\\ \\ (e) $t=4$: $4 \leq t_{ij4} \leq 5 \Rightarrow t_{ij4} =4$ $\Rightarrow 4 \leq t_{14} \leq 5$\\ \\ (f) $t=5$: $5 \leq t_{ij4} \leq 6 \Rightarrow t_{ij4} =6$; $\Rightarrow 7 \leq t_{14} \leq 8$\end{tabular}                                                                                                                                                                                       \\ \hline
\end{tabular}
}
\end{table}

\end{proof}

%%%%%%%%%%%%%%%%%%%%%%%%%%%%%%
%\newpage
As a direct consequence of the Theorem \ref{theorem3.6}, we have the following corollary.

\begin{corollary}\label{corollary4.4}Each minimal representation of the code loops $C_{1}^{4},\dots,C_{16}^{4}$ has the following degree and type, respectively:
\end{corollary}
%\vspace{-2cm}

%\begin{table}[ht]
%	\tbl{Degree and type of minimal representations.}
{\begin{center}\begin{tabular}{l|l|l|l|l|l} 
\multicolumn{1}{c|}{$i$} & \multicolumn{1}{c|}{$\mbox{deg}(V_{i})$} & \multicolumn{1}{c|}{type of $V_{i}$}&\multicolumn{1}{c|}{$i$} & \multicolumn{1}{c|}{$\mbox{deg}(V_{i})$} & \multicolumn{1}{c}{type of $V_{i}$} \\ 
\hline
\multicolumn{1}{c|}{1} & \multicolumn{1}{c|}{8} & (11111111) & 
\multicolumn{1}{c|}{2} & \multicolumn{1}{c|}{14} & (11111111222) \\ 
\multicolumn{1}{c|}{3} & \multicolumn{1}{c|}{12} & (111111114) &
\multicolumn{1}{c|}{4} & \multicolumn{1}{c|}{18} & (11111111226) \\ 
\multicolumn{1}{c|}{5} & \multicolumn{1}{c|}{18} & (111111112224) &
\multicolumn{1}{c|}{6} & \multicolumn{1}{c|}{11} & (11111114) \\ 
\multicolumn{1}{c|}{7} & \multicolumn{1}{c|}{17} & (11113334) &
\multicolumn{1}{c|}{8} & \multicolumn{1}{c|}{17} & (11111122223) \\ 
\multicolumn{1}{c|}{9} & \multicolumn{1}{c|}{19} & (11111222233) &
\multicolumn{1}{c|}{10} & \multicolumn{1}{c|}{19} & (111223333) \\ 
\multicolumn{1}{c|}{11} & \multicolumn{1}{c|}{17} & (111122333) & 
\multicolumn{1}{c|}{12} & \multicolumn{1}{c|}{17} & (1111112234) \\ 
\multicolumn{1}{c|}{13} & \multicolumn{1}{c|}{17} & (111111236) &
\multicolumn{1}{c|}{14} & \multicolumn{1}{c|}{13} & (111111223) \\ 
\multicolumn{1}{c|}{15} & \multicolumn{1}{c|}{17} & (111111227) &
\multicolumn{1}{c|}{16} & \multicolumn{1}{c|}{17} & (111112235) \\ 
\end{tabular}\end{center}}
%\end{table}
%\vspace{-0.8cm}

Note that in the case of code loops of rank 3 and 4 the type of code loop define this loop up to isomophism. May be it is true in general case.

%%%Other diagramms

%%

%\vspace{-0.7cm}
%\begin{conjecture}
%	Let $V_{1}$ and $V_{2}$ be representations of a code loop $L$. If this representations have the same degree and type, then $V_{1}$ and $V_{2}$ are isomorphic. 
%\end{conjecture}

%
%This section should come before the References. Dedications and funding
%information may also be included here.

Let's see a illustration. According to Corollary \ref{corollary4.4}, the minimal representation $V_{3}$ has 9 equivalence classes. These classes are presented in the following diagram:
\vspace{-2cm}

\begin{figure}[h!]
\centering
\definecolor{uuuuuu}{rgb}{0.26666666666666666,0.26666666666666666,0.26666666666666666}
\definecolor{ffffff}{rgb}{1,1,1}
\begin{tikzpicture}[line cap=round,line join=round,>=triangle 45,x=1.5cm,y=0.8cm,scale=1.8]
\clip(-0.5,-0.3) rectangle (8.0,5.6);
\fill[line width=0.8pt,color=ffffff] (4,0) -- (4,4) -- (0,4) -- (0,0) -- cycle;
\draw [green,line width=0.7pt] (1,4)-- (1,0);
\draw [orange,line width=0.7pt] (2,4)-- (2,0);
\draw [green,line width=0.7pt] (3,4)-- (3,0);
\draw [blue,line width=0.7pt] (0,3)-- (4,3);
\draw [red,line width=0.7pt] (0,2)-- (4,2);
\draw [blue,line width=0.7pt] (0,1)-- (4,1);
\draw [black,line width=0.7pt] (0,4)-- (4,4);
\draw [orange,line width=0.7pt] (4,4)-- (4,0);
\draw [red,line width=0.7pt] (4,0)-- (0,0);
\draw [black,line width=0.7pt] (0,0)-- (0,4);
\draw (2.3,1.65) node[anchor=north west] {$1$};
\draw (1.2,1.8) node[anchor=north west] {$2$};
\draw (1.2,1.4) node[anchor=north west] {$4$};
\draw (1.6,1.4) node[anchor=north west] {$5$};
\draw (1.6,1.8) node[anchor=north west] {$3$};
\draw (2.3,2.65) node[anchor=north west] {$6$};
%\draw (1.3,2.65) node[anchor=north west] {$X_{12}$};
%\draw (0.3,2.65) node[anchor=north west] {$X_{1}$};
%\draw (0.3,1.65) node[anchor=north west] {$X_{13}$};
\draw (3.3,1.65) node[anchor=north west] {$7$};
%\draw (1.3,3.65) node[anchor=north west] {$X_{2}$};
\draw (3.3,2.65) node[anchor=north west] {$8$};
%\draw (0.3,0.65) node[anchor=north west] {$X_{3}$};
%\draw (1.3,0.65) node[anchor=north west] {$X_{23}$};
\draw (2.3,0.65) node[anchor=north west] {$9$};
\draw (3.3,0.65) node[anchor=north west] {$11$};
\draw (2.3,3.65) node[anchor=north west] {$10$};
\draw (3.3,3.65) node[anchor=north west] {$12$};
\draw (-0.2,0.65) node[anchor=north west] {$v_{3}$};
\draw (-0.2,2.65) node[anchor=north west] {$v_{1}$};
\draw (-0.5,1.65) node[anchor=north west] {$v_{1}\cap v_{3}$};
\draw (1.3,4.35) node[anchor=north west] {$v_{2}$};
\draw (2.2,4.35) node[anchor=north west] {$v_{2}\cap v_{4}$};
\draw (3.3,4.35) node[anchor=north west] {$v_{4}$};

\begin{scriptsize}
\draw [fill=uuuuuu] (0,4) circle (0.1pt);
\draw [fill=uuuuuu] (4,4) circle (0.1pt);
\draw [fill=uuuuuu] (4,0) circle (0.1pt);
\draw [fill=uuuuuu] (0,0) circle (0.1pt);
\draw [fill=uuuuuu] (1.2,1.6) circle (0.8pt); %2
\draw [fill=uuuuuu] (2.3,1.45) circle (0.8pt);%1
\draw [fill=uuuuuu] (1.6,1.6) circle (0.8pt);%3
\draw [fill=uuuuuu] (1.2,1.2) circle (0.8pt); %4
\draw [fill=uuuuuu] (1.6,1.2) circle (0.8pt); %5
\draw [fill=uuuuuu] (2.3,2.45) circle (0.8pt); %6
\draw [fill=uuuuuu] (3.3,1.45) circle (0.8pt); %7
\draw [fill=uuuuuu] (3.3,2.45) circle (0.8pt); %8
\draw [fill=uuuuuu] (2.3,0.45) circle (0.8pt); %9
\draw [fill=uuuuuu] (2.3,3.45) circle (0.8pt); %10
\draw [fill=uuuuuu] (3.3,0.45) circle (0.8pt); %11
\draw [fill=uuuuuu] (3.3,3.45) circle (0.8pt); %12

\end{scriptsize}
\end{tikzpicture}
%\caption{}\label{}
\end{figure}

\section*{Acknowledgments}
For financial support, the first author thanks National Council for Scientific and Technological Development - CNPq (grant 406932/2023-9); the second author thanks  FAPESP (grant 2018/23690-6), CNPq (grants 307593/2023-1 and 406932/2023-9) and in accordance with the state task of the IM SB RAS, project FWNF-2022-003; and the third author thanks FAPESP (grant 2021/12820-9) and CNPq (grant 406932/2023-9). The authors are grateful to the anonymous referees for their valuable suggestions.

\newpage

\appendix

\newpage
\section{Diagrams: Minimal representations of code loops of rank $4$} \label{diagrams}
%%%V1
\begin{figure}[h!]
\centering
\definecolor{xdxdff}{rgb}{0.49019607843137253,0.49019607843137253,1}
\begin{tikzpicture}[line cap=round,line join=round,>=triangle 45,x=0.7cm,y=0.7cm, scale=0.9]
%\begin{tikzpicture}[>=latex',join=bevel,scale=0.5]

\clip(-0.9,-4) rectangle (19,4.3);
\draw [line width=0.6pt] (1,0)-- (8,0);
\begin{scriptsize}
\draw [fill=xdxdff] (1,0) circle (0.6pt);
\draw[color=xdxdff] (1,-0.1) node[below] {$1$};
\draw [fill=xdxdff] (2,0) circle (0.6pt);
\draw[color=xdxdff] (2,-0.1) node[below] {$2$};
\draw [fill=xdxdff] (3,0) circle (0.6pt);
\draw[color=xdxdff] (3,-0.1) node[below] {$3$};
\draw [fill=xdxdff] (4,0) circle (0.6pt);
\draw[color=xdxdff] (4,-0.1) node[below] {$4$};
\draw [fill=xdxdff] (5,0) circle (0.6pt);
\draw[color=xdxdff] (5,-0.1) node[below] {$5$};
\draw [fill=xdxdff] (6,0) circle (0.6pt);
\draw[color=xdxdff] (6,-0.1) node[below] {$6$};
\draw [fill=xdxdff] (7,0) circle (0.6pt);
\draw[color=xdxdff] (7,-0.1) node[below] {$7$};
\draw [fill=xdxdff] (8,0) circle (0.6pt);
\draw[color=xdxdff] (8,-0.1) node[below] {$8$};
% \draw [fill=xdxdff] (9,0) circle (0.6pt);
% \draw[color=xdxdff] (9,-0.1) node[below] {$9$};
% \draw [fill=xdxdff] (10,0) circle (0.6pt);
% \draw[color=xdxdff] (10,-0.1) node[below] {$10$};
% \draw [fill=xdxdff] (11,0) circle (0.6pt);
% \draw[color=xdxdff] (11,-0.1) node[below] {$11$};	
% \draw [fill=xdxdff] (12,0) circle (0.6pt);
% \draw[color=xdxdff] (12,-0.1) node[below] {$12$};
% %v1
\draw [line width=0.6pt] (1,4)-- (4,4);
\node  at (3bp,80bp) {\normalsize $v_1$};
%node[pos=-0.1, left] {\normalsize $v_1$}; Outro modo de label,  nao esquecer de tirar ; anterior.
%		\draw [fill=xdxdff] (1,1) circle (0.5pt);
%		\draw[color=xdxdff] (1,-0.1) node[below] {};
%		\draw [fill=xdxdff] (2,1) circle (0.5pt);
%		\draw[color=xdxdff] (2,-0.1) node[below] {};
%		%v2
\draw [line width=0.6pt] (1,3)-- (2,3);
\draw [line width=0.6pt] (5,3)-- (6,3);
\node  at (3bp,60bp) {\normalsize $v_2$};
%v3
\draw [fill=xdxdff] (1,2) circle (0.6pt);
\draw [fill=xdxdff] (3,2) circle (0.6pt);
\draw [fill=xdxdff] (5,2) circle (0.6pt);
\draw [fill=xdxdff] (7,2) circle (0.6pt);
\node  at (3bp,40bp) {\normalsize $v_3$};
%v4
\draw [line width=0.6pt] (1,1)-- (8,1);
%\draw [fill=xdxdff] (1,1) circle (0.6pt);
%\draw [line width=0.6pt] (6,1)-- (12,1);
\node  at (3bp,20bp) {\normalsize $v_4$};
\node (7) at (80bp,-11.0bp) [label=270:{{{\normalsize $V_{1}$}}}] {};

\end{scriptsize}
\end{tikzpicture}\vspace{-2.5cm}
\end{figure}

\vspace{0.2in}
%%%%%V2
\begin{figure}[h!]
\centering
\definecolor{xdxdff}{rgb}{0.49019607843137253,0.49019607843137253,1}
\begin{tikzpicture}[line cap=round,line join=round,>=triangle 45,x=0.7cm,y=0.7cm, scale=0.9]
%\begin{tikzpicture}[>=latex',join=bevel,scale=0.5]

\clip(-0.9,-4) rectangle (19,4.3);
\draw [line width=0.6pt] (1,0)-- (14,0);
\begin{scriptsize}
\draw [fill=xdxdff] (1,0) circle (0.6pt);
\draw[color=xdxdff] (1,-0.1) node[below] {$1$};
\draw [fill=xdxdff] (2,0) circle (0.6pt);
\draw[color=xdxdff] (2,-0.1) node[below] {$2$};
\draw [fill=xdxdff] (3,0) circle (0.6pt);
\draw[color=xdxdff] (3,-0.1) node[below] {$3$};
\draw [fill=xdxdff] (4,0) circle (0.6pt);
\draw[color=xdxdff] (4,-0.1) node[below] {$4$};
\draw [fill=xdxdff] (5,0) circle (0.6pt);
\draw[color=xdxdff] (5,-0.1) node[below] {$5$};
\draw [fill=xdxdff] (6,0) circle (0.6pt);
\draw[color=xdxdff] (6,-0.1) node[below] {$6$};
\draw [fill=xdxdff] (7,0) circle (0.6pt);
\draw[color=xdxdff] (7,-0.1) node[below] {$7$};
\draw [fill=xdxdff] (8,0) circle (0.6pt);
\draw[color=xdxdff] (8,-0.1) node[below] {$8$};
\draw [fill=xdxdff] (9,0) circle (0.6pt);
\draw[color=xdxdff] (9,-0.1) node[below] {$9$};
\draw [fill=xdxdff] (10,0) circle (0.6pt);
\draw[color=xdxdff] (10,-0.1) node[below] {$10$};
\draw [fill=xdxdff] (11,0) circle (0.6pt);
\draw[color=xdxdff] (11,-0.1) node[below] {$11$};	
\draw [fill=xdxdff] (12,0) circle (0.6pt);
\draw[color=xdxdff] (12,-0.1) node[below] {$12$};
\draw [fill=xdxdff] (14,0) circle (0.6pt);
\draw[color=xdxdff] (14,-0.1) node[below] {$14$};
\draw [fill=xdxdff] (13,0) circle (0.6pt);
\draw[color=xdxdff] (13,-0.1) node[below] {$13$};
%v1
\draw [line width=0.6pt] (1,4)-- (8,4);
\node  at (3bp,80bp) {\normalsize $v_1$};
%node[pos=-0.1, left] {\normalsize $v_1$}; Outro modo de label,  nao esquecer de tirar ; anterior.
%		\draw [fill=xdxdff] (1,1) circle (0.5pt);
%		\draw[color=xdxdff] (1,-0.1) node[below] {};
%		\draw [fill=xdxdff] (2,1) circle (0.5pt);
%		\draw[color=xdxdff] (2,-0.1) node[below] {};
%		%v2
\draw [line width=0.6pt] (1,3)-- (4,3);
\draw [line width=0.6pt] (9,3)-- (12,3);
\node  at (3bp,60bp) {\normalsize $v_2$};
%v3
\draw [line width=0.6pt] (5,2)-- (7,2);
\draw [line width=0.6pt] (9,2)-- (11,2);
%\draw [line width=0.6pt] (9,2)-- (9,2);
%\draw [line width=0.6pt] (11,2)-- (11,2);
\draw [fill=xdxdff] (1,2) circle (0.6pt);
\draw [fill=xdxdff] (13,2) circle (0.6pt);
%\draw [fill=xdxdff] (11,2) circle (0.6pt);
\node  at (3bp,40bp) {\normalsize $v_3$};
%v4
\draw [line width=0.6pt] (1,1)-- (2,1);
\draw [fill=xdxdff] (5,1) circle (0.6pt);
\draw [line width=0.6pt] (8,1)-- (9,1);
\draw [fill=xdxdff] (14,1) circle (0.6pt);
\draw [line width=0.6pt] (12,1)-- (13,1);
\node  at (3bp,20bp) {\normalsize $v_4$};
\node (7) at (140bp,-11.0bp) [label=270:{{{\normalsize $V_{2}$}}}] {};

\end{scriptsize}
\end{tikzpicture}\vspace{-2.5cm}

\end{figure}
\vspace{0.2in}

%%%V3
\begin{figure}[h!]%escala estava 0.9
\centering
\definecolor{xdxdff}{rgb}{0.49019607843137253,0.49019607843137253,1}
\begin{tikzpicture}[line cap=round,line join=round,>=triangle 45,x=0.7cm,y=0.7cm, scale=0.9]
%\begin{tikzpicture}[>=latex',join=bevel,scale=0.5]

\clip(-0.9,-4) rectangle (19,4.3);
\draw [line width=0.6pt] (1,0)-- (12,0);
\begin{scriptsize}
\draw [fill=xdxdff] (1,0) circle (0.6pt);
\draw[color=xdxdff] (1,-0.1) node[below] {$1$};
\draw [fill=xdxdff] (2,0) circle (0.6pt);
\draw[color=xdxdff] (2,-0.1) node[below] {$2$};
\draw [fill=xdxdff] (3,0) circle (0.6pt);
\draw[color=xdxdff] (3,-0.1) node[below] {$3$};
\draw [fill=xdxdff] (4,0) circle (0.6pt);
\draw[color=xdxdff] (4,-0.1) node[below] {$4$};
\draw [fill=xdxdff] (5,0) circle (0.6pt);
\draw[color=xdxdff] (5,-0.1) node[below] {$5$};
\draw [fill=xdxdff] (6,0) circle (0.6pt);
\draw[color=xdxdff] (6,-0.1) node[below] {$6$};
\draw [fill=xdxdff] (7,0) circle (0.6pt);
\draw[color=xdxdff] (7,-0.1) node[below] {$7$};
\draw [fill=xdxdff] (8,0) circle (0.6pt);
\draw[color=xdxdff] (8,-0.1) node[below] {$8$};
\draw [fill=xdxdff] (9,0) circle (0.6pt);
\draw[color=xdxdff] (9,-0.1) node[below] {$9$};
\draw [fill=xdxdff] (10,0) circle (0.6pt);
\draw[color=xdxdff] (10,-0.1) node[below] {$10$};
\draw [fill=xdxdff] (11,0) circle (0.6pt);
\draw[color=xdxdff] (11,-0.1) node[below] {$11$};	
\draw [fill=xdxdff] (12,0) circle (0.6pt);
\draw[color=xdxdff] (12,-0.1) node[below] {$12$};
%v1
\draw [line width=0.6pt] (1,4)-- (8,4);
\node  at (3bp,80bp) {\normalsize $v_1$};
%node[pos=-0.1, left] {\normalsize $v_1$}; Outro modo de label,  nao esquecer de tirar ; anterior.
%		\draw [fill=xdxdff] (1,1) circle (0.5pt);
%		\draw[color=xdxdff] (1,-0.1) node[below] {};
%		\draw [fill=xdxdff] (2,1) circle (0.5pt);
%		\draw[color=xdxdff] (2,-0.1) node[below] {};
%		%v2
\draw [line width=0.6pt] (1,3)-- (6,3);
\draw [line width=0.6pt] (9,3)-- (10,3);
\node  at (3bp,60bp) {\normalsize $v_2$};
%v3
\draw [line width=0.6pt] (1,2)-- (5,2);
\draw [line width=0.6pt] (7,2)-- (7,2);
\draw [line width=0.6pt] (9,2)-- (9,2);
\draw [line width=0.6pt] (11,2)-- (11,2);
\draw [fill=xdxdff] (7,2) circle (0.6pt);
\draw [fill=xdxdff] (9,2) circle (0.6pt);
\draw [fill=xdxdff] (11,2) circle (0.6pt);
\node  at (3bp,40bp) {\normalsize $v_3$};
%v4
%\draw [line width=0.6pt] (1,1)-- (1,1);
\draw [fill=xdxdff] (1,1) circle (0.6pt);
\draw [line width=0.6pt] (6,1)-- (12,1);
\node  at (3bp,20bp) {\normalsize $v_4$};
\node (7) at (120bp,-11.0bp) [label=270:{{{\normalsize $V_{3}$}}}] {};

\end{scriptsize}
\end{tikzpicture}\vspace{-2.5cm}

\end{figure}
\vspace{0.2in}
%%%V4
\begin{figure}[h!]%escala estava 0.9
\centering
\definecolor{xdxdff}{rgb}{0.49019607843137253,0.49019607843137253,1}
\begin{tikzpicture}[line cap=round,line join=round,>=triangle 45,x=0.7cm,y=0.7cm, scale=0.9]
%\begin{tikzpicture}[>=latex',join=bevel,scale=0.5]

\clip(-0.9,-4) rectangle (19,4.3);
\draw [line width=0.6pt] (1,0)-- (18,0);
\begin{scriptsize}
\draw [fill=xdxdff] (1,0) circle (0.6pt);
\draw[color=xdxdff] (1,-0.1) node[below] {$1$};
\draw [fill=xdxdff] (2,0) circle (0.6pt);
\draw[color=xdxdff] (2,-0.1) node[below] {$2$};
\draw [fill=xdxdff] (3,0) circle (0.6pt);
\draw[color=xdxdff] (3,-0.1) node[below] {$3$};
\draw [fill=xdxdff] (4,0) circle (0.6pt);
\draw[color=xdxdff] (4,-0.1) node[below] {$4$};
\draw [fill=xdxdff] (5,0) circle (0.6pt);
\draw[color=xdxdff] (5,-0.1) node[below] {$5$};
\draw [fill=xdxdff] (6,0) circle (0.6pt);
\draw[color=xdxdff] (6,-0.1) node[below] {$6$};
\draw [fill=xdxdff] (7,0) circle (0.6pt);
\draw[color=xdxdff] (7,-0.1) node[below] {$7$};
\draw [fill=xdxdff] (8,0) circle (0.6pt);
\draw[color=xdxdff] (8,-0.1) node[below] {$8$};
\draw [fill=xdxdff] (9,0) circle (0.6pt);
\draw[color=xdxdff] (9,-0.1) node[below] {$9$};
\draw [fill=xdxdff] (10,0) circle (0.6pt);
\draw[color=xdxdff] (10,-0.1) node[below] {$10$};
\draw [fill=xdxdff] (11,0) circle (0.6pt);
\draw[color=xdxdff] (11,-0.1) node[below] {$11$};	
\draw [fill=xdxdff] (12,0) circle (0.6pt);
\draw[color=xdxdff] (12,-0.1) node[below] {$12$};
\draw [fill=xdxdff] (13,0) circle (0.6pt);
\draw[color=xdxdff] (13,-0.1) node[below] {$13$};
\draw [fill=xdxdff] (14,0) circle (0.6pt);
\draw[color=xdxdff] (14,-0.1) node[below] {$14$};
\draw [fill=xdxdff] (15,0) circle (0.6pt);
\draw[color=xdxdff] (15,-0.1) node[below] {$15$};
\draw [fill=xdxdff] (16,0) circle (0.6pt);
\draw[color=xdxdff] (16,-0.1) node[below] {$16$};
\draw [fill=xdxdff] (17,0) circle (0.6pt);
\draw[color=xdxdff] (17,-0.1) node[below] {$17$};
\draw [fill=xdxdff] (18,0) circle (0.6pt);
\draw[color=xdxdff] (18,-0.1) node[below] {$18$};
%v1
\draw [line width=0.6pt] (1,4)-- (8,4);
\node  at (3bp,80bp) {\normalsize $v_1$};
%node[pos=-0.1, left] {\normalsize $v_1$}; Outro modo de label,  nao esquecer de tirar ; anterior.
%		\draw [fill=xdxdff] (1,1) circle (0.5pt);
%		\draw[color=xdxdff] (1,-0.1) node[below] {};
%		\draw [fill=xdxdff] (2,1) circle (0.5pt);
%		\draw[color=xdxdff] (2,-0.1) node[below] {};
%		%v2
\draw [line width=0.6pt] (1,3)-- (6,3);
\draw [line width=0.6pt] (9,3)-- (10,3);
\node  at (3bp,60bp) {\normalsize $v_2$};
%v3
\draw [line width=0.6pt] (1,2)-- (3,2);
%\draw [line width=0.6pt] (7,2)-- (7,2);
%\draw [line width=0.6pt] (9,2)-- (9,2);
\draw [line width=0.6pt] (11,2)-- (17,2);
\draw [fill=xdxdff] (7,2) circle (0.6pt);
\draw [fill=xdxdff] (9,2) circle (0.6pt);
%\draw [fill=xdxdff] (11,2) circle (0.6pt);
\node  at (3bp,40bp) {\normalsize $v_3$};
%v4
\draw [fill=xdxdff] (1,1) circle (0.6pt);
\draw [fill=xdxdff] (4,1) circle (0.6pt);
\draw [fill=xdxdff] (18,1) circle (0.6pt);
\draw [line width=0.6pt] (7,1)-- (11,1);
\node  at (3bp,20bp) {\normalsize $v_4$};
\node (7) at (180bp,-11.0bp) [label=270:{{{\normalsize $V_{4}$}}}] {};

\end{scriptsize}
\end{tikzpicture}\vspace{-2cm}

\end{figure}

\vspace{0.2in}
%%V5
\begin{figure}[h!]
\centering
\definecolor{xdxdff}{rgb}{0.49019607843137253,0.49019607843137253,1}
\begin{tikzpicture}[line cap=round,line join=round,>=triangle 45,x=0.7cm,y=0.7cm, scale=0.9]
%\begin{tikzpicture}[>=latex',join=bevel,scale=0.5]

\clip(-0.9,-4) rectangle (19,4.3);
\draw [line width=0.6pt] (1,0)-- (18,0);
\begin{scriptsize}
\draw [fill=xdxdff] (1,0) circle (0.6pt);
\draw[color=xdxdff] (1,-0.1) node[below] {$1$};
\draw [fill=xdxdff] (2,0) circle (0.6pt);
\draw[color=xdxdff] (2,-0.1) node[below] {$2$};
\draw [fill=xdxdff] (3,0) circle (0.6pt);
\draw[color=xdxdff] (3,-0.1) node[below] {$3$};
\draw [fill=xdxdff] (4,0) circle (0.6pt);
\draw[color=xdxdff] (4,-0.1) node[below] {$4$};
\draw [fill=xdxdff] (5,0) circle (0.6pt);
\draw[color=xdxdff] (5,-0.1) node[below] {$5$};
\draw [fill=xdxdff] (6,0) circle (0.6pt);
\draw[color=xdxdff] (6,-0.1) node[below] {$6$};
\draw [fill=xdxdff] (7,0) circle (0.6pt);
\draw[color=xdxdff] (7,-0.1) node[below] {$7$};
\draw [fill=xdxdff] (8,0) circle (0.6pt);
\draw[color=xdxdff] (8,-0.1) node[below] {$8$};
\draw [fill=xdxdff] (9,0) circle (0.6pt);
\draw[color=xdxdff] (9,-0.1) node[below] {$9$};
\draw [fill=xdxdff] (10,0) circle (0.6pt);
\draw[color=xdxdff] (10,-0.1) node[below] {$10$};
\draw [fill=xdxdff] (11,0) circle (0.6pt);
\draw[color=xdxdff] (11,-0.1) node[below] {$11$};	
\draw [fill=xdxdff] (12,0) circle (0.6pt);
\draw[color=xdxdff] (12,-0.1) node[below] {$12$};
\draw [fill=xdxdff] (13,0) circle (0.6pt);
\draw[color=xdxdff] (13,-0.1) node[below] {$13$};
\draw [fill=xdxdff] (14,0) circle (0.6pt);
\draw[color=xdxdff] (14,-0.1) node[below] {$14$};
\draw [fill=xdxdff] (15,0) circle (0.6pt);
\draw[color=xdxdff] (15,-0.1) node[below] {$15$};
\draw [fill=xdxdff] (16,0) circle (0.6pt);
\draw[color=xdxdff] (16,-0.1) node[below] {$16$};
\draw [fill=xdxdff] (17,0) circle (0.6pt);
\draw[color=xdxdff] (17,-0.1) node[below] {$17$};
\draw [fill=xdxdff] (18,0) circle (0.6pt);
\draw[color=xdxdff] (18,-0.1) node[below] {$18$};

%v1
\draw [line width=0.6pt] (1,4)-- (8,4);
\node  at (3bp,80bp) {\normalsize $v_1$};
%node[pos=-0.1, left] {\normalsize $v_1$}; Outro modo de label,  nao esquecer de tirar ; anterior.
%		\draw [fill=xdxdff] (1,1) circle (0.5pt);
%		\draw[color=xdxdff] (1,-0.1) node[below] {};
%		\draw [fill=xdxdff] (2,1) circle (0.5pt);
%		\draw[color=xdxdff] (2,-0.1) node[below] {};
%		%v2
\draw [line width=0.6pt] (1,3)-- (4,3);
\draw [line width=0.6pt] (9,3)-- (12,3);
\node  at (3bp,60bp) {\normalsize $v_2$};
%v3
\draw [line width=0.6pt] (13,2)-- (17,2);
\draw [fill=xdxdff] (1,2) circle (0.6pt);
\draw [fill=xdxdff] (5,2) circle (0.6pt);
\draw [fill=xdxdff] (9,2) circle (0.6pt);
\node  at (3bp,40bp) {\normalsize $v_3$};
%v4
\draw [line width=0.6pt] (1,1)-- (2,1);
\draw [line width=0.6pt] (5,1)-- (6,1);
\draw [line width=0.6pt] (9,1)-- (10,1);			
\draw [fill=xdxdff] (13,1) circle (0.6pt);
\draw [fill=xdxdff] (18,1) circle (0.6pt);
\node  at (3bp,20bp) {\normalsize $v_4$};
\node (7) at (180bp,-11.0bp) [label=270:{{{\normalsize $V_{5}$}}}] {};

\end{scriptsize}
\end{tikzpicture}\vspace{-2cm}

\end{figure}
\vspace{1.1in}

%%V6
\begin{figure}[h!]
\centering
\definecolor{xdxdff}{rgb}{0.49019607843137253,0.49019607843137253,1}
\begin{tikzpicture}[line cap=round,line join=round,>=triangle 45,x=0.7cm,y=0.7cm, scale=0.9]
%\begin{tikzpicture}[>=latex',join=bevel,scale=0.5]

\clip(-0.9,-4) rectangle (19,4.3);
\draw [line width=0.6pt] (1,0)-- (11,0);
\begin{scriptsize}
\draw [fill=xdxdff] (1,0) circle (0.6pt);
\draw[color=xdxdff] (1,-0.1) node[below] {$1$};
\draw [fill=xdxdff] (2,0) circle (0.6pt);
\draw[color=xdxdff] (2,-0.1) node[below] {$2$};
\draw [fill=xdxdff] (3,0) circle (0.6pt);
\draw[color=xdxdff] (3,-0.1) node[below] {$3$};
\draw [fill=xdxdff] (4,0) circle (0.6pt);
\draw[color=xdxdff] (4,-0.1) node[below] {$4$};
\draw [fill=xdxdff] (5,0) circle (0.6pt);
\draw[color=xdxdff] (5,-0.1) node[below] {$5$};
\draw [fill=xdxdff] (6,0) circle (0.6pt);
\draw[color=xdxdff] (6,-0.1) node[below] {$6$};
\draw [fill=xdxdff] (7,0) circle (0.6pt);
\draw[color=xdxdff] (7,-0.1) node[below] {$7$};
\draw [fill=xdxdff] (8,0) circle (0.6pt);
\draw[color=xdxdff] (8,-0.1) node[below] {$8$};
\draw [fill=xdxdff] (9,0) circle (0.6pt);
\draw[color=xdxdff] (9,-0.1) node[below] {$9$};
\draw [fill=xdxdff] (10,0) circle (0.6pt);
\draw[color=xdxdff] (10,-0.1) node[below] {$10$};
\draw [fill=xdxdff] (11,0) circle (0.6pt);
\draw[color=xdxdff] (11,-0.1) node[below] {$11$};	
%v1
\draw [line width=0.6pt] (1,4)-- (4,4);
\node  at (3bp,80bp) {\normalsize $v_1$};
%node[pos=-0.1, left] {\normalsize $v_1$}; Outro modo de label,  nao esquecer de tirar ; anterior.
%		\draw [fill=xdxdff] (1,1) circle (0.5pt);
%		\draw[color=xdxdff] (1,-0.1) node[below] {};
%		\draw [fill=xdxdff] (2,1) circle (0.5pt);
%		\draw[color=xdxdff] (2,-0.1) node[below] {};
%		%v2
\draw [line width=0.6pt] (1,3)-- (2,3);
\draw [line width=0.6pt] (5,3)-- (6,3);
\node  at (3bp,60bp) {\normalsize $v_2$};
%v3
\draw [fill=xdxdff] (1,2) circle (0.6pt);
\draw [fill=xdxdff] (3,2) circle (0.6pt);
\draw [fill=xdxdff] (5,2) circle (0.6pt);
\draw [fill=xdxdff] (7,2) circle (0.6pt);
\node  at (3bp,40bp) {\normalsize $v_3$};
%v4
\draw [line width=0.6pt] (8,1)-- (11,1);
%\draw [fill=xdxdff] (1,1) circle (0.6pt);
%\draw [line width=0.6pt] (6,1)-- (12,1);
\node  at (3bp,20bp) {\normalsize $v_4$};
\node (7) at (110bp,-11.0bp) [label=270:{{{\normalsize $V_{6}$}}}] {};

\end{scriptsize}
\end{tikzpicture}\vspace{-2cm}

\end{figure}

\vspace{1.5in}

%%%V7

\begin{figure}[h!]%escala estava 0.9
\centering
\definecolor{xdxdff}{rgb}{0.49019607843137253,0.49019607843137253,1}
\begin{tikzpicture}[line cap=round,line join=round,>=triangle 45,x=0.7cm,y=0.7cm, scale=0.9]
%\begin{tikzpicture}[>=latex',join=bevel,scale=0.5]

\clip(-0.9,-4) rectangle (19,4.3);
\draw [line width=0.6pt] (1,0)-- (16,0);
\begin{scriptsize}
\draw [fill=xdxdff] (1,0) circle (0.6pt);
\draw[color=xdxdff] (1,-0.1) node[below] {$1$};
\draw [fill=xdxdff] (2,0) circle (0.6pt);
\draw[color=xdxdff] (2,-0.1) node[below] {$2$};
\draw [fill=xdxdff] (3,0) circle (0.6pt);
\draw[color=xdxdff] (3,-0.1) node[below] {$3$};
\draw [fill=xdxdff] (4,0) circle (0.6pt);
\draw[color=xdxdff] (4,-0.1) node[below] {$4$};
\draw [fill=xdxdff] (5,0) circle (0.6pt);
\draw[color=xdxdff] (5,-0.1) node[below] {$5$};
\draw [fill=xdxdff] (6,0) circle (0.6pt);
\draw[color=xdxdff] (6,-0.1) node[below] {$6$};
\draw [fill=xdxdff] (7,0) circle (0.6pt);
\draw[color=xdxdff] (7,-0.1) node[below] {$7$};
\draw [fill=xdxdff] (8,0) circle (0.6pt);
\draw[color=xdxdff] (8,-0.1) node[below] {$8$};
\draw [fill=xdxdff] (9,0) circle (0.6pt);
\draw[color=xdxdff] (9,-0.1) node[below] {$9$};
\draw [fill=xdxdff] (10,0) circle (0.6pt);
\draw[color=xdxdff] (10,-0.1) node[below] {$10$};
\draw [fill=xdxdff] (11,0) circle (0.6pt);
\draw[color=xdxdff] (11,-0.1) node[below] {$11$};	
\draw [fill=xdxdff] (12,0) circle (0.6pt);
\draw[color=xdxdff] (12,-0.1) node[below] {$12$};
\draw [fill=xdxdff] (13,0) circle (0.6pt);
\draw[color=xdxdff] (13,-0.1) node[below] {$13$};
\draw [fill=xdxdff] (14,0) circle (0.6pt);
\draw[color=xdxdff] (14,-0.1) node[below] {$14$};
\draw [fill=xdxdff] (15,0) circle (0.6pt);
\draw[color=xdxdff] (15,-0.1) node[below] {$15$};
\draw [fill=xdxdff] (16,0) circle (0.6pt);
\draw[color=xdxdff] (16,-0.1) node[below] {$16$};
%v1
\draw [line width=0.6pt] (1,4)-- (8,4);
\node  at (3bp,80bp) {\normalsize $v_1$};
%node[pos=-0.1, left] {\normalsize $v_1$}; Outro modo de label,  nao esquecer de tirar ; anterior.
%		\draw [fill=xdxdff] (1,1) circle (0.5pt);
%		\draw[color=xdxdff] (1,-0.1) node[below] {};
%		\draw [fill=xdxdff] (2,1) circle (0.5pt);
%		\draw[color=xdxdff] (2,-0.1) node[below] {};
%		%v2
\draw [line width=0.6pt] (1,3)-- (4,3);
\draw [line width=0.6pt] (9,3)-- (12,3);
\node  at (3bp,60bp) {\normalsize $v_2$};
%v3
\draw [line width=0.6pt] (1,2)-- (3,2);
\draw [line width=0.6pt] (13,2)-- (15,2);
\draw [fill=xdxdff] (5,2) circle (0.6pt);
\draw [fill=xdxdff] (9,2) circle (0.6pt);
%\draw [fill=xdxdff] (11,2) circle (0.6pt);
\node  at (3bp,40bp) {\normalsize $v_3$};
%v4
\draw [line width=0.6pt] (1,1)-- (16,1);
\node  at (3bp,20bp) {\normalsize $v_4$};
\node (7) at (160bp,-11.0bp) [label=270:{{{\normalsize $V_{7}$}}}] {};

\end{scriptsize}
\end{tikzpicture}\vspace{-2cm}

\end{figure}

\vspace{0.2in}
%%%V8
\begin{figure}[h!]
\centering
\definecolor{xdxdff}{rgb}{0.49019607843137253,0.49019607843137253,1}
\begin{tikzpicture}[line cap=round,line join=round,>=triangle 45,x=0.7cm,y=0.7cm, scale=0.9]
%\begin{tikzpicture}[>=latex',join=bevel,scale=0.5]

\clip(-0.9,-4) rectangle (19,4.3);
\draw [line width=0.6pt] (1,0)-- (17,0);
\begin{scriptsize}
\draw [fill=xdxdff] (1,0) circle (0.6pt);
\draw[color=xdxdff] (1,-0.1) node[below] {$1$};
\draw [fill=xdxdff] (2,0) circle (0.6pt);
\draw[color=xdxdff] (2,-0.1) node[below] {$2$};
\draw [fill=xdxdff] (3,0) circle (0.6pt);
\draw[color=xdxdff] (3,-0.1) node[below] {$3$};
\draw [fill=xdxdff] (4,0) circle (0.6pt);
\draw[color=xdxdff] (4,-0.1) node[below] {$4$};
\draw [fill=xdxdff] (5,0) circle (0.6pt);
\draw[color=xdxdff] (5,-0.1) node[below] {$5$};
\draw [fill=xdxdff] (6,0) circle (0.6pt);
\draw[color=xdxdff] (6,-0.1) node[below] {$6$};
\draw [fill=xdxdff] (7,0) circle (0.6pt);
\draw[color=xdxdff] (7,-0.1) node[below] {$7$};
\draw [fill=xdxdff] (8,0) circle (0.6pt);
\draw[color=xdxdff] (8,-0.1) node[below] {$8$};
\draw [fill=xdxdff] (9,0) circle (0.6pt);
\draw[color=xdxdff] (9,-0.1) node[below] {$9$};
\draw [fill=xdxdff] (10,0) circle (0.6pt);
\draw[color=xdxdff] (10,-0.1) node[below] {$10$};
\draw [fill=xdxdff] (11,0) circle (0.6pt);
\draw[color=xdxdff] (11,-0.1) node[below] {$11$};	
\draw [fill=xdxdff] (12,0) circle (0.6pt);
\draw[color=xdxdff] (12,-0.1) node[below] {$12$};
\draw [fill=xdxdff] (13,0) circle (0.6pt);
\draw[color=xdxdff] (13,-0.1) node[below] {$13$};
\draw [fill=xdxdff] (14,0) circle (0.6pt);
\draw[color=xdxdff] (14,-0.1) node[below] {$14$};
\draw [fill=xdxdff] (15,0) circle (0.6pt);
\draw[color=xdxdff] (15,-0.1) node[below] {$15$};
\draw [fill=xdxdff] (16,0) circle (0.6pt);
\draw[color=xdxdff] (16,-0.1) node[below] {$16$};
\draw [fill=xdxdff] (17,0) circle (0.6pt);
\draw[color=xdxdff] (17,-0.1) node[below] {$17$};
%v1
\draw [line width=0.6pt] (1,4)-- (8,4);
\node  at (3bp,80bp) {\normalsize $v_1$};
%node[pos=-0.1, left] {\normalsize $v_1$}; Outro modo de label,  nao esquecer de tirar ; anterior.
%		\draw [fill=xdxdff] (1,1) circle (0.5pt);
%		\draw[color=xdxdff] (1,-0.1) node[below] {};
%		\draw [fill=xdxdff] (2,1) circle (0.5pt);
%		\draw[color=xdxdff] (2,-0.1) node[below] {};
%		%v2
\draw [line width=0.6pt] (1,3)-- (4,3);
\draw [line width=0.6pt] (9,3)-- (12,3);
\node  at (3bp,60bp) {\normalsize $v_2$};
%v3
\draw [line width=0.6pt] (1,2)-- (3,2);
\draw [line width=0.6pt] (13,2)-- (15,2);
\draw [fill=xdxdff] (5,2) circle (0.6pt);
\draw [fill=xdxdff] (9,2) circle (0.6pt);
\node  at (3bp,40bp) {\normalsize $v_3$};
%v4
\draw [line width=0.6pt] (1,1)-- (2,1);
\draw [line width=0.6pt] (1,1)-- (2,1);
\draw [line width=0.6pt] (10,1)-- (11,1);
\draw [line width=0.6pt] (13,1)-- (14,1);
\draw [line width=0.6pt] (16,1)-- (17,1);
\node  at (3bp,20bp) {\normalsize $v_4$};
\node (7) at (170bp,-11.0bp) [label=270:{{{\normalsize $V_{8}$}}}] {};

\end{scriptsize}
\end{tikzpicture}\vspace{-2cm}

\end{figure}

%%V9
\begin{figure}[h!]
\centering
\definecolor{xdxdff}{rgb}{0.49019607843137253,0.49019607843137253,1}
\begin{tikzpicture}[line cap=round,line join=round,>=triangle 45,x=0.7cm,y=0.7cm, scale=0.9]
%\begin{tikzpicture}[>=latex',join=bevel,scale=0.5]

\clip(-0.9,-4) rectangle (20,4.3);
\draw [line width=0.6pt] (1,0)-- (19,0);
\begin{scriptsize}
\draw [fill=xdxdff] (1,0) circle (0.6pt);
\draw[color=xdxdff] (1,-0.1) node[below] {$1$};
\draw [fill=xdxdff] (2,0) circle (0.6pt);
\draw[color=xdxdff] (2,-0.1) node[below] {$2$};
\draw [fill=xdxdff] (3,0) circle (0.6pt);
\draw[color=xdxdff] (3,-0.1) node[below] {$3$};
\draw [fill=xdxdff] (4,0) circle (0.6pt);
\draw[color=xdxdff] (4,-0.1) node[below] {$4$};
\draw [fill=xdxdff] (5,0) circle (0.6pt);
\draw[color=xdxdff] (5,-0.1) node[below] {$5$};
\draw [fill=xdxdff] (6,0) circle (0.6pt);
\draw[color=xdxdff] (6,-0.1) node[below] {$6$};
\draw [fill=xdxdff] (7,0) circle (0.6pt);
\draw[color=xdxdff] (7,-0.1) node[below] {$7$};
\draw [fill=xdxdff] (8,0) circle (0.6pt);
\draw[color=xdxdff] (8,-0.1) node[below] {$8$};
\draw [fill=xdxdff] (9,0) circle (0.6pt);
\draw[color=xdxdff] (9,-0.1) node[below] {$9$};
\draw [fill=xdxdff] (10,0) circle (0.6pt);
\draw[color=xdxdff] (10,-0.1) node[below] {$10$};
\draw [fill=xdxdff] (11,0) circle (0.6pt);
\draw[color=xdxdff] (11,-0.1) node[below] {$11$};	
\draw [fill=xdxdff] (12,0) circle (0.6pt);
\draw[color=xdxdff] (12,-0.1) node[below] {$12$};
\draw [fill=xdxdff] (13,0) circle (0.6pt);
\draw[color=xdxdff] (13,-0.1) node[below] {$13$};
\draw [fill=xdxdff] (14,0) circle (0.6pt);
\draw[color=xdxdff] (14,-0.1) node[below] {$14$};
\draw [fill=xdxdff] (15,0) circle (0.6pt);
\draw[color=xdxdff] (15,-0.1) node[below] {$15$};
\draw [fill=xdxdff] (16,0) circle (0.6pt);
\draw[color=xdxdff] (16,-0.1) node[below] {$16$};
\draw [fill=xdxdff] (17,0) circle (0.6pt);
\draw[color=xdxdff] (17,-0.1) node[below] {$17$};
\draw [fill=xdxdff] (18,0) circle (0.6pt);
\draw[color=xdxdff] (18,-0.1) node[below] {$18$};
\draw [fill=xdxdff] (19,0) circle (0.6pt);
\draw[color=xdxdff] (19,-0.1) node[below] {$19$};

%v1
\draw [line width=0.6pt] (1,4)-- (8,4);
\node  at (3bp,80bp) {\normalsize $v_1$};
%node[pos=-0.1, left] {\normalsize $v_1$}; Outro modo de label,  nao esquecer de tirar ; anterior.
%		\draw [fill=xdxdff] (1,1) circle (0.5pt);
%		\draw[color=xdxdff] (1,-0.1) node[below] {};
%		\draw [fill=xdxdff] (2,1) circle (0.5pt);
%		\draw[color=xdxdff] (2,-0.1) node[below] {};
%		%v2
\draw [line width=0.6pt] (1,3)-- (4,3);
\draw [line width=0.6pt] (9,3)-- (16,3);
\node  at (3bp,60bp) {\normalsize $v_2$};
%v3
\draw [line width=0.6pt] (5,2)-- (7,2);
\draw [line width=0.6pt] (9,2)-- (11,2);
\draw [fill=xdxdff] (1,2) circle (0.6pt);
\draw [fill=xdxdff] (17,2) circle (0.6pt);
\node  at (3bp,40bp) {\normalsize $v_3$};
%v4
\draw [line width=0.6pt] (5,1)-- (6,1);
\draw [line width=0.6pt] (9,1)-- (10,1);
\draw [line width=0.6pt] (12,1)-- (13,1);
\draw [line width=0.6pt] (18,1)-- (19,1);
\node  at (3bp,20bp) {\normalsize $v_4$};
\node (7) at (190bp,-11.0bp) [label=270:{{{\normalsize $V_{9}$}}}] {};

\end{scriptsize}
\end{tikzpicture}\vspace{-2cm}

\end{figure}

%%%V10
\begin{figure}[h!]%escala estava 0.9
\centering
\definecolor{xdxdff}{rgb}{0.49019607843137253,0.49019607843137253,1}
\begin{tikzpicture}[line cap=round,line join=round,>=triangle 45,x=0.7cm,y=0.7cm, scale=0.9]
%\begin{tikzpicture}[>=latex',join=bevel,scale=0.5]

\clip(-0.9,-4) rectangle (20,4.3);
\draw [line width=0.6pt] (1,0)-- (19,0);
\begin{scriptsize}
\draw [fill=xdxdff] (1,0) circle (0.6pt);
\draw[color=xdxdff] (1,-0.1) node[below] {$1$};
\draw [fill=xdxdff] (2,0) circle (0.6pt);
\draw[color=xdxdff] (2,-0.1) node[below] {$2$};
\draw [fill=xdxdff] (3,0) circle (0.6pt);
\draw[color=xdxdff] (3,-0.1) node[below] {$3$};
\draw [fill=xdxdff] (4,0) circle (0.6pt);
\draw[color=xdxdff] (4,-0.1) node[below] {$4$};
\draw [fill=xdxdff] (5,0) circle (0.6pt);
\draw[color=xdxdff] (5,-0.1) node[below] {$5$};
\draw [fill=xdxdff] (6,0) circle (0.6pt);
\draw[color=xdxdff] (6,-0.1) node[below] {$6$};
\draw [fill=xdxdff] (7,0) circle (0.6pt);
\draw[color=xdxdff] (7,-0.1) node[below] {$7$};
\draw [fill=xdxdff] (8,0) circle (0.6pt);
\draw[color=xdxdff] (8,-0.1) node[below] {$8$};
\draw [fill=xdxdff] (9,0) circle (0.6pt);
\draw[color=xdxdff] (9,-0.1) node[below] {$9$};
\draw [fill=xdxdff] (10,0) circle (0.6pt);
\draw[color=xdxdff] (10,-0.1) node[below] {$10$};
\draw [fill=xdxdff] (11,0) circle (0.6pt);
\draw[color=xdxdff] (11,-0.1) node[below] {$11$};	
\draw [fill=xdxdff] (12,0) circle (0.6pt);
\draw[color=xdxdff] (12,-0.1) node[below] {$12$};
\draw [fill=xdxdff] (13,0) circle (0.6pt);
\draw[color=xdxdff] (13,-0.1) node[below] {$13$};
\draw [fill=xdxdff] (14,0) circle (0.6pt);
\draw[color=xdxdff] (14,-0.1) node[below] {$14$};
\draw [fill=xdxdff] (15,0) circle (0.6pt);
\draw[color=xdxdff] (15,-0.1) node[below] {$15$};
\draw [fill=xdxdff] (16,0) circle (0.6pt);
\draw[color=xdxdff] (16,-0.1) node[below] {$16$};
\draw [fill=xdxdff] (17,0) circle (0.6pt);
\draw[color=xdxdff] (17,-0.1) node[below] {$17$};
\draw [fill=xdxdff] (18,0) circle (0.6pt);
\draw[color=xdxdff] (18,-0.1) node[below] {$18$};
\draw [fill=xdxdff] (19,0) circle (0.6pt);
\draw[color=xdxdff] (19,-0.1) node[below] {$19$};

%v1
\draw [line width=0.6pt] (1,4)-- (8,4);
\node  at (3bp,80bp) {\normalsize $v_1$};
%node[pos=-0.1, left] {\normalsize $v_1$}; Outro modo de label,  nao esquecer de tirar ; anterior.
%		\draw [fill=xdxdff] (1,1) circle (0.5pt);
%		\draw[color=xdxdff] (1,-0.1) node[below] {};
%		\draw [fill=xdxdff] (2,1) circle (0.5pt);
%		\draw[color=xdxdff] (2,-0.1) node[below] {};
%		%v2
\draw [line width=0.6pt] (1,3)-- (2,3);
\draw [line width=0.6pt] (9,3)-- (14,3);
\node  at (3bp,60bp) {\normalsize $v_2$};
%v3
\draw [line width=0.6pt] (9,2)-- (11,2);
\draw [line width=0.6pt] (15,2)-- (17,2);
\draw [fill=xdxdff] (1,2) circle (0.6pt);
\draw [fill=xdxdff] (3,2) circle (0.6pt);
\node  at (3bp,40bp) {\normalsize $v_3$};
%v4
\draw [line width=0.6pt] (4,1)-- (5,1);
\draw [line width=0.6pt] (18,1)-- (19,1);
\node  at (3bp,20bp) {\normalsize $v_4$};
\node (7) at (180bp,-11.0bp) [label=270:{{{\normalsize $V_{10}$}}}] {};

\end{scriptsize}
\end{tikzpicture}\vspace{-2cm}

\end{figure}
\vspace{1in}
%%%V11
\begin{figure}[h!]
\centering
\definecolor{xdxdff}{rgb}{0.49019607843137253,0.49019607843137253,1}
\begin{tikzpicture}[line cap=round,line join=round,>=triangle 45,x=0.7cm,y=0.7cm, scale=0.9]
%\begin{tikzpicture}[>=latex',join=bevel,scale=0.5]

\clip(-0.9,-4) rectangle (19,4.3);
\draw [line width=0.6pt] (1,0)-- (17,0);
\begin{scriptsize}
\draw [fill=xdxdff] (1,0) circle (0.6pt);
\draw[color=xdxdff] (1,-0.1) node[below] {$1$};
\draw [fill=xdxdff] (2,0) circle (0.6pt);
\draw[color=xdxdff] (2,-0.1) node[below] {$2$};
\draw [fill=xdxdff] (3,0) circle (0.6pt);
\draw[color=xdxdff] (3,-0.1) node[below] {$3$};
\draw [fill=xdxdff] (4,0) circle (0.6pt);
\draw[color=xdxdff] (4,-0.1) node[below] {$4$};
\draw [fill=xdxdff] (5,0) circle (0.6pt);
\draw[color=xdxdff] (5,-0.1) node[below] {$5$};
\draw [fill=xdxdff] (6,0) circle (0.6pt);
\draw[color=xdxdff] (6,-0.1) node[below] {$6$};
\draw [fill=xdxdff] (7,0) circle (0.6pt);
\draw[color=xdxdff] (7,-0.1) node[below] {$7$};
\draw [fill=xdxdff] (8,0) circle (0.6pt);
\draw[color=xdxdff] (8,-0.1) node[below] {$8$};
\draw [fill=xdxdff] (9,0) circle (0.6pt);
\draw[color=xdxdff] (9,-0.1) node[below] {$9$};
\draw [fill=xdxdff] (10,0) circle (0.6pt);
\draw[color=xdxdff] (10,-0.1) node[below] {$10$};
\draw [fill=xdxdff] (11,0) circle (0.6pt);
\draw[color=xdxdff] (11,-0.1) node[below] {$11$};	
\draw [fill=xdxdff] (12,0) circle (0.6pt);
\draw[color=xdxdff] (12,-0.1) node[below] {$12$};
\draw [fill=xdxdff] (13,0) circle (0.6pt);
\draw[color=xdxdff] (13,-0.1) node[below] {$13$};
\draw [fill=xdxdff] (14,0) circle (0.6pt);
\draw[color=xdxdff] (14,-0.1) node[below] {$14$};
\draw [fill=xdxdff] (15,0) circle (0.6pt);
\draw[color=xdxdff] (15,-0.1) node[below] {$15$};
\draw [fill=xdxdff] (16,0) circle (0.6pt);
\draw[color=xdxdff] (16,-0.1) node[below] {$16$};
\draw [fill=xdxdff] (17,0) circle (0.6pt);
\draw[color=xdxdff] (17,-0.1) node[below] {$17$};
%v1
\draw [line width=0.6pt] (1,4)-- (8,4);
\node  at (3bp,80bp) {\normalsize $v_1$};
%node[pos=-0.1, left] {\normalsize $v_1$}; Outro modo de label,  nao esquecer de tirar ; anterior.
%		\draw [fill=xdxdff] (1,1) circle (0.5pt);
%		\draw[color=xdxdff] (1,-0.1) node[below] {};
%		\draw [fill=xdxdff] (2,1) circle (0.5pt);
%		\draw[color=xdxdff] (2,-0.1) node[below] {};
%		%v2
\draw [line width=0.6pt] (1,3)-- (4,3);
\draw [line width=0.6pt] (9,3)-- (12,3);
\node  at (3bp,60bp) {\normalsize $v_2$};
%v3
\draw [line width=0.6pt] (1,2)-- (3,2);
\draw [line width=0.6pt] (13,2)-- (15,2);
\draw [fill=xdxdff] (5,2) circle (0.6pt);
\draw [fill=xdxdff] (9,2) circle (0.6pt);
\node  at (3bp,40bp) {\normalsize $v_3$};
%v4
\draw [line width=0.6pt] (6,1)-- (7,1);
\draw [line width=0.6pt] (16,1)-- (17,1);
\node  at (3bp,20bp) {\normalsize $v_4$};
\node (7) at (170bp,-11.0bp) [label=270:{{{\normalsize $V_{11}$}}}] {};

\end{scriptsize}
\end{tikzpicture}\vspace{-2cm}

\end{figure}

\vspace{1.3in}

%%%V12
\begin{figure}[h!]
\centering
\definecolor{xdxdff}{rgb}{0.49019607843137253,0.49019607843137253,1}
\begin{tikzpicture}[line cap=round,line join=round,>=triangle 45,x=0.7cm,y=0.7cm, scale=0.9]
%\begin{tikzpicture}[>=latex',join=bevel,scale=0.5]

\clip(-0.9,-4) rectangle (19,4.3);
\draw [line width=0.6pt] (1,0)-- (17,0);
\begin{scriptsize}
\draw [fill=xdxdff] (1,0) circle (0.6pt);
\draw[color=xdxdff] (1,-0.1) node[below] {$1$};
\draw [fill=xdxdff] (2,0) circle (0.6pt);
\draw[color=xdxdff] (2,-0.1) node[below] {$2$};
\draw [fill=xdxdff] (3,0) circle (0.6pt);
\draw[color=xdxdff] (3,-0.1) node[below] {$3$};
\draw [fill=xdxdff] (4,0) circle (0.6pt);
\draw[color=xdxdff] (4,-0.1) node[below] {$4$};
\draw [fill=xdxdff] (5,0) circle (0.6pt);
\draw[color=xdxdff] (5,-0.1) node[below] {$5$};
\draw [fill=xdxdff] (6,0) circle (0.6pt);
\draw[color=xdxdff] (6,-0.1) node[below] {$6$};
\draw [fill=xdxdff] (7,0) circle (0.6pt);
\draw[color=xdxdff] (7,-0.1) node[below] {$7$};
\draw [fill=xdxdff] (8,0) circle (0.6pt);
\draw[color=xdxdff] (8,-0.1) node[below] {$8$};
\draw [fill=xdxdff] (9,0) circle (0.6pt);
\draw[color=xdxdff] (9,-0.1) node[below] {$9$};
\draw [fill=xdxdff] (10,0) circle (0.6pt);
\draw[color=xdxdff] (10,-0.1) node[below] {$10$};
\draw [fill=xdxdff] (11,0) circle (0.6pt);
\draw[color=xdxdff] (11,-0.1) node[below] {$11$};	
\draw [fill=xdxdff] (12,0) circle (0.6pt);
\draw[color=xdxdff] (12,-0.1) node[below] {$12$};
\draw [fill=xdxdff] (13,0) circle (0.6pt);
\draw[color=xdxdff] (13,-0.1) node[below] {$13$};
\draw [fill=xdxdff] (14,0) circle (0.6pt);
\draw[color=xdxdff] (14,-0.1) node[below] {$14$};
\draw [fill=xdxdff] (15,0) circle (0.6pt);
\draw[color=xdxdff] (15,-0.1) node[below] {$15$};
\draw [fill=xdxdff] (16,0) circle (0.6pt);
\draw[color=xdxdff] (16,-0.1) node[below] {$16$};
\draw [fill=xdxdff] (17,0) circle (0.6pt);
\draw[color=xdxdff] (17,-0.1) node[below] {$17$};
%v1
\draw [line width=0.6pt] (1,4)-- (8,4);
\node  at (3bp,80bp) {\normalsize $v_1$};
%node[pos=-0.1, left] {\normalsize $v_1$}; Outro modo de label,  nao esquecer de tirar ; anterior.
%		\draw [fill=xdxdff] (1,1) circle (0.5pt);
%		\draw[color=xdxdff] (1,-0.1) node[below] {};
%		\draw [fill=xdxdff] (2,1) circle (0.5pt);
%		\draw[color=xdxdff] (2,-0.1) node[below] {};
%		%v2
\draw [line width=0.6pt] (1,3)-- (4,3);
\draw [line width=0.6pt] (9,3)-- (12,3);
\node  at (3bp,60bp) {\normalsize $v_2$};
%v3
\draw [line width=0.6pt] (1,2)-- (3,2);
\draw [line width=0.6pt] (9,2)-- (11,2);
\draw [fill=xdxdff] (5,2) circle (0.6pt);
\draw [fill=xdxdff] (13,2) circle (0.6pt);
\node  at (3bp,40bp) {\normalsize $v_3$};
%v4
\draw [line width=0.6pt] (1,1)-- (2,1);
\draw [line width=0.6pt] (9,1)-- (10,1);
\draw [line width=0.6pt] (14,1)-- (17,1);
\node  at (3bp,20bp) {\normalsize $v_4$};
\node (7) at (170bp,-11.0bp) [label=270:{{{\normalsize $V_{12}$}}}] {};

\end{scriptsize}
\end{tikzpicture}\vspace{-2cm}

\end{figure}
\vspace{0.2in}
%\newpage

%%%V13
\begin{figure}[h!]
\centering
\definecolor{xdxdff}{rgb}{0.49019607843137253,0.49019607843137253,1}
\begin{tikzpicture}[line cap=round,line join=round,>=triangle 45,x=0.7cm,y=0.7cm, scale=0.9]
%\begin{tikzpicture}[>=latex',join=bevel,scale=0.5]

\clip(-0.9,-4) rectangle (19,4.3);
\draw [line width=0.6pt] (1,0)-- (17,0);
\begin{scriptsize}
\draw [fill=xdxdff] (1,0) circle (0.6pt);
\draw[color=xdxdff] (1,-0.1) node[below] {$1$};
\draw [fill=xdxdff] (2,0) circle (0.6pt);
\draw[color=xdxdff] (2,-0.1) node[below] {$2$};
\draw [fill=xdxdff] (3,0) circle (0.6pt);
\draw[color=xdxdff] (3,-0.1) node[below] {$3$};
\draw [fill=xdxdff] (4,0) circle (0.6pt);
\draw[color=xdxdff] (4,-0.1) node[below] {$4$};
\draw [fill=xdxdff] (5,0) circle (0.6pt);
\draw[color=xdxdff] (5,-0.1) node[below] {$5$};
\draw [fill=xdxdff] (6,0) circle (0.6pt);
\draw[color=xdxdff] (6,-0.1) node[below] {$6$};
\draw [fill=xdxdff] (7,0) circle (0.6pt);
\draw[color=xdxdff] (7,-0.1) node[below] {$7$};
\draw [fill=xdxdff] (8,0) circle (0.6pt);
\draw[color=xdxdff] (8,-0.1) node[below] {$8$};
\draw [fill=xdxdff] (9,0) circle (0.6pt);
\draw[color=xdxdff] (9,-0.1) node[below] {$9$};
\draw [fill=xdxdff] (10,0) circle (0.6pt);
\draw[color=xdxdff] (10,-0.1) node[below] {$10$};
\draw [fill=xdxdff] (11,0) circle (0.6pt);
\draw[color=xdxdff] (11,-0.1) node[below] {$11$};	
\draw [fill=xdxdff] (12,0) circle (0.6pt);
\draw[color=xdxdff] (12,-0.1) node[below] {$12$};
\draw [fill=xdxdff] (13,0) circle (0.6pt);
\draw[color=xdxdff] (13,-0.1) node[below] {$13$};
\draw [fill=xdxdff] (14,0) circle (0.6pt);
\draw[color=xdxdff] (14,-0.1) node[below] {$14$};
\draw [fill=xdxdff] (15,0) circle (0.6pt);
\draw[color=xdxdff] (15,-0.1) node[below] {$15$};
\draw [fill=xdxdff] (16,0) circle (0.6pt);
\draw[color=xdxdff] (16,-0.1) node[below] {$16$};
\draw [fill=xdxdff] (17,0) circle (0.6pt);
\draw[color=xdxdff] (17,-0.1) node[below] {$17$};
%v1
\draw [line width=0.6pt] (1,4)-- (8,4);
\node  at (3bp,80bp) {\normalsize $v_1$};
%node[pos=-0.1, left] {\normalsize $v_1$}; Outro modo de label,  nao esquecer de tirar ; anterior.
%		\draw [fill=xdxdff] (1,1) circle (0.5pt);
%		\draw[color=xdxdff] (1,-0.1) node[below] {};
%		\draw [fill=xdxdff] (2,1) circle (0.5pt);
%		\draw[color=xdxdff] (2,-0.1) node[below] {};
%		%v2
\draw [line width=0.6pt] (1,3)-- (2,3);
\draw [line width=0.6pt] (9,3)-- (10,3);
\node  at (3bp,60bp) {\normalsize $v_2$};
%v3
\draw [fill=xdxdff] (1,2) circle (0.6pt);
\draw [fill=xdxdff] (3,2) circle (0.6pt);
\draw [fill=xdxdff] (9,2) circle (0.6pt);
\draw [fill=xdxdff] (11,2) circle (0.6pt);
\node  at (3bp,40bp) {\normalsize $v_3$};
%v4
\draw [line width=0.6pt] (4,1)-- (5,1);
\draw [line width=0.6pt] (12,1)-- (17,1);
\node  at (3bp,20bp) {\normalsize $v_4$};
\node (7) at (170bp,-11.0bp) [label=270:{{{\normalsize $V_{13}$}}}] {};

\end{scriptsize}
\end{tikzpicture}\vspace{-2cm}

\end{figure}

%%V14

\begin{figure}[h!]
\centering
\definecolor{xdxdff}{rgb}{0.49019607843137253,0.49019607843137253,1}
\begin{tikzpicture}[line cap=round,line join=round,>=triangle 45,x=0.7cm,y=0.7cm, scale=0.9]
%\begin{tikzpicture}[>=latex',join=bevel,scale=0.5]

\clip(-0.9,-4) rectangle (19,4.3);
\draw [line width=0.6pt] (1,0)-- (13,0);
\begin{scriptsize}
\draw [fill=xdxdff] (1,0) circle (0.6pt);
\draw[color=xdxdff] (1,-0.1) node[below] {$1$};
\draw [fill=xdxdff] (2,0) circle (0.6pt);
\draw[color=xdxdff] (2,-0.1) node[below] {$2$};
\draw [fill=xdxdff] (3,0) circle (0.6pt);
\draw[color=xdxdff] (3,-0.1) node[below] {$3$};
\draw [fill=xdxdff] (4,0) circle (0.6pt);
\draw[color=xdxdff] (4,-0.1) node[below] {$4$};
\draw [fill=xdxdff] (5,0) circle (0.6pt);
\draw[color=xdxdff] (5,-0.1) node[below] {$5$};
\draw [fill=xdxdff] (6,0) circle (0.6pt);
\draw[color=xdxdff] (6,-0.1) node[below] {$6$};
\draw [fill=xdxdff] (7,0) circle (0.6pt);
\draw[color=xdxdff] (7,-0.1) node[below] {$7$};
\draw [fill=xdxdff] (8,0) circle (0.6pt);
\draw[color=xdxdff] (8,-0.1) node[below] {$8$};
\draw [fill=xdxdff] (9,0) circle (0.6pt);
\draw[color=xdxdff] (9,-0.1) node[below] {$9$};
\draw [fill=xdxdff] (10,0) circle (0.6pt);
\draw[color=xdxdff] (10,-0.1) node[below] {$10$};
\draw [fill=xdxdff] (11,0) circle (0.6pt);
\draw[color=xdxdff] (11,-0.1) node[below] {$11$};	
\draw [fill=xdxdff] (12,0) circle (0.6pt);
\draw[color=xdxdff] (12,-0.1) node[below] {$12$};
\draw [fill=xdxdff] (13,0) circle (0.6pt);
\draw[color=xdxdff] (13,-0.1) node[below] {$13$};
%v1
\draw [line width=0.6pt] (1,4)-- (8,4);
\node  at (3bp,80bp) {\normalsize $v_1$};
%node[pos=-0.1, left] {\normalsize $v_1$}; Outro modo de label,  nao esquecer de tirar ; anterior.
%		\draw [fill=xdxdff] (1,1) circle (0.5pt);
%		\draw[color=xdxdff] (1,-0.1) node[below] {};
%		\draw [fill=xdxdff] (2,1) circle (0.5pt);
%		\draw[color=xdxdff] (2,-0.1) node[below] {};
%		%v2
\draw [line width=0.6pt] (1,3)-- (4,3);
\draw [line width=0.6pt] (9,3)-- (12,3);
\node  at (3bp,60bp) {\normalsize $v_2$};
%v3
\draw [line width=0.6pt] (1,2)-- (3,2);
\draw [line width=0.6pt] (9,2)-- (11,2);
\draw [fill=xdxdff] (5,2) circle (0.6pt);
\draw [fill=xdxdff] (13,2) circle (0.6pt);
\node  at (3bp,40bp) {\normalsize $v_3$};
%v4
\draw [line width=0.6pt] (1,1)-- (2,1);
\draw [line width=0.6pt] (9,1)-- (10,1);
\node  at (3bp,20bp) {\normalsize $v_4$};
\node (7) at (130bp,-11.0bp) [label=270:{{{\normalsize $V_{14}$}}}] {};

\end{scriptsize}
\end{tikzpicture}\vspace{-2cm}

\end{figure}

%%%V15
\begin{figure}[h!]
\centering
\definecolor{xdxdff}{rgb}{0.49019607843137253,0.49019607843137253,1}
\begin{tikzpicture}[line cap=round,line join=round,>=triangle 45,x=0.7cm,y=0.7cm, scale=0.9]
%\begin{tikzpicture}[>=latex',join=bevel,scale=0.5]

\clip(-0.9,-4) rectangle (19,4.3);
\draw [line width=0.6pt] (1,0)-- (17,0);
\begin{scriptsize}
\draw [fill=xdxdff] (1,0) circle (0.6pt);
\draw[color=xdxdff] (1,-0.1) node[below] {$1$};
\draw [fill=xdxdff] (2,0) circle (0.6pt);
\draw[color=xdxdff] (2,-0.1) node[below] {$2$};
\draw [fill=xdxdff] (3,0) circle (0.6pt);
\draw[color=xdxdff] (3,-0.1) node[below] {$3$};
\draw [fill=xdxdff] (4,0) circle (0.6pt);
\draw[color=xdxdff] (4,-0.1) node[below] {$4$};
\draw [fill=xdxdff] (5,0) circle (0.6pt);
\draw[color=xdxdff] (5,-0.1) node[below] {$5$};
\draw [fill=xdxdff] (6,0) circle (0.6pt);
\draw[color=xdxdff] (6,-0.1) node[below] {$6$};
\draw [fill=xdxdff] (7,0) circle (0.6pt);
\draw[color=xdxdff] (7,-0.1) node[below] {$7$};
\draw [fill=xdxdff] (8,0) circle (0.6pt);
\draw[color=xdxdff] (8,-0.1) node[below] {$8$};
\draw [fill=xdxdff] (9,0) circle (0.6pt);
\draw[color=xdxdff] (9,-0.1) node[below] {$9$};
\draw [fill=xdxdff] (10,0) circle (0.6pt);
\draw[color=xdxdff] (10,-0.1) node[below] {$10$};
\draw [fill=xdxdff] (11,0) circle (0.6pt);
\draw[color=xdxdff] (11,-0.1) node[below] {$11$};	
\draw [fill=xdxdff] (12,0) circle (0.6pt);
\draw[color=xdxdff] (12,-0.1) node[below] {$12$};
\draw [fill=xdxdff] (13,0) circle (0.6pt);
\draw[color=xdxdff] (13,-0.1) node[below] {$13$};
\draw [fill=xdxdff] (14,0) circle (0.6pt);
\draw[color=xdxdff] (14,-0.1) node[below] {$14$};
\draw [fill=xdxdff] (15,0) circle (0.6pt);
\draw[color=xdxdff] (15,-0.1) node[below] {$15$};
\draw [fill=xdxdff] (16,0) circle (0.6pt);
\draw[color=xdxdff] (16,-0.1) node[below] {$16$};
\draw [fill=xdxdff] (17,0) circle (0.6pt);
\draw[color=xdxdff] (17,-0.1) node[below] {$17$};

%v1
\draw [line width=0.6pt] (1,4)-- (12,4);
\node  at (3bp,80bp) {\normalsize $v_1$};
%node[pos=-0.1, left] {\normalsize $v_1$}; Outro modo de label,  nao esquecer de tirar ; anterior.
%		\draw [fill=xdxdff] (1,1) circle (0.5pt);
%		\draw[color=xdxdff] (1,-0.1) node[below] {};
%		\draw [fill=xdxdff] (2,1) circle (0.5pt);
%		\draw[color=xdxdff] (2,-0.1) node[below] {};
%		%v2
\draw [line width=0.6pt] (1,3)-- (4,3);
\draw [line width=0.6pt] (13,3)-- (16,3);
\node  at (3bp,60bp) {\normalsize $v_2$};
%v3
\draw [line width=0.6pt] (1,2)-- (3,2);
\draw [line width=0.6pt] (13,2)-- (15,2);
\draw [fill=xdxdff] (5,2) circle (0.6pt);
\draw [fill=xdxdff] (17,2) circle (0.6pt);
\node  at (3bp,40bp) {\normalsize $v_3$};
%v4
\draw [line width=0.6pt] (1,1)-- (2,1);
\draw [line width=0.6pt] (13,1)-- (14,1);
\node  at (3bp,20bp) {\normalsize $v_4$};
\node (7) at (170bp,-11.0bp) [label=270:{{{\normalsize $V_{15}$}}}] {};

\end{scriptsize}
\end{tikzpicture}\vspace{-2cm}

\end{figure}

%%%V16
\begin{figure}[!h]
\centering
\definecolor{xdxdff}{rgb}{0.49019607843137253,0.49019607843137253,1}
\begin{tikzpicture}[line cap=round,line join=round,>=triangle 45,x=0.7cm,y=0.7cm, scale=0.9]
%\begin{tikzpicture}[>=latex',join=bevel,scale=0.5]

\clip(-0.9,-4) rectangle (19,4.3);
\draw [line width=0.6pt] (1,0)-- (17,0);
\begin{scriptsize}
\draw [fill=xdxdff] (1,0) circle (0.6pt);
\draw[color=xdxdff] (1,-0.1) node[below] {$1$};
\draw [fill=xdxdff] (2,0) circle (0.6pt);
\draw[color=xdxdff] (2,-0.1) node[below] {$2$};
\draw [fill=xdxdff] (3,0) circle (0.6pt);
\draw[color=xdxdff] (3,-0.1) node[below] {$3$};
\draw [fill=xdxdff] (4,0) circle (0.6pt);
\draw[color=xdxdff] (4,-0.1) node[below] {$4$};
\draw [fill=xdxdff] (5,0) circle (0.6pt);
\draw[color=xdxdff] (5,-0.1) node[below] {$5$};
\draw [fill=xdxdff] (6,0) circle (0.6pt);
\draw[color=xdxdff] (6,-0.1) node[below] {$6$};
\draw [fill=xdxdff] (7,0) circle (0.6pt);
\draw[color=xdxdff] (7,-0.1) node[below] {$7$};
\draw [fill=xdxdff] (8,0) circle (0.6pt);
\draw[color=xdxdff] (8,-0.1) node[below] {$8$};
\draw [fill=xdxdff] (9,0) circle (0.6pt);
\draw[color=xdxdff] (9,-0.1) node[below] {$9$};
\draw [fill=xdxdff] (10,0) circle (0.6pt);
\draw[color=xdxdff] (10,-0.1) node[below] {$10$};
\draw [fill=xdxdff] (11,0) circle (0.6pt);
\draw[color=xdxdff] (11,-0.1) node[below] {$11$};	
\draw [fill=xdxdff] (12,0) circle (0.6pt);
\draw[color=xdxdff] (12,-0.1) node[below] {$12$};
\draw [fill=xdxdff] (13,0) circle (0.6pt);
\draw[color=xdxdff] (13,-0.1) node[below] {$13$};
\draw [fill=xdxdff] (14,0) circle (0.6pt);
\draw[color=xdxdff] (14,-0.1) node[below] {$14$};
\draw [fill=xdxdff] (15,0) circle (0.6pt);
\draw[color=xdxdff] (15,-0.1) node[below] {$15$};
\draw [fill=xdxdff] (16,0) circle (0.6pt);
\draw[color=xdxdff] (16,-0.1) node[below] {$16$};
\draw [fill=xdxdff] (17,0) circle (0.6pt);
\draw[color=xdxdff] (17,-0.1) node[below] {$17$};
%v1
\draw [line width=0.6pt] (1,4)-- (8,4);
\node  at (3bp,80bp) {\normalsize $v_1$};
%node[pos=-0.1, left] {\normalsize $v_1$}; Outro modo de label,  nao esquecer de tirar ; anterior.
%		\draw [fill=xdxdff] (1,1) circle (0.5pt);
%		\draw[color=xdxdff] (1,-0.1) node[below] {};
%		\draw [fill=xdxdff] (2,1) circle (0.5pt);
%		\draw[color=xdxdff] (2,-0.1) node[below] {};
%		%v2
\draw [line width=0.6pt] (1,3)-- (2,3);
\draw [line width=0.6pt] (9,3)-- (14,3);
\node  at (3bp,60bp) {\normalsize $v_2$};
%v3
\draw [line width=0.6pt] (9,2)-- (13,2);
\draw [fill=xdxdff] (1,2) circle (0.6pt);
\draw [fill=xdxdff] (3,2) circle (0.6pt);
\draw [fill=xdxdff] (15,2) circle (0.6pt);
\node  at (3bp,40bp) {\normalsize $v_3$};
%v4
\draw [line width=0.6pt] (4,1)-- (5,1);
\draw [line width=0.6pt] (16,1)-- (17,1);
\node  at (3bp,20bp) {\normalsize $v_4$};
\node (7) at (170bp,-11.0bp) [label=270:{{{\normalsize $V_{16}$}}}] {};

\end{scriptsize}
\end{tikzpicture}\vspace{-2cm}
%\caption{Representation $V_{3}$}
%\label{fig.V3}

\end{figure}

%%%OUTRA FORMA, NAO DEU CERTO
%\begin{figure}[!h]
%	\centering
%	\begin{tabular}{l}
%		\includegraphics[width=1.0\linewidth]{teste5.png}\\%V1-V4
%		\includegraphics[width=1.0\linewidth]{teste5.png}\\%V5-V8
%			\end{tabular}
%\end{figure}
%
%\begin{figure}[!h]
%	\centering
%	\begin{tabular}{l}
%		\includegraphics[width=1.0\linewidth]{teste5.png}\\%V9-V12
%		\includegraphics[width=1.0\linewidth]{teste5.png}\\%V13-V16
%	\end{tabular}
%\end{figure}
%

%\clearpage

%\newpage
\FloatBarrier       % garante que nenhuma figura/tabela suba
\clearpage          % força quebra de página
\begingroup         % grupo para ajustes locais (opcional)
\small              % deixa as referências um pouco menores
\renewcommand{\refname}

%%%%%%%%%%%%%%%%%%%%%%%%%%%%%%%NEW  NEW

\end{document}